\documentclass[10pt]{amsart}

\usepackage{amsmath}
\usepackage{amsfonts}
\usepackage{amscd}
\usepackage{amsthm}
\usepackage{amssymb}
\usepackage{mathrsfs}
\usepackage{enumerate}
\usepackage{bbm}
\usepackage{bm}
\usepackage{multirow} 
\usepackage{makecell}
\usepackage{booktabs, here}
\usepackage{indentfirst}
\usepackage{epsfig}
\usepackage{latexsym}
\usepackage{float}
\usepackage{epstopdf}
\usepackage{kotex}
\usepackage{enumitem}

\usepackage{caption}
\usepackage{subcaption}
\usepackage{extpfeil}
\usepackage{graphicx}
\usepackage{colortbl}
\usepackage{tikz}

\usepackage{hyperref}

\numberwithin{equation}{section}

\title[Asymptotic phase-locking in the inertial Kuramoto model]{Quantitative relaxation dynamics from generic initial configurations in the inertial Kuramoto model}

\author[Cho]{Hangjun Cho}
\address[Hangjun Cho]{\newline Department of Mechanical Engineering \newline University of Washington, Seattle, WA 98195, United States }
\email{hangjun@uw.edu}

\author[Dong]{Jiu-Gang Dong}
\address[Jiu-Gang Dong]{\newline School of Mathematical Sciences \newline Dalian University of Technology, Dalian 116024, People's Republic of China}
\email{jgdong@dlut.edu.cn}

\author[Ha]{Seung-Yeal Ha}
\address[Seung-Yeal Ha]{\newline Department of Mathematical Sciences and Research Institute of Mathematics \newline Seoul National University, Seoul 08826, Republic of Korea}
\email{syha@snu.ac.kr}

\author[Ryoo]{Seung-Yeon Ryoo}
\address[Seung-Yeon Ryoo]{\newline The Division of Physics, Mathematics and Astronomy \newline California Institute of Technology, Pasadena California 91125, United States}
\email{sryoo@caltech.edu}

\newtheorem{theorem}{Theorem}[section]
\newtheorem{lemma}{Lemma}[section]
\newtheorem{corollary}{Corollary}[section]
\newtheorem{proposition}{Proposition}[section]
\newtheorem{remark}{Remark}[section]
\newtheorem{definition}{Definition}[section]
\newtheorem{example}{Example}[section]
\newtheorem{question}{Question}[section]
\newtheorem{conjecture}{Conjecture}[section]

\newcommand{\bbr}{\mathbb{R}}

\allowdisplaybreaks
\begin{document}

\date{\today}


\thanks{\textbf{Acknowledgment.}
We thank S. Olmi for a helpful reference recommendation.
The work of H. Cho was supported by the National Research Foundation(NRF) of Korea grant funded by the Korea government(MSIT) (RS-2023-00253171), the work of J.-G. Dong was supported by the National Key R\&D Program of China (No. 2023YFA1009200) and the National Natural Science Foundation of China through grant 12171069, the work of S.-Y. Ha was supported by NRF grant (NRF-RS2025-00514472), and the work of S.-Y. Ryoo was partially supported by the fellowship of Korea Foundation for Advanced Studies.}

\begin{abstract}
We study the relaxation dynamics of the inertial Kuramoto model toward a phase-locked state from a generic initial phase configuration. For this, we propose a sufficient framework in terms of initial data and system parameters for asymptotic phase-locking. It can be roughly stated as set of conditions such as a positive initial order parameter, a coupling strength sufficiently larger than initial frequency diameter and intrinsic frequency diameter, but less than the inverse of inertia. Under the proposed framework, generic initial configuration undergoes three dynamic stages (initial layer, condensation and relaxation stages)  before it reaches a phase-locked state asymptotically. The first stage is the initial layer stage in analogy with fluid mechanics, during which the effect of the initial natural frequency distribution is dominant, compared to that of the sinusoidal coupling between oscillators. The second stage is the condensation stage, during which the order parameter increases, and at the end of which a majority cluster is contained in a sufficiently small arc. Finally, the third stage is the persistence and relaxation stage, during which the majority cluster remains stable (persistence) and the total configuration relaxes toward a phase-locked state asymptotically (relaxation). The intricate proof involves with several key tools such as the quasi-monotonicity of the order parameter (for the condensation stage), a nonlinear Gr\"onwall inequality on the diameter of the majority cluster (for the persistence stage), and a variant of the classical {\L}ojasiewicz gradient theorem (for the relaxation stage). 
\end{abstract}

\maketitle

\tableofcontents

%
%
%
%
\section{Introduction} \label{sec:1}
\setcounter{equation}{0}
\emph{Synchronization} refers to an adjustment of rhythms in interacting oscillatory systems, and it can be regarded as the formation of consensus in phases and frequencies among oscillators. The gradual appearance of synchronization from a desynchronized state is referred to as an \emph{emergent behavior}, and this is often observed in natural systems. To name a few, swarming of fish, flocking of birds, or aggregation of bacteria \cite{A-B-F-H-P-P-J,B-B,P-R-K,T-B,V-Z} etc. Despite its ubiquitous appearance in nature, synchronization was mathematically formulated only half a century ago by two pioneers Arthur Winfree \cite{Wi2, Wi1} and Yoshiki Kuramoto \cite{Ku2,Ku}. They proposed mathematically tractable phase models that describe the dynamics of weakly interacting limit-cycle oscillators, and identified fundamental synchronous behavior.

The novel feature of Winfree and Kuramoto models is that they both exhibit phase transition phenomena from disordered states (incoherent states) to partially ordered states (partially phase-locked states) and then to fully ordered states (completely phase-locked states), as the \emph{coupling strength}, a variable representing the amount of interaction between the individuals, exceeds certain critical threshold \cite{A-S,Cr, K-B}. Due to this phase transition like feature, Winfree and Kuramoto's mathematical approach to synchronous phenomena, and more generally the theory of weakly coupled oscillators, has received lots of attention from control theory, neuroscience, and statistical physics communities \cite{A-B, B-T11,B-T13,D-B14,Ermentrout19,H-K-P-Z,Hoppensteadt12,Rodrigues16,Str}.

In this article, we are mainly interested in the \emph{inertial Kuramoto model} which corresponds to the second-order correction of the Kuramoto model introduced by Arthur Bergen and David Hill \cite{B-H} to model electric networks with generators, and by Bard Ermentrout \cite{Er} to model synchronous flashing of the firefly species Pteroptyx malaccae.
Due to its second-order nature, the inertial Kuramoto model possesses several novel features absent in the Kuramoto model such as first-order hysteretic transition \cite{T-L-O1,T-L-O2, Olmi14,Barre16}.

To set up the stage, we begin with a brief description of the inertial Kuramoto model. Fix a positive integer $N$, the number of particles, and for $i\in [N]\coloneqq \{1,\cdots,N\}$, let $\theta_i = \theta_i(t)\in \mathbb{R}$ and $\omega_i = {\dot \theta}_i$ be the phase and (instantaneous) frequency of the $i$-th Kuramoto oscillator, given as a real-valued function of nonnegative time $t\ge 0$. The dynamics of the phase ensemble $\{\theta_i\}_{i=1}^N$ under the inertial Kuramoto model is governed by the following Cauchy problem:
\begin{equation}
\begin{cases} \label{A-1}
\displaystyle m \ddot\theta_i + \dot\theta_i = \nu_i +\frac{\kappa}{N}\sum_{j=1}^N \sin(\theta_j - \theta_i),\quad t > 0,\\
\displaystyle (\theta_i, {\dot \theta}_i) \Big|_{t = 0+} = (\theta_i^0, \omega_i^0), \quad i\in [N].
\end{cases}
\end{equation}
Here, $m, \kappa$ and $\nu_i$ are nonnegative constants representing \emph{(positive uniform\footnote{``Uniform'' refers to the fact that $m$ does not depend on $i$. The ``multi-rate Kuramoto model'' of \cite{D-B11} considers the case when $m$ is different for each oscillator $\theta_i$. We will not consider this case in this article.}) inertia}, \emph{coupling strength} and natural (intrinsic) frequency of the $i$-th oscillators, respectively. 

By the standard Cauchy-Lipschitz theory, the Cauchy problem \eqref{A-1} admits a global unique solution, which must be real analytic in terms of time and all other parameters by the Cauchy–Kovalevskaya theorem. Thus, in this paper, neither uniqueness, global existence, nor smoothness of solutions to \eqref{A-1} will be an issue. In what follows, we discuss main results of this paper. \newline

\subsection{Main results} \label{sec:1.1}
The first set of results is concerned with the emergence of complete and partial phase-lockings. Roughly, it says that if initial configuration and system parameters satisfy 
\begin{align}
\begin{aligned} \label{eq:R^0}
& R(t) \coloneqq \left|\frac{1}{N} \sum_{j=1}^N e^{\mathrm{i}\theta_j(t)} \right|, \quad R(0) = R^0 > 0, \\
&  \max \Bigg \{ \frac{\displaystyle\max_{1\le i,j\le N}|\nu_i-\nu_j|}{\kappa},~~\frac{\displaystyle\max_{1\le i,j\le N}|\omega_i^0-\omega_j^0|}{\kappa},~~m\kappa \Bigg \} \ll |R^0|^2,
\end{aligned}
\end{align}
then the solution to \eqref{A-1} converges to a single traveling solution (see Theorem \ref{simplemainthm} below). Here, the symbol $\ll$ simply means the left-hand side is very smaller than the right-hand side in a non-rigorous manner. It will only be used when discussing heuristics; we do not assign a precise meaning. 
\begin{theorem}\label{simplemainthm}
Suppose the initial data and system parameters satisfy\footnote{A simple dimensional analysis shows that $R^0$ is dimensionless, while $\mathcal{{V}}$, $\Omega$, $\kappa$, and $1/m$ have the dimension of the inverse of time,  hence the forms of the left-hand sides of \eqref{eq:vanillacondition}. For a more rigorous dimensional analysis, see the time dilatation symmetries discussed in subsection \ref{subsec:sym}.}
    \begin{equation}\label{eq:vanillacondition}
\frac{\displaystyle\max_{1\le i,j\le N}|\nu_i-\nu_j|}{\kappa}\le x|R^0|^2,\quad m\kappa \le y|R^0|^2,\quad \mathrm{and}\quad \frac{\displaystyle\max_{1\le i,j\le N}|\omega_i^0-\omega_j^0|}{\kappa}\le z |R^0|^2,
\end{equation}
where $x,y,z$ are positive real numbers satisfying
\begin{align}\label{eq:xyz-cond}
\begin{aligned}
\inf_{\eta>0}&\Bigg((1-e^{-\eta})y\left(\frac{1}{2}(z+\eta x)+(1-e^{-\eta})^2y(\frac 34 z+\eta x+2\eta)\right)\\
&+\sqrt{3.068\left(y(x+2)+\max\{1,\eta\}e^{-\max\{1,\eta\}}yz+\frac x2+\frac{e^{-\eta}}{1-e^{-\eta}}\frac z2\right)}< 1.
\end{aligned}
\end{align}
Then, the following assertions hold. \newline
\begin{enumerate}
    \item  (Asymptotic phase-locking):~The global solution $(\Theta, \Omega := {\dot \Theta})$ (with $\Theta=(\theta_1,\ldots,\theta_N)$) to the Cauchy problem \eqref{A-1} exhibits ``asymptotic phase-locking'':
\[
\exists \lim_{t\to\infty} (\theta_i(t)-\nu_c t),\quad \mathrm{and}\quad \lim_{t\to\infty} \dot{\theta}_i(t)=\nu_c\quad \mathrm{for~all~} i \in [N],
\]
where $\nu_c=\frac 1N\sum_{i=1}^N\nu_i$ is the average of natural frequencies. 
\vspace{0.2cm}
\item  (Asymptotic partial phase-locking):~There exists a subset $\mathcal{A}\subset [N]$ with $|\mathcal{A}|>\frac N2$ and integers $k_i$ for each $i\in \mathcal{A}$ such that
\[
\limsup_{t\to\infty}\max_{i,j\in \mathcal{A}} \Big |(\theta_i(t)-2k_i\pi)-(\theta_j(t)-2k_j\pi) \Big |<\frac \pi 2
\]
and a constant $c = c(x, y, z, R^0)$ such that for any $i,j\in \mathcal{A}$ with $\nu_i\ge \nu_j$,
\[
\frac{\nu_i-\nu_j}{\kappa}\le \lim_{t\to\infty}(\theta_i(t)-\theta_j(t))\le c\frac{\nu_i-\nu_j}{\kappa}.
\]
\end{enumerate}
\end{theorem}
\begin{proof}
We will prove Theorem \ref{simplemainthm} assuming Theorem \ref{T3.1} in Appendix \ref{app:mainthm}.
\end{proof}
\begin{remark}
Condition \eqref{eq:xyz-cond} is satisfied, for example, by
\[
(x,y,z,\eta)=(0.5,0.015,0.12,1), (0.3, 0.05, 0.76, 3).
\]
These constants are artifacts of our proof and we do not claim that they are optimal, and Theorem \ref{simplemainthm} will follow from a more general framework, namely Theorem \ref{T3.1}.
\end{remark}
In literature \cite{Choi15,C-H-M,C-H-Y1,C-L,D-B12,D-B11,H-J-K,W-Q}, the complete synchronization problem (or asymptotic phase-locking) has been investigated for the inertial Kuramoto model \eqref{A-1}  for a restricted initial configuration. Numerical simulations suggest that the inertial Kuramoto model exhibits asymptotic phase-locking (as defined in the statement of Theorem \ref{simplemainthm} (1)) for generic initial data in the large coupling regime. Thus, we are led to ask the following question.
\begin{question}\label{ques:critical-kappa}
\emph{(Existence of critical coupling strength)}~ For a fixed natural frequency vector ${\mathcal V}=(\nu_1,\ldots,\nu_N)$ and Lebesgue a.e.~initial data $(\Theta^0,\Omega^0)$ (with $\Theta^0 = (\theta_1^0,\ldots,\theta_N^0)$ and $\Omega^0=(\omega_1^0,\ldots,\omega_N^0)$), does there exist a critical value $\kappa_c = \kappa_c({\mathcal V}, \Theta^0, \Omega^0)$ such that $\kappa>\kappa_c$ implies asymptotic phase-locking in system \eqref{A-1}, regardless of the inertia $m>0$?
\end{question}
From this point of view, the contribution of Theorem \ref{simplemainthm} is that it guarantees a range $\kappa\in [\kappa_1,\kappa_2]$, where $\kappa_1$ is a function of ${\mathcal V},~\Theta^0,~\Omega^0$ and $\kappa_2= {\mathcal O}(|R^0|^2/m)$, at which asymptotic phase-locking occurs. To our knowledge, Theorem \ref{simplemainthm} is the first in the literature to provide a sufficient framework for asymptotic phase-locking for the inertial Kuramoto model \eqref{A-1} in the generality of generic initial data.

As for the physical significance of the second-order model \eqref{A-1}, Ermentrout \cite{Er} considered \eqref{A-1} as a model for frequency modulation, essentially writing the ODE of \eqref{A-1} as
\[
\ddot{\theta}_i=\frac 1m (\nu_i-\dot{\theta}_i)+\frac{\kappa}{mN}\sum_{j=1}^N \sin(\theta_j-\theta_i).
\]
Note that the right-hand side can be regarded as competition between two forcing terms. More precisely, the term $\frac 1m (\nu_i-\dot{\theta}_i)$ denotes the tendency of the frequency to return to the natural frequency $\nu_i$, whereas the term $\frac{\kappa}{mN}\sin(\theta_j-\theta_i)$ represents the enforcing consensus via the phase response of the $i$-th oscillator to the $j$-th oscillator.

Bergen and Hill \cite{B-H} considered the Cauchy problem \eqref{A-1} as the swing dynamics for a power network consisting of $N$ electrical generators \cite{C,Kundur,S-P98}: the inertial term $m\ddot{\theta}_i$ corresponds to the generator inertia; the dissipative terms $\dot{\theta}_i$ correspond to the loads or the mechanical damping; the intrinsic frequencies $\nu_i$ correspond to power injections; the nonlinear interaction terms $\frac \kappa N \sin (\theta_j-\theta_i)$ correspond to the power flows along transmission lines; and synchronization is interpreted as robustness or transient stability \cite{D13,D-B10,D-B12,D-B14,F-N-P,G-Z-L-W,S-U-S-P}. From this point of view, the inertial Kuramoto model has found an application as an elementary model for ``smart grids'' \cite{D-C-B}. The case $m=0$, i.e., the model \eqref{A-2}, signifies that the nodes have loads and no generation. One could also consider systems with pure generation and zero load or damping \cite{H-J-Z} as well.

There has been recent interest in low-inertia power grids \cite{M-D-H-H-V}: conventional fossil fuel-based power plants, which often involve steam and hydroelectric turbines, tend to have high inertia, whereas renewable energy sources such as solar and wind energy tend to have low inertia. How does one understand the stability and synchronization properties of such low-inertia power grids? In this light, the present work develops a theory of synchronization applicable in the low-inertia setting.\footnote{The literature suggests conflicting results, with some suggesting that low inertia contributes to synchronization \cite{D-B12}, while others suggest that low inertia destabilizes system \eqref{A-1} \cite{A-B00,C-C10}.} 

As by-products of the arguments employed in the proof of Theorem \ref{simplemainthm}, we can obtain improved results for specific situations. First, we obtain a condition for partial phase-locking (see Definition \ref{def:partial}), which states that if a majority of the oscillators is contained in a sufficiently small arc, it becomes self-sustaining and limits the dynamics of other oscillators. See Theorem \ref{L4.4}. A cruder version of this idea has already appeared in \cite{H-J-K}; their results are subsumed under Theorem \ref{L4.4}. Second, we provide a framework more general than that of Theorem \ref{simplemainthm}, namely Theorem \ref{T3.1}.  Third, we deal with a complete phase-locking for a three-oscillator system for {\it all initial data}:
\begin{theorem}\label{thm:n=32nd}
Suppose that system parameters satisfy
\[
N=3,\quad m\max_{1\le i,j\le N}|\nu_i-\nu_j|+2m\kappa +\frac{\displaystyle\max_{1\le i,j\le N}|\nu_i-\nu_j|}{2\kappa}
<\frac{1}{8}\sqrt{\frac{1}{6}(69 - 11\sqrt{33})} \approx 0.123003.
\]
Then, the solution to \eqref{A-1} exhibits asymptotic phase-locking.
\end{theorem}
\begin{proof}
We provide the proof in Appendix \ref{app:collision}, after proving the equivalence of asymptotic phase-locking and finiteness of collisions in the small inertia regime $m\kappa \le \frac{1}{4}$.
\end{proof}

Besides Theorem \ref{simplemainthm} and Theorem \ref{thm:n=32nd}, we have a plethora of new results in this paper, namely, Theorem \ref{T3.1},  Theorem \ref{L4.4}, Theorem \ref{thm:finite-collision}, Proposition \ref{prop:sym}, and Corollary \ref{cor:partialphaselocking-initial}.
\bigskip


\subsection{Roadmap}  \label{sec:1.2} The rest of this paper is organized as follows.  Below, we summarize some notations and conventions to be used throughout the paper.  In Section \ref{sec:galilean}, we first recall the (first-order) Kuramoto model and related results on asymptotic phase-locking and then demonstrate Galilean invariance and the corresponding momentum conservation of \eqref{A-1} and \eqref{A-2}. We use this to provide natural definitions of complete and partial phase-lockings. This sheds light on why we obtain single traveling solutions with group frequency $\nu_c$ in the conclusions of Theorem \ref{simplemainthm}. We then pose our main questions on the emergence of phase-locking. We also explain our basic rationale behind analysis of \eqref{A-1}, which is to approximate it to the Kuramoto model \eqref{A-2}. This process requires the smallness of $\displaystyle\max_{1\le i,j\le N}|\omega_i^0-\omega_j^0|/\kappa$ and $m\kappa$.

In Section \ref{sec:2-divided}, we explore three synchronization mechanisms of models \eqref{A-1} and \eqref{A-2}, namely stability of majority clusters (subsection \ref{subsec:1.4}), quasi-monotonicity of the order parameter (subsection \ref{subsec:1.5}), and inertial gradient flow formulation of \eqref{A-1} and resulting application of the {\L}ojasiewicz gradient theorem (subsection \ref{subsec:1.6}). These mechanisms require, in addition, the smallness of $\displaystyle\max_{1\le i,j\le N}|\nu_i-\nu_j|/\kappa$. We conjecture that this third mechanism can be used to fully characterize the asymptotic behaviors of models \eqref{A-1} and \eqref{A-2}.

In Section \ref{sec:3}, we study the properties of the majority cluster to establish a sufficient criteria for partial phase-locking. 

In Section \ref{sec:5}, we present detailed proofs for presented main results. Finally, Section \ref{sec:6} is devoted to a brief summary of  the main results and discussions for future work. In appendix sections, we provide the various technical details alluded throughout the paper.

\subsection{Notations and conventions} We fix some notations and conventions for the remainder of the paper. Let $(\theta_1,\ldots,\theta_N)$ be a solution to \eqref{A-1} or \eqref{A-2}. We denote $\omega_i(t)\coloneqq \dot\theta_i(t)$. We use capital Greek letters to denote the collection of the corresponding lower Greek letters:
\begin{equation*}
\Theta\coloneqq  (\theta_1, \ldots, \theta_N), \quad \Omega\coloneqq (\omega_1, \ldots, \omega_N), \quad \mathcal{V}\coloneqq (\nu_1, \ldots, \nu_N).
\end{equation*}
We will use $\omega_i$ and $\dot{\theta}_i$ interchangeably and $\Omega$ and $\dot{\Theta}$ interchangeably throughout this article.
We denote ${\bf 1}_{[N]} = (1, \cdots, 1) \in {\mathbb R}^N$. For vectors $X=(x_i)^N_{i=1}\in \mathbb{R}^N$ and $Y=(y_i)^N_{i=1}\in \mathbb{R}^N$, we denote their exterior product $X\wedge Y\in \bigwedge^2 \mathbb{R}^N$. For the standard basis $e_1,\cdots,e_N$ of $\mathbb{R}^N$, we consider $e_i\wedge e_j$, $i,j\in [N]$, $i<j$ to be the standard basis of $\bigwedge^2\mathbb{R}^N$. Thus, in these coordinates, we may write
\[
X\wedge Y= (x_iy_j-x_jy_i)_{i,j\in [N],i<j}.
\]
As usual, we use $\|\cdot\|_p$ to denote the $\ell_p$-norm in $\bbr^{N}$:
\begin{equation*}
\|\Theta\|_{p} \coloneqq \left( \sum_{i=1}^{N} |\theta_i|^p \right)^{\frac{1}{p}}, ~ 1 \leq p < \infty, \quad  \|\Theta\|_{\infty} \coloneqq \max_{1\leq i \leq N} |\theta_i|.
\end{equation*}
We use $\mathcal{D}$ to denote the diameter: for a vector $X= ( x_i )^N_{i=1}\in \mathbb{R}^N$,
\[
\mathcal{D}(X)\coloneqq \max_{1\le i,j\le N}|x_i-x_j|.
\]
For example, for the above configurations $\Theta, \Omega$ and ${\mathcal V}$, we denote
\begin{align*}
\begin{aligned}
{\mathcal D}(\Theta) \coloneqq \max_{1\leq i,j\leq N}|\theta_i -\theta_j|,\quad {\mathcal D}(\Omega) \coloneqq \max_{1\leq i,j\leq N}|\omega_i -\omega_j|,\quad {\mathcal D}(\mathcal{V}) \coloneqq \max_{1\leq i,j\leq N}|\nu_i - \nu_j|.
\end{aligned}
\end{align*}
For $X\in \mathbb{R}^N$, we have
\[
\mathcal{D}(X)=\|X\wedge {\bf 1}_{[N]}\|_\infty,
\]
where we endow $\bigwedge^2\mathbb{R}^N$ the $\ell_\infty$ norm with respect to its standard basis. We observe the triangle inequality
\begin{equation*}
|\mathcal{D}(A)-\mathcal{D}(B)|\le \mathcal{D}(A-B),\quad A,B\in \mathbb{R}^N.
\end{equation*}
We observe the inequality
\[
\frac 12\mathcal{D}(X)\le \|X\|_\infty,\quad X\in \mathbb{R}^N,
\]
which follows from the identity
\begin{equation*}
    \|X\|_\infty-\frac 12\mathcal{D}(X)=\frac 12\left|\max_{i\in [N]} x_i+\min_{i\in [N]} x_i\right|,\quad X=(x_1,\cdots,x_N)\in \mathbb{R}^N.
\end{equation*}
We also denote the variances as follows:
\begin{align*}
\begin{aligned}
\operatorname{Var}(\mathcal{V}) \coloneqq \frac 1N\sum_{i=1}^N |\nu_i-\nu_c|^2,\quad \operatorname{Var}(\Omega^0) \coloneqq \frac 1N\sum_{i=1}^N |\omega_i^0-\omega_c|^2.
\end{aligned}
\end{align*}
It is well known that
\begin{align*}
\operatorname{Var}(X)\le \frac{\mathcal{D}(X)^2}{4},\quad X\in \mathbb{R}^N.
\end{align*}
For $\mathcal{A}\subset [N]$ and $X=(x_i)_{i=1}^N\in \mathbb{R}^N$, we define the restricted vector
\[
X_\mathcal{A}\coloneqq (x_i)_{i\in\mathcal{A}}\in \mathbb{R}^{|\mathcal{A}|}.
\]
The concepts of $\mathcal{D}(X_\mathcal{A})$ and $\|X_\mathcal{A}\|_p$ are defined in the same manner:
\[
\mathcal{D}(\Theta_\mathcal{A})\coloneqq \max_{i,j\in \mathcal{A}}|\theta_i-\theta_j|,\quad {\mathcal D}(\Omega_\mathcal{A}) \coloneqq \max_{i,j\in \mathcal{A}}|\omega_i -\omega_j|,\quad {\mathcal D}(\mathcal{V}_\mathcal{A}) \coloneqq \max_{i,j\in \mathcal{A}}|\nu_i - \nu_j|.
\]
and
\begin{equation*}
\|\Theta_\mathcal{A}\|_{p} \coloneqq \Big( \sum_{i\in\mathcal{A}} |\theta_i|^p \Big)^{\frac{1}{p}}, ~ 1 \leq p < \infty, \quad  \|\Theta_\mathcal{A}\|_{\infty} \coloneqq \max_{i\in\mathcal{A}} |\theta_i|.
\end{equation*}
Throughout the paper, we call $m$, $\kappa$, and $\mathcal{V}$ system parameters, $\Theta^0$ and $\Omega^0$ initial data, and all other external parameters free parameters. We sometimes say `{\it parameters}' to refer to any of these variables.

Although \eqref{A-1} describes a dynamical system on $({\mathbb S}^1)^N \times {\mathbb R}^N$, we will interpret it as a dynamical system on $\mathbb R^N \times {\mathbb R}^N$, to avoid the existence issue of the scalar potential $P$ defined in \eqref{B-5-1}. However, we will still use the geometric concept of an arc, i.e., a connected component of ${\mathbb S}^1$, as this will help to visualize the configuration geometrically. Anytime this terminology is used, various statements must be understood modulo $2\pi$. For example, whenever we say some oscillators $\Theta_\mathcal{A}$ are contained in an arc of length $\ell$, we will express this with $\mathcal{D}(\Theta_\mathcal{A})\le \ell$; when we say this it will be clear that we can make harmless $2\pi$-translations of $\Theta_\mathcal{A}$ that do not affect the conclusions of our theorems nor the logical structure of their proofs.

\section{Preliminaries}\label{sec:galilean}
\setcounter{equation}{0}
In this section, we study several preparatory facts to be used in later sections. First, we briefly review the related result for the Kuramoto model on asymptotic phase-locking, and the translation invariance property of the inertial Kuramoto model, phase-locked states as relative equilibria and symmetries related to the inertial Kuramoto model. 

\subsection{The Kuramoto model}
\label{subsec:first-order}
The (globally coupled\footnote{This means that the weight for the terms $\sin(\theta_j-\theta_i)$ are uniform. The case of weighted connectivity terms $a_{ij}\sin(\theta_j-\theta_i)$ is considered for example in \cite{Jadbabaie04}. Again, we will not work in this generality.}) first-order Kuramoto model, originally proposed by Kuramoto in \cite{Ku2,Ku}, is the model formally obtained\footnote{Notice that when we take the limit $m\to 0+$ and formally pass from system \eqref{A-1} to system \eqref{A-2}, we `forget' the initial velocity data $\{\omega_i^0\}_{i=1}^N$. This is because the solution to \eqref{A-1} converges to the solution to \eqref{A-2} as $m\to 0+$ in the $C[0,\infty)$ and $C^\infty(0,\infty)$ topologies, but not necessarily in the $C^1[0,\infty)$ topology.} from \eqref{A-1} by taking zero inertia $m=0$ and unit damping coefficient:\footnote{This refers to the coefficient of $\dot\theta_i$. Generally, we could also consider an arbitrary damping coefficient $\gamma>0$ and replace $\dot{\theta}_i$ by $\gamma \dot{\theta}_i$, but of course then we may divide the equations \eqref{A-1} and \eqref{A-2} by $\gamma$. From a different point of view, we may say that we are considering systems with nonzero damping coefficient.}
\begin{align}\label{A-2}
\begin{cases}
\displaystyle \dot\theta_i = \nu_i + \frac{\kappa}{N}\sum_{j=1}^N \sin(\theta_j - \theta_i),\quad t > 0,\\
\displaystyle \theta_i\Big|_{t = 0+} = \theta_i^0,\quad i\in [N].
\end{cases}
\end{align}
Again, by the standard Cauchy-Lipschitz theory and Cauchy-Kovalevskaya theorem, the model \eqref{A-2} admits global unique analytic solutions.

There have been lots of studies \cite{Aeyels04,Bronski12,Chopra09,C-H-J-K,D-B11,D-B14,D-H-K,D-X,Mirollo05,H-K-R,H-R,Jadbabaie04,Van93} on the asymptotic dynamics of the models \eqref{A-2} and their variants. Indeed, the aforementioned references provide several sufficient frameworks for the complete synchronization problem in which all the relative frequencies tend to zero asymptotically in a large coupling regime, for initial data restricted in a half-circle when the intrinsic velocities $\nu_i$ are distinct, and for a generic initial configuration when the intrinsic velocities $\nu_i$ are identical. Similarly to \eqref{A-1}, the complete synchronization problem seems to be numerically true for all generic data in a large coupling regime, even for nonidentical intrinsic velocities $\nu_i$ \cite{D-B11,H-R}. However, to verify this simple fact is difficult.

In recent works \cite{H-K-R,H-R}, the authors provided sufficient conditions for generic initial data and coupling strength leading to the complete synchronization via the gradient flow formulation of \eqref{A-2} and technical estimates on the phase diameter and order parameter for the Kuramoto model \eqref{A-2}. This affirmatively answered the variant of Question \ref{ques:critical-kappa} for the first-order Kuramoto model \eqref{A-2}. The statement is as follows.

\begin{theorem}[{\cite[Theorem 1.1]{H-K-R}},{\cite[Theorem 3.2]{H-R}}]\label{thm:1stKu-vanilla}
If the initial data and system parameters satisfy 
\begin{equation} \label{B-15-1} 
R^0 > 0, \quad \kappa >1.6\frac{\max_{1\le i,j\le N}|\nu_i-\nu_j|}{|R^0|^2},
\end{equation}
then the global solution $\Theta = \Theta(t)$ to \eqref{A-2} exhibits asymptotic phase-locking: for all $i\in [N]$,
\[
\exists \lim_{t\to\infty} (\theta_i(t)-\nu_c t),\quad \mathrm{and}\quad \lim_{t\to\infty} \dot{\theta}_i(t)=\nu_c.
\]
\end{theorem}
\begin{remark}
Theorem \ref{simplemainthm} was inspired by this counterpart theorem for the model \eqref{A-2}.
\end{remark}

For initial data $\left(\Theta^0,\Omega^0\right)$, intrinsic frequencies  $\mathcal{V}$, a coupling strength $\kappa$ and  $m>0$, we temporarily denote by $\Theta(m,t)$, $t\ge 0$, the solution to the Cauchy problem \eqref{A-1}, and denote by $\Theta(0,t)$, $t\ge 0$, the solution to the Cauchy problem \eqref{A-2}.
A quantitative version of Tikhonov's theorem implies
\[ \omega_i(m,t)-\omega_i(0,t)= {\mathcal O}(e^{-t/m})+ {\mathcal O}(m). \]
Thus, we will consider an ``initial time layer''\footnote{We adopted this jargon from fluid mechanics.} of the form $[0,\eta m]$. The synchronization analysis of this paper, such as partial phase-locking (Theorem \ref{L4.4}) or the quasi-monotonicity estimate of the order parameter (Lemma \ref{L2.4}), will be performed after this initial time layer $[0,\eta m]$. 
The classical Tikhonov theorem does not provide quantitative bounds, so we defer such analysis to a forthcoming paper. As a corollary, we are also able to show convergence in $C^\infty(0,\infty)$.

\subsection{Galilean Invariance}
One defining feature of \eqref{A-1} is its sinusoidal coupling $\sin(\theta_j-\theta_i)$. This nonlinearity severely limits the applicability of existing classical tools, particularly those catering to linear equations, for analyzing \eqref{A-1} and \eqref{A-2}. 
\begin{lemma}\label{L2.1}
Let $(\Theta, \Omega)$ be a global solution to \eqref{A-1}. The phase and frequency averages satisfy
\begin{align}
\begin{aligned} \label{eq:galconserve}
\theta_c(t) &= m\omega_c^0(1-e^{-t/m}) + \nu_c \left(t-m+me^{-t/m}\right) + \theta_c^0, \\
\omega_c(t) &= \omega_c^0 e^{-t/m} + \nu_c\left(1-e^{-t/m}\right), \quad t \geq 0,
\end{aligned}
\end{align}
where the subscript $c$ denotes the average over particles:
\begin{equation} \label{B-2}
\theta_c^0 \coloneqq \frac{1}{N} \sum_{i=1}^{N}  \theta_i^0,\quad
\theta_c \coloneqq \frac{1}{N}\sum_{i=1}^N\theta_i,\quad
\omega_c^0 \coloneqq \frac{1}{N} \sum_{i=1}^{N} \omega_i^0,\quad
\omega_c \coloneqq \frac{1}{N}\sum_{i=1}^N\omega_i,\quad
\nu_c = \frac{1}{N}\sum_{i=1}^N \nu_i.
\end{equation}
\end{lemma}
\begin{proof} 
We sum up the both sides of
\[
\dot\theta_i = \omega_i, \quad \dot\omega_i = \frac{1}{m} \Big( \nu_i - \omega_i + \frac{\kappa}{N}\sum_{j=1}^N \sin(\theta_j -\theta_i) \Big)
\]
with respect to $i$, and then use the defining relations of mean values \eqref{B-2} to find 
\[
{\dot \theta}_c = \omega_c, \quad  {\dot \omega}_c = \frac{\nu_c}{m} - \frac{\omega_c}{m}, \quad t > 0.
\]
We integrate the above ODEs to find the desired estimates. 
\end{proof}
\begin{remark}\label{R2.1}
Below, we provide several comments. 
\begin{enumerate}
\item
It follows from \eqref{eq:galconserve} that
\[ \lim_{t \to \infty} |\theta_c(t) - (\theta_c^0  + m \omega_c^0 + \nu_c(t-m))| = 0, \quad \lim_{t \to \infty} |\omega_c(t) - \nu_c| = 0.\]
\item
For $\nu_c = 0$, it follows from \eqref{eq:galconserve} that 
\[ \sup_{0 \leq t < \infty} |\theta_c(t)| \leq m |\omega_c^0 | +  |\theta_c^0|, \qquad  \sup_{0 \leq t < \infty} |\omega_c(t)| \leq  |\omega_c^0|, \]
so that
\[
|\theta_i(t) | \leq |\theta_i(t) - \theta_c(t)| +  |\theta_c(t)| \leq \frac{1}{N} \sum_{i=1}^{N} |\theta_i(t) - \theta_j(t)| + \sup_{0 \leq t < \infty} |\theta_c(t)| \leq {\mathcal D}(\Theta(t)) +  m |\omega_c^0 | +  |\theta_c^0|.
\]
This yields the equivalence between the uniform boundedness of the maximal phase and the phase diameter in the case $\nu_c = 0$:
\begin{equation} \label{B-3}
\sup_{0 \leq t < \infty} |\theta_i(t)| < \infty \quad \Longleftrightarrow \quad \sup_{0 \leq t < \infty} {\mathcal D}(\Theta(t)) < \infty.
\end{equation}
\end{enumerate}
\end{remark}
The conserved quantities of Lemma \ref{L2.1} correspond to a linear Galilean symmetry of \eqref{A-1}, given as follows. For $\nu,\omega,\theta\in \mathbb{R}$, we define
\begin{equation}\label{eq:galilean}
\begin{cases}
\displaystyle \tilde{\nu}_i\coloneqq\nu_i-\nu,\quad \tilde{\theta}_i^0\coloneqq \theta_i^0-\theta,\quad \tilde{\omega}_i^0\coloneqq \omega_i^0-\omega,\\
\displaystyle \tilde{\theta}_i(t)\coloneqq \theta_i(t)-\theta-m\omega \left(1-e^{-t/m}\right)-\nu \left(t-m+me^{-t/m}\right),\\
\displaystyle \tilde{\omega}_i(t)\coloneqq \omega_i(t)-\omega e^{-t/m}-\nu \left(1-e^{-t/m}\right).
\end{cases}
\end{equation}
\begin{proposition}\label{prop:sym}
Let $(\Theta(t),\dot{\Theta}(t))$ be a global solution to the Cauchy problem \eqref{A-1} with initial data $(\Theta^0,\dot{\Theta}^0)$, 
Then, the following assertions hold.
\begin{enumerate}
\item
The transformed state $(\tilde{\Theta}(t),\dot{\tilde{\Theta}}(t))$ in \eqref{eq:galilean} is a global solution to the Cauchy problem \eqref{A-1} with initial data $(\tilde{\Theta}^0,\dot{\tilde{\Theta}}^0)$. 
\vspace{0.1cm}
\item
If we set $\nu=\nu_c$ and $\omega=\omega^0_c$, $\Theta(t)$ is a phase-locked state, meaning that $\theta_i(t)-\theta_j(t)$ is constant with respect to $t$ for all $i,j$, if and only if $\tilde{\Theta}(t)$ is constant.
\vspace{0.1cm}
\item
If we set $\nu=\nu_c$, then $\Theta(t)$ exhibits asymptotic phase-locking if and only if $\tilde{\Theta}(t)$ converges.
\end{enumerate}
\end{proposition}
\begin{proof}
It is easy to see that $(\tilde{\theta}_i, \tilde{\nu}_i)$ satisfies the inertial Kuramoto model with the stipulated parameters, i.e.,
\begin{equation*}
\begin{cases} 
\displaystyle m\ddot{\tilde{\theta}}_i + \dot{\tilde{\theta}}_i = \tilde{\nu}_i +\frac{\kappa}{N}\sum_{j=1}^N \sin(\tilde{\theta}_j - \tilde{\theta}_i),\quad t>0,\\
\displaystyle (\tilde{\theta}_i, \dot{\tilde \theta}_i) \Big|_{t = 0+} = (\theta^0_i-\theta, \omega^0_i - \omega),\quad i\in[N].
\end{cases}
\end{equation*}

From $\theta_i(t)-\theta_j(t)=\tilde{\theta}_i(t)-\tilde{\theta}_j(t)$, it is easy to see that if $\tilde{\Theta}(t)$ is constant then $\Theta(t)$ is a phase-locked state, and that if $\tilde{\Theta}(t)$ converges then $\Theta(t)$ exhibits asymptotic phase-locking.

Conversely, if $\nu=\nu_c$ and $\omega=\omega_c^0$ and $\Theta(t)$ is a phase-locked state, then from the fact that the pairwise differences $\tilde{\theta}_i(t)-\tilde{\theta}_j(t)$ are constant and from the conservation law
\[
\frac 1N\sum_{i=1}^N\tilde{\theta}_i(t)\stackrel{\eqref{eq:galconserve},\eqref{eq:galilean}}{=}\frac 1N\sum_{i=1}^N\tilde{\theta}_i^0 \eqqcolon \tilde{\theta}_c^0,
\]
it follows that $\tilde{\Theta}(t)$ is constant. To show this, one may sum up the constants $\tilde{\theta}_i(t)-\tilde{\theta}_j(t)$ over the index $j$ (with $i$ fixed). In this case, this yields 
$\sum_{j=1}^N (\tilde{\theta}_i(t)-\tilde{\theta}_j(t))/N = \tilde{\theta}_i(t) - \tilde{\theta}_c^0$. If $\nu=\nu_c$ and $\Theta(t)$ exhibits asymptotic phase-locking, then again from the fact that the pairwise differences $\tilde{\theta}_i(t)-\tilde{\theta}_j(t)$ converge and from the fact that the average normalized phase
\[
\frac 1N\sum_{i=1}^N\tilde{\theta}_i(t)\stackrel{\eqref{eq:galconserve},\eqref{eq:galilean}}{=}m(\omega_c^0-\omega)(1-e^{-t/m})+\frac 1N\sum_{i=1}^N\tilde{\theta}_i^0
\]
converges, it follows that $\tilde{\Theta}(t)$ converges.
\end{proof}

\begin{remark}
Because the conditions we require in our synchronization Theorems \ref{simplemainthm} and \ref{T3.1} are invariant under Galilean symmetry \eqref{eq:galilean}, 
we may assume in the rest of this paper that
\[
\nu_c=0\quad \mathrm{and} \quad\omega_c^0=0.
\]
\end{remark}

\subsection{Phase-locking as relative equilibria}
The Galilean symmetry of Proposition \ref{prop:sym} is the basis of the definition of relative equilibria for \eqref{A-1}, and the statements and arguments of this paper are invariant under this symmetry. More precisely, due to \eqref{eq:galconserve}, system \eqref{A-1} can possess (absolute) equilibria only if $\nu_c= 0$. Nevertheless, even when $\nu_c\neq 0$, we can take the Galilean transformation \eqref{eq:galilean} with $\nu=\nu_c$, $\theta=\theta^0_c$, and $\omega=\omega_c^0$, after which we have $\theta_c(t)\equiv 0$ and $\omega_c(t)\equiv 0$ and it makes sense to discuss equilibria of the transformed variables. Solutions that transform under \eqref{eq:galilean} into equilibria are called \textit{phase-locked states}, i.e., they are equilibria relative to an appropriately rotating frame with asymptotic velocity $\nu_c$. Solutions are said to exhibit \textit{asymptotic phase-locking} if their Galilean transforms converge to equilibria; these are the single traveling solutions appearing in Theorem \ref{simplemainthm}. It is not hard to see that these definitions are equivalent to the following definitions. 
\begin{definition}\label{D1.1}
Let $(\Theta(t), {\dot \Theta}(t))$ be a global solution to the Cauchy problem \eqref{A-1} or \eqref{A-2}.
\begin{enumerate}
\item
$(\Theta(t), {\dot \Theta}(t))$ is a \emph{phase-locked state} if for all $i,j\in [N]$, $\theta_i(t)-\theta_j(t)$ is constant with respect to $t$.
\item
$(\Theta(t), {\dot \Theta}(t))$ exhibits \emph{asymptotic phase-locking} if
\[  \exists~  \lim_{t \to \infty} (\theta_i(t) - \theta_j(t) ), \quad \forall~ i, j \in [N]. \]
\item
$(\Theta(t), {\dot \Theta}(t))$ exhibits \emph{complete frequency synchronization} if
\[  \lim_{t \to \infty} \max_{i, j \in [N]} |{\dot \theta}_i(t) - {\dot \theta}_j(t) | = 0. \]
\item
$(\Theta(t), {\dot \Theta}(t))$ exhibits \emph{complete phase synchronization} if
\[  \lim_{t \to \infty} \left(\theta_i(t) - \theta_j(t)\right)  \in 2\pi \mathbb{Z},\quad \forall i,j\in [N]. \]
\end{enumerate}
\end{definition}
\begin{remark}\label{rem:pls}
Below, we provide several comments on Definition \ref{D1.1}.
\begin{enumerate}
    \item Asymptotic phase-locking implies complete frequency synchronization; one can easily see this from the Duhamel principle, later presented in \eqref{B-4}.
    \vspace{0.1cm}
    \item Complete phase synchronization can happen only if $\nu_i=\nu_j$ for all $i,j\in [N]$. Again, this is due to the Duhamel principle \eqref{B-4}.
\vspace{0.1cm}
 \item Phase-locked states for \eqref{A-1} coincide with phase-locked states for \eqref{A-2}, up to rotation of the circle, when $\omega^0_i\neq\omega_j^0$ for some $i,j$.
For a fixed initial phase $\Theta^0$, the constant map $\Theta(t) = \Theta^0$ is a phase-locked state for both \eqref{A-1} and \eqref{A-2}.
If $\nu_c = \nu_i + \frac{\kappa}{N}\sum_{j=1}^N\sin(\theta_j^0-\theta_i^0)$ for all $i,j$, then the linear trajectory $\Theta(t)=\nu_c t \mathbf{1}_{[N]} +\Theta^0$ is a phase-locked state for both \eqref{A-1} and \eqref{A-2}. A difference arises when, in addtion, $\omega_i^0 = \omega_j^0 \eqqcolon \omega^0$ for all $i,j$. In this case, $\Theta(t) = [m\omega^0(1-e^{-t/m})+\nu_c(t-m+me^{-t/m})]\mathbf{1}_{[N]}+\Theta^0$ is a phase-locked state for \eqref{A-1}.
\end{enumerate}
    
\end{remark}
Having fixed other parameters and increasing $\kappa$ to large values, it can be observed numerically that phase-locked states emerge after a certain critical threshold \cite{D-B14, V-M}. However, before reaching that threshold, one sees partially ordered behavior among the oscillators \cite{B-W}. The following definition captures this concept.
\begin{definition}[{\cite[Definition 4.1]{H-K-R}}, {\cite[Definition 2.2]{H-R}}]\label{def:partial}
Let $(\Theta(t), {\dot \Theta}(t))$ be a global solution to the Cauchy problem \eqref{A-1}.
\begin{enumerate}
\item
Given $\mathcal{A}\subset [N]$, we say the solution $(\Theta(t), {\dot \Theta}(t))$ exhibits \emph{$\mathcal{A}$-partial phase-locking} if
\[ 
\sup_{t\ge 0}\sup_{i,j\in \mathcal{A}}|\theta_i(t)-\theta_j(t)|<\infty.
\]
\item
Given $\lambda\in (0,1]$, we say the solution $(\Theta(t), {\dot \Theta}(t))$ exhibits \emph{$\lambda$-partial phase-locking} if there exists $\mathcal{A}\subset[N]$ with $|\mathcal{A}|\ge \lambda N$ such that the solution $(\Theta(t), {\dot \Theta}(t))$ exhibits $\mathcal{A}$-partial phase-locking.
\end{enumerate}
\end{definition}
\begin{remark}
In fact, $[N]$-partial phase locking, or equivalently $1$-partial phase-locking, is equivalent to asymptotic phase-locking. See Remark \ref{R2.2}.
\end{remark}

\subsection{Symmetries of the inertial Kuramoto model}\label{subsec:sym}
In this subsection, we list several symmetries related to the inertial Kuramoto model in what follows. (Here, $\Delta(\Theta)$ is defined in \eqref{eq:Delta}.)
\vspace{0.1cm}
\begin{enumerate}
\item  (Translation symmetry):~ the quantities $\mathcal{D}(\mathcal{V})$, $\mathcal{D}(\Omega^0)$, $\mathcal{D}(\Theta)$, $R(\Theta)$ and $\Delta(\Theta)$ are invariant under the transformation \eqref{eq:galilean}.
\vspace{0.1cm}
\item  (Dilation symmetry):~the `normalized intrinsic frequencies' $\nu_i/\kappa$, `normalized inertia' $m\kappa$, and `normalized initial velocites' $\omega_i^0/\kappa$ are invariant under the time dilation symmetry: for fixed $\alpha>0$,
\begin{align}\label{eq:dilatation}
\begin{aligned}
&\kappa\mapsto \alpha\kappa,\quad \nu_i\mapsto\alpha\nu_i,\quad m\mapsto m/\alpha,\\
&\theta_i^0\mapsto \theta_i^0, \quad \omega_i^0\mapsto \alpha \omega_i^0,\quad \theta_i(t)\mapsto \theta_i(\alpha t),\quad \omega_i(t)\mapsto \alpha\omega_i(\alpha t),
\end{aligned}
\end{align}
\vspace{0.1cm}
\item (Reflection symmetry): the quantities $\mathcal{D}(\mathcal{V})$, $\mathcal{D}(\Omega^0)$, $\mathcal{D}(\Theta)$, $R(\Theta)$ and $\Delta(\Theta)$ are invariant under the following transformation:
\begin{equation*}
\nu_i\mapsto -\nu_i,\quad \theta^0_i\mapsto -\theta^0_i,\quad \omega_i^0\mapsto -\omega_i^0,\quad \theta_i(t)\mapsto -\theta_i(t),\quad \omega_i(t)\mapsto -\omega_i(t).
\end{equation*}
\item (Particle exchange symmetry)\footnote{ In variants of \eqref{A-1} involving connectivity weights other than all-to-all uniform coupling of the model \eqref{A-1}, the particle exchange symmetry should also act on the connectivity weights $a_{ij}$, sending $a_{ij}\mapsto a_{\pi^{-1}(i)\pi^{-1}(j)}$.}: the quantities $\mathcal{D}(\mathcal{V})$, $\mathcal{D}(\Omega^0)$, $\mathcal{D}(\Theta)$, $R(\Theta)$, and $\Delta(\Theta)$ are invariant under the following transformation:
\begin{equation}\label{eq:particle-exchange}
  \nu_i\mapsto \nu_{\pi^{-1}(i)},\quad \theta_i^0\mapsto \theta_{\pi^{-1}(i)}^0,\quad \omega_i^0\mapsto \omega_{\pi^{-1}(i)}^0
\end{equation}
for fixed $\pi\in S_N$ (the symmetric group on $N$ elements).
\end{enumerate}
\vspace{0.2cm}

Note that the quantities $m\kappa$, $\mathcal{D}(\mathcal{V})/\kappa$, and $\mathcal{D}(\Omega^0)$, that we are requiring to be small in our synchronization framework, are invariant under all four symmetries listed above. When establishing theorems regarding the Kuramoto model \eqref{A-1}, it would be ``natural'' for those theorems to have assumptions and statements that are invariant under the above symmetries. In some cases, if we have a non-symmetric statement regarding the Kuramoto model \eqref{A-1}, we can find a symmetric counterpart. For example, if the solution to \eqref{A-1} with initial data $(\Theta^0,\dot{\Theta}^0)$ and parameters $\kappa$, $m$, and $\mathcal{V}$ exhibits asymptotic phase-locking, so does the solution to \eqref{A-1} with initial data $(\Theta^0,\dot{\tilde{\Theta}}^0)$ and parameters $\tilde{\kappa}$, $\tilde{m}$, and $\tilde{\mathcal{V}}$, if $\dot{\tilde{\Theta}}^0/\tilde{\kappa}=\dot{\Theta}^0/\kappa$, $\tilde{m}\tilde{\kappa}=m\kappa$, and $\tilde{\mathcal{V}}/\tilde{\kappa}=\mathcal{V}/\kappa$. Another example is the concepts of ``phase-locked states'' and ``asymptotic phase-locking'', given in Definition \ref{D1.1}. These are equivalent to the concept of ``equilibria'' and ``convergence to equilibria'' after we transform system \eqref{A-1} using the Galilean transformation \eqref{eq:galilean} with $\nu=\nu_c$ and $\omega=\omega_c$, and these concepts are invariant under the above symmetries.

\subsection{A dynamic approximation scheme} \label{sec:2.1}
In this subsection, we will show that when $m\kappa$, $m\mathcal{D}(\mathcal{V})$, $m\mathcal{D}(\Omega)$ are small and $t/m$ is large, $\dot\Theta(t)$ can be approximately computed by the value given by the ODE of \eqref{A-2}. Such an approximation is vital since Gr\"onwall's inequality applies to first-order differential inequalities but not to second-order differential inequalities. For this, we first rewrite the $N$-dimensional second-order Cauchy problem \eqref{A-1} as the $2N$-dimensional Cauchy problem
\begin{equation}\label{B-1}
\begin{cases}
\displaystyle \dot\theta_i = \omega_i, \quad t >0, \\
\displaystyle \dot\omega_i = \frac{1}{m} \Big( \nu_i - \omega_i + \frac{\kappa}{N}\sum_{j=1}^N \sin(\theta_j -\theta_i) \Big),\\
\displaystyle  (\theta_i, \omega_i) \Big|_{t = 0+} = (\theta^{0}_i, \omega^0_i),\quad i\in[N].
\end{cases}
\end{equation}
Again, system \eqref{B-1} admits a unique global solution and describes the same system as \eqref{A-1}.\footnote{To be precise, the application of the Cauchy-Lipschitz theory to \eqref{A-1} is through the system \eqref{B-1}, so the uniqueness and global existence of solutions is first established for \eqref{B-1} and then transferred to \eqref{A-1}.} One advantage of this point of view is that it allows us to apply the Duhamel principle to $\omega_i$: write the ODE of $\omega_i$ as
\[
\begin{cases}
\displaystyle \left(\frac{d}{dt}+\frac 1m\right)\omega_i=\frac{1}{m} \left( \nu_i + \frac{\kappa}{N}\sum_{j=1}^N \sin(\theta_j -\theta_i) \right),~t>0, \\
\displaystyle  \omega_i(0)=\omega_i^0,
\end{cases}
\]
where we treat $m^{-1}\left(\nu_i+\frac{\kappa}{N}\sum_{j=1}^N \sin(\theta_j -\theta_i)\right)$ as an extraneous source term. Since $G(t)=e^{-t/m}$ solves
\[
\left(\frac{d}{dt}+\frac 1m\right)G=0,~t>0,\quad G(0)=1,
\]
we invoke Duhamel's principle to obtain
\begin{align}\label{B-4}
\begin{aligned}
\omega_i(t) &=\omega^0_iG(t)+\int_0^t G(t-s)\frac{1}{m} \left( \nu_i + \frac{\kappa}{N}\sum_{j=1}^N \sin(\theta_j(s) -\theta_i(s)) \right)ds\\
&=  \omega^0_i e^{-t/m} + \nu_i (1 -e^{-t/m}) + \frac{\kappa}{Nm}\sum_{j=1}^N \int_0^t e^{-(t-s)/m}\sin(\theta_j(s) -\theta_i(s)) ds.
\end{aligned}
\end{align}
As this is a weighted time-delayed version of the first-order model \eqref{A-2}, we could interpret the effect of inertia as a weighted time delay in interactions.

In fact, we have the following decomposition:
\begin{align}\label{eq:omega-decomposition}
\begin{aligned}
\omega_i(t) =&  \underbrace{\frac{\kappa\left(1-e^{-t/m}\right)}{N}\sum_{j=1}^N \sin(\theta_j(t)-\theta_i(t))}_{\mathrm{nonlinear~interaction~term}}+\underbrace{\omega^0_i e^{-t/m}}_{\mathrm{initial~frequency~term}} + \underbrace{\nu_i (1 -e^{-t/m})}_{\mathrm{intrinsic~frequency~term}} \\
&+ \underbrace{\frac{\kappa}{Nm}\sum_{j=1}^N \int_0^t e^{-(t-s)/m}\left(\sin(\theta_j(s) -\theta_i(s))-\sin(\theta_j(t) -\theta_i(t))\right) ds}_{\mathrm{time-delay~error~term}}.
\end{aligned}
\end{align}
As we will see in Subsection \ref{subsec:1.4} and Subsection \ref{subsec:1.5}, the ``nonlinear interaction term'' works to increase the degree of synchronization, measured either in terms of the order parameter (see \eqref{eq:Rphi} for the definition of order parameter) or the diameter of a majority cluster (see Definition \ref{def:majority-cluster}). However, it is unclear how the other terms, namely the ``initial frequency term'', ``intrinsic frequency term'', and ``time delay error terms'', contribute to synchronization; in the short term, they may even work towards desynchronizing the system, in the sense of order parameters or diameters of majority clusters. We also note that the nonlinear interaction term comes with the factor $1-e^{-t/m}$ and hence it is very small when $t\ll m$.

A proof strategy for Theorem \ref{simplemainthm} can be described as a passive game in which we, who wish to synchronize the system, are playing against an adversary who wishes to desynchronize the system. We may only passively set the nonlinear interaction term as it is, while the adversary is free to choose the other terms: they are even allowed to choose $\omega_i^0$ and $\nu_i$ adaptively and can manipulate the history of the particle, within certain limitations necessitated by the physics of the model \eqref{A-1}.

At times $t\ll m$, the adversary is free to set $\omega_i(t)$ within some physical limitations because, due to $1-e^{-t/m}\ll 1$, our feeble nonlinear interaction term contributes virtually nothing to \eqref{eq:omega-decomposition}. We only passively observe the adversary desynchronize the system (in whatever sense). After waiting until $t\gg m$, we activate the nonlinear interaction term while the adversary conspires the other three terms against us (again within physical limitations). To guarantee synchronization within our game, the total effect of the first term should triumph over the total effect of the last three terms.

Here are the physical limitations we impose on the adversary. Along with enforcing the diameters $\mathcal{D}(\mathcal{V})$ and $\mathcal{D}(\Omega^0)$, we also set
\[
\mathcal{D}(\Omega(t))\le \max\{\mathcal{D}(\Omega^0),\mathcal{D}(\mathcal{V})+2\kappa\},
\]
which is indeed true for \eqref{A-1} due to Lemma \ref{L2.2} below. For us to have any hope of winning this game, the physical limitations imposed against the adversary must:
\begin{enumerate}
\item limit their total influence on the system during the initial layer $t\ll m$; this can be enforced if
\begin{equation}\label{eq:small-influ}
m\cdot \sup_{i,j\in [N],~t\ge 0}|\omega_i(t)-\omega_j(t)|=m\cdot \sup_{t\ge 0}\mathcal{D}(\Omega(t))\le \max\{m\mathcal{D}(\Omega^0),m\mathcal{D}(\mathcal{V})+2m\kappa\}\ll 1
\end{equation}
(recall that we only care about the phase differences);
\item limit their choice of $\omega_i^0$ and $\nu_i$  compared to $\kappa$ (which is the scale of our nonlinear interaction term):
\begin{equation}\label{eq:small-param}
\mathcal{D}(\Omega^0)/\kappa,~\mathcal{D}(\mathcal{V})/\kappa\ll 1;
\end{equation}
\item make the time delay error term diminutive compared to $\kappa$: because of the crude bound
\begin{align*}
\begin{aligned}
& \frac 1m \int_0^t e^{-(t-s)/m}\left|\sin(\theta_l(s) -\theta_i(s))-\sin(\theta_l(t) -\theta_i(t))\right| ds \\
& \hspace{1cm} \le \frac 1m \int_0^t e^{-(t-s)/m}(t-s) ds\cdot \sup_{t\ge 0}\mathcal{D}(\Omega(t)) \le m\cdot \sup_{t\ge 0}\mathcal{D}(\Omega(t))
\end{aligned}
\end{align*}
this can be enforced using \eqref{eq:small-influ}.
\end{enumerate}

The conditions \eqref{eq:small-influ} and \eqref{eq:small-param} can be established if the system parameters satisfy
\begin{equation}\label{eq:framework-small-param}
\mathcal{D}(\Omega^0)/\kappa,~\mathcal{D}(\mathcal{V})/\kappa,~m\kappa \ll 1.
\end{equation}
These conditions are roughly our proposed framework for synchronization in Theorem \ref{simplemainthm} and Theorem \ref{T3.1} and are necessary for us to win the game.

Heuristically, the small $m$ regime has two benefits. The first is that the initial layer exposure to an adversarial attack, in the form of $\omega_i^0$ set unfavorably towards synchronization, is short and the effect towards the dynamics of the phase differences $\theta_i-\theta_j$ is minuscule. The second is that, after the initial layer, we can quickly recover a dominant first-order term that makes the model behave like the first-order Kuramoto model. In reality, the Cauchy problem \eqref{A-1} is not a game, at least in the sense we described: no adversary is working against us with the authority to arbitrarily set $\mathcal{V}$, $\Omega$, and the time delay error term to their whim. This myopic viewpoint of \eqref{A-1} led to the restrictions $\mathcal{D}(\Omega^0)/\kappa \ll 1$ and $m\kappa\ll 1$ in \eqref{eq:framework-small-param} which are potentially unnecessary,\footnote{The smallness of $\mathcal{D}(\mathcal{V})/\kappa$ is necessary; see \cite{D-B11}.} as suggested by the simulations in the next section. To dispense with the conditions $\mathcal{D}(\Omega^0)/\kappa \ll 1$ and $m\kappa\ll 1$, we would have to come up with a viewpoint that is robust to minor changes in the intrinsic velocity term and major changes in the initial velocity term and the time delay error term, possibly even making these three terms cooperate towards synchronization; for a candidate of such a framework, see {\bf Conjecture \ref{conj:R}.} Next, we derive crude bounds on $\dot\Theta(t)$, which corresponds to the physical limitations placed on the adversary in the above game.
\begin{lemma}[Finite propagation speed {\cite[Lemma 2.2]{C-L}}, {\cite[Lemmas 1{~and~}4]{H-J-K}}]\label{L2.2}
Let $(\Theta, \Omega)$ be a global solution to \eqref{B-1}. Then, for $i, j \in [N]$ and $t\ge 0$,  one has
\begin{align*}
\begin{aligned}
& (1)~e^{-t/m}\omega_i^0 +(1-e^{-t/m})(\nu_i - \kappa) \le \omega_i(t) \le e^{-t/m}\omega_i^0 +(1-e^{-t/m})(\nu_i + \kappa). \\
& (2)~ |\omega_i(t)-\omega_j(t)| \le e^{-t/m}|\omega_i^0-\omega_j^0| +(1-e^{-t/m})(|\nu_i-\nu_j| + 2\kappa). \\
& (3)~\mathcal{D}(\Omega(t))\le e^{-t/m}{\mathcal D}(\Omega^0)  + (1-e^{-t/m})\left( {\mathcal D}({\mathcal V}) + 2 \kappa\right). \\
& (4)~\int_0^t e^{-(t-s)/m}\left|\left(\theta_i(s)-\theta_j(s)\right)-\left(\theta_i(t)-\theta_j(t)\right)\right|ds \\
& \hspace{1cm} \le m|\omega^0_i-\omega_j^0|te^{-t/m}\left(1-e^{-t/m}\right)+m^2 (|\nu_i-\nu_j|+2\kappa)\left(1-e^{-t/m}\right)^3.
\end{aligned}
\end{align*}    
\end{lemma}
\begin{proof}
(i)~We use \eqref{B-4} and  with the estimate
\begin{equation*}
\left|  \int_0^t e^{-(t-s)/m}\sin(\theta_j(s) -\theta_i(s)) ds \right| \le  \int_0^t e^{-(t-s)/m} ds  \le   m (1-e^{-t/m}). 
\end{equation*}
to derive the desired estimate.   \newline

\noindent (ii)~We use the estimate (1) to find 
\[
e^{-t/m}(\omega^0_i - \omega_j^0)  + (1 -e^{-t/m})(\nu_i - \nu_j-2\kappa)\le \omega_i(t) - \omega_j(t) \le e^{-t/m}(\omega^0_i - \omega_j^0)  + (1 -e^{-t/m})(\nu_i - \nu_j+2\kappa).
\]
This implies the desired estimate for $|\omega_i - \omega_j|$.  \newline

\noindent (iii)~We take the maximum in (2) over all $i,j\in [N]$ to find the desired estimate. \newline

\noindent (iv)~We use the estimate (2) to see that for $0\le s\le t$,
\begin{align}
\begin{aligned} \label{NN-4}
&\left|\left(\theta_i(s)-\theta_j(s)\right)-\left(\theta_i(t)-\theta_j(t)\right)\right|\\
& \hspace{0.5cm} \le \int_s^t |\omega_i(\tau)-\omega_j(\tau)|d\tau \stackrel{(2)}{\le} \int_s^t \left(|\omega_i^0-\omega_j^0| e^{-\tau/m} + \left( |\nu_i-\nu_j| + 2 \kappa\right) (1-e^{-\tau/m})\right)d\tau\\
& \hspace{0.5cm} \le |\omega_i^0-\omega_j^0|m\left( e^{-s/m}-e^{-t/m}\right) + \left( |\nu_i-\nu_j| + 2 \kappa\right) (t-s-me^{-s/m}+me^{-t/m}).
\end{aligned}
\end{align}
Now, we multiply $e^{-(t-s)/m}$ to \eqref{NN-4} and integrate the resulting relation to obtain 
\begin{align*}
&\int_0^t e^{-(t-s)/m}\left|\left(\theta_i(s)-\theta_j(s)\right)-\left(\theta_i(t)-\theta_j(t)\right)\right|ds\\
& \hspace{0.5cm} \le m|\omega_i^0-\omega_j^0|\left(te^{-t/m}-me^{-t/m}+me^{-2t/m}\right) + m (|\nu_i-\nu_j|+2\kappa)\left(m-2te^{-t/m}-me^{-2t/m}\right)\\
&  \hspace{0.5cm}  \le m|\omega_i^0-\omega_j^0|te^{-t/m}\left(1-e^{-t/m}\right) +m^2 (|\nu_i-\nu_j|+2\kappa)\left(1-e^{-t/m}\right)^3,
\end{align*}
where we have used the calculus inequalities
\[
xe^{-x}-e^{-x}+e^{-2x}\le xe^{-x}\left(1-e^{-x}\right),~1-2xe^{-x}-e^{-2x}\le \left(1-e^{-x}\right)^3,\quad x\ge 0.
\]
\end{proof}
In \eqref{eq:omega-decomposition}, the ``nonlinear interaction term'', ``initial frequency term'', and the ``intrinsic frrequency term'' are computable given the initial conditions $\mathcal{V}$ and $\Omega^0$ and the current condition $\Theta(t)$. Lemma \ref{L2.2} (4) tames the ``time-delay error term.'' This allows us to gain an even better approximation of $\dot\Theta(t)$. Furthermore, this allows for controlling $\dot\Theta_\mathcal{A}(t)$ given knowledge of $\Theta_\mathcal{A}(t)$, $\Omega_\mathcal{A}^0$, and $\mathcal{V}_\mathcal{A}$, for a subset $\mathcal{A}\subset [N]$; see Lemma \ref{L:approxaut} below. This is needed when $\mathcal{A}$ is a majority cluster concentrated on a small arc (see Definition \ref{def:majority-cluster}). We again stress the importance of this ``partial controlling lemma'' in that it will allow us to use the first-order Gr\"onwall inequalities.

\begin{lemma}\label{L:approxaut}
Let $(\Theta,\Omega)$ be a global solution to \eqref{A-1}. For $\mathcal{A}\subset [N]$ and $t\ge 0$,  the following statements hold. 
\begin{enumerate}
\item
For $i\in \mathcal{A}$,
\begin{align*}
&\left|\dot{\theta}_i(t)-\omega_i^0 e^{-t/m}-\nu_i\left(1-e^{-t/m}\right)-\frac \kappa N\sum_{l\in \mathcal{A}}\sin\left(\theta_l(t)-\theta_i(t)\right)\left(1-e^{-t/m}\right)\right|\\
& \hspace{0.5cm} \le \frac{|\mathcal{A}|}{N}\kappa\mathcal{D}(\Omega^0_\mathcal{A})te^{-t/m}\left(1-e^{-t/m}\right) +\frac{|\mathcal{A}|}{N}m\kappa (\mathcal{D}(\mathcal{V}_\mathcal{A})+2\kappa)\left(1-e^{-t/m}\right)^3 +\frac{N-|\mathcal{A}|}{N}\kappa (1-e^{-t/m}).
\end{align*}
\item
For $i,j\in \mathcal{A}$,
\begin{align*}
&\Bigg|\dot\theta_i(t)-\dot\theta_j(t)-(\omega^0_i-\omega^0_j) e^{-t/m}-(\nu_i-\nu_j) (1 -e^{-t/m})\\
& \hspace{0.5cm} -  \frac{\kappa}{N}\left(1-e^{-t/m}\right)\sum_{l\in \mathcal{A}}\left(\sin(\theta_l(t)-\theta_i(t))-\sin(\theta_l(t)-\theta_j(t))\right)\Bigg|\\
& \hspace{0.5cm} \le  2\kappa \left(\mathcal{D}(\Omega^0_\mathcal{A})te^{-t/m}\left(1-e^{-t/m}\right)+m (\mathcal{D}(\mathcal{V}_\mathcal{A})+2\kappa)\left(1-e^{-t/m}\right)^3\right)\\
& \hspace{0.5cm} +\frac{N-|\mathcal{A}|}{N}2\kappa\left|\sin\frac{\theta_i(t) -\theta_j(t)}{2}\right|\left(1-e^{-t/m}\right).
\end{align*}
\item
For $\mathcal{A}=[N]$ and $i\in [N]$, 
\begin{align*}
&\left|\dot{\theta}_i(t)-\omega_i^0 e^{-t/m}-\nu_i\left(1-e^{-t/m}\right)-\frac \kappa N\sum_{l=1}^N\sin\left(\theta_l(t)-\theta_i(t)\right)\left(1-e^{-t/m}\right)\right|\\
& \hspace{0.5cm} \le \kappa\mathcal{D}(\Omega^0)te^{-t/m}\left(1-e^{-t/m}\right) +m\kappa (\mathcal{D}(\mathcal{V})+2\kappa)\left(1-e^{-t/m}\right)^3.
\end{align*}
\end{enumerate}
\end{lemma}
\begin{proof}
\noindent (i)~It follows from \eqref{eq:omega-decomposition} that 
\begin{align}\label{eq:omega-decomposition-partial}
\begin{aligned}
\omega_i(t) &=  \frac{\kappa\left(1-e^{-t/m}\right)}{N}\sum_{l\in \mathcal{A}}\sin(\theta_l(t)-\theta_i(t))+\omega^0_i e^{-t/m} + \nu_i (1 -e^{-t/m}) \\
&+ \frac{\kappa}{Nm}\sum_{l\in \mathcal{A}} \int_0^t e^{-(t-s)/m}\left(\sin(\theta_l(s) -\theta_i(s))-\sin(\theta_l(t) -\theta_i(t))\right) ds\\
&+ \frac{\kappa}{Nm}\sum_{l\in [N]\setminus \mathcal{A}} \int_0^t e^{-(t-s)/m}\sin(\theta_l(s) -\theta_i(s)) ds.
\end{aligned}
\end{align}
Then, for $i\in \mathcal{A}$,
\begin{align*}
&\left|\dot{\theta}_i(t)-\omega_i^0 e^{-t/m}-\nu_i\left(1-e^{-t/m}\right)-\frac \kappa N\left(1-e^{-t/m}\right)\sum_{l\in \mathcal{A}}\sin\left(\theta_l(t)-\theta_i(t)\right)\right|\\
& \hspace{0.5cm} \stackrel{\mathclap{\eqref{eq:omega-decomposition-partial}}}{\le} \frac{\kappa}{Nm}\sum_{l\in \mathcal{A}}\int_0^t e^{-(t-s)/m}\left|\sin\left(\theta_l(s)-\theta_i(s)\right)-\sin\left(\theta_l(t)-\theta_i(t)\right)\right|ds\\
&  \hspace{0.5cm} \quad + \frac{\kappa}{Nm}\sum_{l\in [N]\setminus \mathcal{A}} \int_0^t e^{-(t-s)/m}\left|\sin(\theta_l(s) -\theta_i(s))\right| ds\\
&  \hspace{0.5cm}\le \frac{\kappa}{Nm}\sum_{l\in \mathcal{A}}\int_0^t e^{-(t-s)/m}\left|\left(\theta_l(s)-\theta_i(s)\right)-\left(\theta_l(t)-\theta_i(t)\right)\right|ds\\
&  \hspace{0.5cm} \quad+ \frac{N-|\mathcal{A}|}{N}\frac{\kappa}{m}\int_0^t e^{-(t-s)/m} ds\\
&  \hspace{0.5cm} \stackrel{\mathclap{\mathrm{Lemma~}\ref{L2.2}\mathrm{(4)}}}{\le}\qquad\frac{|\mathcal{A}|}{N}\kappa\mathcal{D}(\Omega^0_\mathcal{A})te^{-t/m}\left(1-e^{-t/m}\right) +\frac{|\mathcal{A}|}{N}m\kappa (\mathcal{D}(\mathcal{V}_\mathcal{A})+2\kappa)\left(1-e^{-t/m}\right)^3 \\
&  \hspace{0.5cm} \quad +\frac{N-|\mathcal{A}|}{N}\kappa (1-e^{-t/m}).
\end{align*}

\vspace{0.2cm}

\noindent (ii)~For $i,j\in \mathcal{A}$, we use \eqref{eq:omega-decomposition-partial} and the same argument as in (i) to find 
\begin{align*}
&\Bigg|\dot\theta_i(t)-\dot\theta_j(t)-(\omega^0_i-\omega^0_j) e^{-t/m}-(\nu_i-\nu_j) (1 -e^{-t/m})\\
& \hspace{1cm} -  \frac{\kappa}{N}\left(1-e^{-t/m}\right)\sum_{l\in \mathcal{A}}\left(\sin(\theta_l(t)-\theta_i(t))-\sin(\theta_l(t)-\theta_j(t))\right)\Bigg|\\
& \hspace{0.5cm} \le \left|\frac{\kappa}{Nm}\sum_{l\in \mathcal{A}} \int_0^t e^{-(t-s)/m}\left(\sin(\theta_l(s) -\theta_i(s))-\sin(\theta_l(t) -\theta_i(t))\right) ds\right|\\
& \hspace{0.5cm}  \quad+\left|\frac{\kappa}{Nm}\sum_{l\in \mathcal{A}} \int_0^t e^{-(t-s)/m}\left(\sin(\theta_l(s) -\theta_j(s))-\sin(\theta_l(t) -\theta_j(t))\right) ds\right|\\
& \hspace{0.5cm}  \quad+ \left|\frac{\kappa}{Nm}\sum_{l\in [N]\setminus \mathcal{A}} \int_0^t e^{-(t-s)/m}\left(\sin(\theta_l(s) -\theta_i(s))-\sin(\theta_l(s) -\theta_j(s))\right) ds\right|\\
&  \hspace{0.5cm}  \le   \frac{\kappa}{Nm}\sum_{l\in \mathcal{A}} \int_0^t e^{-(t-s)/m}\left|(\theta_l(s) -\theta_i(s))-(\theta_l(t) -\theta_i(t))\right| ds\\
&  \hspace{0.5cm}  \quad +\frac{\kappa}{Nm}\sum_{l\in \mathcal{A}} \int_0^t e^{-(t-s)/m}\left|(\theta_l(s) -\theta_j(s))-(\theta_l(t) -\theta_j(t))\right| ds\\
& \hspace{0.5cm}  \quad + \frac{N-|\mathcal{A}|}{N}\frac{2\kappa}{m} \int_0^t e^{-(t-s)/m}\left|\sin\frac{\theta_i(s) -\theta_j(s)}{2}\right| ds\\
&  \hspace{0.5cm}  \stackrel{\mathclap{\mathrm{Lemma~}\ref{L2.2}\mathrm{(4)}}}{\le}  \qquad 2\cdot\frac{|\mathcal{A}|}{N} \kappa \left(\mathcal{D}(\Omega^0_\mathcal{A})te^{-t/m}\left(1-e^{-t/m}\right)+m (\mathcal{D}(\mathcal{V}_\mathcal{A})+2\kappa)\left(1-e^{-t/m}\right)^3\right)\\
& \hspace{0.5cm}  \quad + \frac{N-|\mathcal{A}|}{N}\frac{2\kappa}{m} \int_0^t e^{-(t-s)/m}\left|\sin\frac{\theta_i(s) -\theta_j(s)}{2}-\sin\frac{\theta_i(t) -\theta_j(t)}{2}\right| ds\\
& \hspace{0.5cm}  \quad + \frac{N-|\mathcal{A}|}{N}\frac{2\kappa}{m} \int_0^t e^{-(t-s)/m}\left|\sin\frac{\theta_i(t) -\theta_j(t)}{2}\right| ds\\
&  \hspace{0.5cm}  \le  2\cdot\frac{|\mathcal{A}|}{N} \kappa \left(\mathcal{D}(\Omega^0_\mathcal{A})te^{-t/m}\left(1-e^{-t/m}\right)+m (\mathcal{D}(\mathcal{V}_\mathcal{A})+2\kappa)\left(1-e^{-t/m}\right)^3\right)\\
&  \hspace{0.5cm}  \quad + \frac{N-|\mathcal{A}|}{N}\frac{\kappa}{m} \int_0^t e^{-(t-s)/m}\left|(\theta_i(s) -\theta_j(s))-(\theta_i(t) -\theta_j(t))\right| ds\\
&  \hspace{0.5cm}  \quad +\frac{N-|\mathcal{A}|}{N}2\kappa\left|\sin\frac{\theta_i(t) -\theta_j(t)}{2}\right|\left(1-e^{-t/m}\right)\\
\stackrel{\mathclap{\mathrm{Lemma~}\ref{L2.2}\mathrm{(4)}}}{\le} &\qquad 2\cdot \frac{|\mathcal{A}|}{N} \kappa \left(\mathcal{D}(\Omega^0_\mathcal{A})te^{-t/m}\left(1-e^{-t/m}\right)+m (\mathcal{D}(\mathcal{V}_\mathcal{A})+2\kappa)\left(1-e^{-t/m}\right)^3\right)\\
& \hspace{0.5cm}  \quad + \frac{N-|\mathcal{A}|}{N}\kappa\left(\mathcal{D}(\Omega^0_\mathcal{A})te^{-t/m}\left(1-e^{-t/m}\right)+m (\mathcal{D}(\mathcal{V}_\mathcal{A})+2\kappa)\left(1-e^{-t/m}\right)^3\right)\\
& \hspace{0.5cm}  \quad +\frac{N-|\mathcal{A}|}{N}2\kappa\left|\sin\frac{\theta_i(t) -\theta_j(t)}{2}\right|\left(1-e^{-t/m}\right)\\
&  \hspace{0.5cm}  \le  2\kappa \left(\mathcal{D}(\Omega^0_\mathcal{A})te^{-t/m}\left(1-e^{-t/m}\right)+m (\mathcal{D}(\mathcal{V}_\mathcal{A})+2\kappa)\left(1-e^{-t/m}\right)^3\right)\\
& \hspace{0.5cm}  \quad +\frac{N-|\mathcal{A}|}{N}2\kappa\left|\sin\frac{\theta_i(t) -\theta_j(t)}{2}\right|\left(1-e^{-t/m}\right).
\end{align*}

\vspace{0.2cm}

\noindent (iii)~In (1), we set $\mathcal{A}=[N]$ to find the desired estimate. 
\end{proof}

\begin{remark}
To check that our framework is applicable, we see that the additive error is bounded by
\[
\kappa\mathcal{D}(\Omega^0)te^{-t/m}\left(1-e^{-t/m}\right) +m\kappa (\mathcal{D}(\mathcal{V})+2\kappa)\left(1-e^{-t/m}\right)^3\le m\kappa \mathcal{D}(\Omega^0)+m\kappa (\mathcal{D}(\mathcal{V})+2\kappa),
\]
which is indeed $\ll \kappa$ under our crude framework \eqref{eq:framework-small-param}.
\end{remark}
The approximation scheme in this subsection raises the following question:
\begin{question}
``Given the system parameters $m$, $\kappa$, $\mathcal{V}$, initial frequency data $\dot\Theta(T)$, a time $t\ge 0$, and position data $\Theta(t)$ at time $t$, is the velocity data $\dot\Theta(t)$ uniquely determined? If so, how can we compute $\dot\Theta(t)$?"
\end{question}
We defer a partial answer and a rigorous analysis to a forthcoming paper.

\section{Synchronization Mechanism, Conjectures and Strategy}\label{sec:2-divided}
So far, we have not yet described why the nonlinear interaction term in \eqref{eq:omega-decomposition} contributes to synchronization. In this section, we will describe three mechanisms behind this in the next three subsections. In the fourth subsection, we will delineate our main framework for synchronization of \eqref{A-1}, namely Theorem \ref{T3.1}. For convenience, related previous results are provided in Appendix \ref{app:suppt}.

\subsection{Synchronization Mechanism I. Inertial gradient flow formulation of \eqref{A-1} and the {\L}ojasiewicz gradient theorem}\label{subsec:1.6}

The inertial Kuramoto model \eqref{A-1}, or equivalently the second-order model \eqref{B-1}, admits an inertial gradient flow formulation. For this, we define the analytic potential $P = P(\Theta)$ as
\begin{equation} \label{B-5-1}
P(\Theta) \coloneqq -\sum_{k=1}^N \nu_k\theta_k + \frac{\kappa}{2}\sum_{k,l=1}^N \Big(1 - \cos(\theta_k -\theta_l) \Big).
\end{equation}
Then system \eqref{A-1} can be rewritten as an inertial gradient flow system:
\begin{equation} \label{B-6}
m\ddot\Theta + \dot\Theta = -\nabla_{\Theta} P(\Theta).
\end{equation}
We introduce the zero set of the potential force $-\nabla_{\Theta} P(\Theta)$:
\[
\mathcal{S}\coloneqq \{ (\Theta, {\bf 0}) \in \mathbb{R}^{2N} \,:\, \nabla_\Theta P(\Theta) = 0\}.
\]
The classical {\L}ojasiewicz gradient theorem \cite{L2}, which is a consequence of the {\L}ojasiewicz gradient inequality \cite{L1}, states that a bounded solution to a \emph{first-order} gradient flow of a real analytic potential converges asymptotically. This was extended to the generality of inertial gradient flow systems in \cite[Theorem 1.1]{H-J}; it was observed in \cite[Proposition 2.1]{C-L} that this formulation applies to the inertial Kuramoto model \eqref{A-1} via the formulation \eqref{B-6}.

\begin{proposition}[{\cite[Proposition 2.1]{C-L}}, {\cite[Theorem 1.1]{H-J}}, {\cite[Corollary 5.1]{B-B-J}}]\label{P2.1}
Let $(\Theta, \Omega)$ be a global solution to \eqref{B-1} satisfying the following zero-sum condition and a priori uniform bound estimate:
\[ 
\nu_c = 0, \quad  \|\Theta\|_{W^{1,\infty}} \coloneqq \|\Theta \|_{L^\infty(\bbr_+)} + \| \Omega\|_{L^\infty(\bbr_+)} <\infty.
\]
Then, there exists $(\Theta^\infty, {\bf 0}) \in\mathcal{S}$ such that
\begin{align*}
\lim_{t\to\infty } \Big( \|\Theta(t) -\Theta^\infty \|_{\infty} + \|\Omega(t)\|_{\infty}  \Big) =0.
\end{align*}
Moreover, the convergence is at least algebraic: there exist constants $c,C>0$ such that
\[
\|\Theta(t) -\Theta^\infty \|_{\infty}\le Ct^{-c},\quad t\ge 0.
\]
\end{proposition}
\begin{remark} \label{R2.2}
We give several comments on the content of Proposition \ref{P2.1}.
\begin{enumerate}
    \item It follows from Lemma \ref{L2.2} that
\[
 \| \Omega\|_{L^\infty(\bbr_+)} <\infty.
\]
Thus, in order to apply the result of Proposition \ref{P2.1}, it suffices to check the following uniform bound on $\Theta$:
\begin{equation} \label{B-7-1}
\|\Theta \|_{L^\infty(\bbr_+)} < \infty.
\end{equation}
Furthermore, by \eqref{B-3} of Remark \ref{R2.1}, the condition \eqref{B-7-1} is equivalent to the condition
\begin{equation} \label{B-7-2}
\sup_{0 \leq t < \infty} {\mathcal D}(\Theta(t)) < \infty,
\end{equation}
invariant under the Galilean transformation \eqref{eq:galilean}.
Thus, this seemingly weaker condition \eqref{B-7-2} is equivalent to asymptotic phase-locking of the Kuramoto phase flow $\Theta$, i.e., to show asymptotic phase-locking, it is enough to show \eqref{B-7-2}.
\vspace{0.1cm}
    \item In the small inertia regime $m\kappa\le \frac 14$, another equivalent formulation of asymptotic phase-locking is the finiteness of collisions, shown in \cite{C-D-H}. For a simple proof and strengthening of this fact, we refer to Appendix \ref{app:collision} and Theorem \ref{thm:finite-collision}.
\end{enumerate}
\end{remark}
In general, it is impossible to construct a Lyapunov functional for \eqref{A-1} or \eqref{A-2} with the torus $\mathbb{T}^N$ as the configuration space because of an example of system parameters $\mathcal{V}$ and initial data $\Theta^0$, $\Omega^0$ given below, such that the solution does not achieve asymptotic phase-locking no matter the choice of $\kappa$ and $m$. 
\begin{example}\label{ex:nonsync}
Consider parameters for \eqref{A-1} where we have a decomposition
\[
[N]=\bigsqcup_{\alpha=1}^M \mathcal{A}_\alpha,~M\ge 2,~A_\alpha\neq \emptyset,
\]
with respect to which the initial data satisfying
\[
\forall \alpha\in [M], \quad \sum_{i\in \mathcal{A}_\alpha}e^{\mathrm{i}\theta_i^0}=0,
\]
and the intrinsic and initial frequencies satisfying
\[
\forall~\alpha\in [M], ~~\forall~ i,j\in \mathcal{A}_\alpha, \quad \nu_i=\nu_j\eqqcolon\tilde{\nu}_\alpha \mathrm{~and~}\omega_i^0=\omega_j^0\eqqcolon\tilde{\omega}_\alpha^0, \]
and 
\[ \nu_i\neq \nu_j\quad\forall i\in \mathcal{A}_{\alpha_1},~\forall j\in \mathcal{A}_{\alpha_2},\quad \alpha_1\neq\alpha_2\in[M].
\]
Then, the solution to \eqref{A-1} is given as
\[
\theta_i(t)=m\omega_i^0(1-e^{-t/m})+\nu_i(t-m+me^{-t/m})+\theta_i^0,\quad \omega_i(t)=\omega_i^0 e^{-t/m}+\nu_i (1-e^{-t/m}),
\]
since solutions to \eqref{A-1} are unique and the above solution satisfies
\begin{align*}
\begin{aligned}
& \sum_{i=1}^Ne^{\mathrm{i}\theta_i(t)}=\sum_{\alpha\in [M]}e^{\mathrm{i}(m\tilde{\omega}_\alpha^0(1-e^{-t/m})+\tilde{\nu}_\alpha(t-m+me^{-t/m}))}\sum_{i\in \mathcal{A}_\alpha}e^{\mathrm{i}\theta_i^0}=0, \\
&  \sum_{j=1}^N\sin(\theta_j(t)-\theta_i(t))=\Im\left(e^{-\mathrm{i}\theta_i(t)}\sum_{j=1}^Ne^{\mathrm{i}\theta_j(t)}\right)=0, \quad 
m\dot\omega_i(t)+\omega_i(t)=\nu_i.
\end{aligned}
\end{align*}
A special case of this solution was given in \cite[Example 2.2]{H-R}:
\[ N=4,\quad \nu_1 = \nu_2 \neq \nu_3 = \nu_4,  \quad  \theta_1 = \nu_1 t,\quad \theta_2 =\nu_1 t + \pi,\quad \theta_3 = \nu_3 t,\quad \theta_4 = \nu_3 t + \pi.
\]
Recalling the definition \eqref{eq:R^0} of the order parameter $R$, we observe that 
\[ R(\Theta(t))=0 \quad \mbox{for all $t\ge 0$}. \]
It can be shown that these are the only configurations for which $R=0$ for all $t\ge 0$.

Observe that no matter the choice of $\kappa$ or $m$, the above solution does not achieve asymptotic phase-locking. Hence, we cannot expect asymptotic phase-locking for arbitrary initial data even in the large coupling regime. Nevertheless, the condition $R>0$ which is true for Lebesgue-almost every initial data\footnote{Indeed, assume $N\ge 2$, since when $N=1$, the case $R=0$ cannot happen. Let $X$ denote the set of bipolar configurations $\Theta\in \mathbb{R}^N$. Then $X$ is a one-dimensional closed submanifold of $\mathbb{R}^N$. On the open set $O=\mathbb{R}^N\setminus X$, the map $f:O\to \mathbb{C}$, $f(\Theta)=\frac 1N\sum_{i=1}^N e^{{\mathrm i}\theta_i}$, has $0$ as a regular value, so $f^{-1}(0)$ is an $(N-2)$-dimensional submanifold of $O$. Therefore the set of $\Theta$ such that $R(\Theta)=0$ forms a subset of $\mathbb{R}^N$ of Hausdorff dimension at most $\max\{1,N-2\}$.} rules out this pathological initial data. The best result we can hope for is asymptotic phase-locking for \emph{generic} initial data in the large coupling regime.

\end{example}
Despite Example \ref{ex:nonsync}, we can still ask if one can construct \emph{weak Lyapunov functionals} for \eqref{A-1} and \eqref{A-2}, namely functions $f:\mathbb{T}\times \mathbb{R}^N\to\mathbb{R}$ and $g:\mathbb{T}^N\to \mathbb{R}$ such that $\frac{d}{dt}f(\Theta,\dot\Theta)\ge 0$ for any solution $\Theta$ to \eqref{A-1}, and $\frac{d}{dt}g(\Theta)\ge 0$ for any solution $\Theta$ to \eqref{A-2}, with equality happening only at phase-locked states traveling at constant speed, and the configurations of Example \ref{ex:nonsync}. We pose this as a question below.
\begin{conjecture}\label{conj:lyapunov}\,
    \begin{enumerate}
        \item (Weak form) There is a constant $c\ge \frac 12$ such that if $\kappa>c\mathcal{D}(\mathcal{V})$, then \eqref{A-1} and \eqref{A-2} admit weak Lyapunov functionals.
        \vspace{0.1cm}
        \item (Strong form) Denoting the critical coupling strength\footnote{This is the coupling strength above which phase-locked states exist. It is given as a function of $\mathcal{V}$ \cite{V-M}.} as $\kappa_c(\mathcal{V})$, if $\kappa>\kappa_c(\mathcal{V})$, then \eqref{A-1} and \eqref{A-2} admit weak Lyapunov functionals.
    \end{enumerate}
\end{conjecture}

\subsection{Synchronization Mechanism II. Stability of majority clusters}\label{subsec:1.4}
Another viewpoint on the role of the term $\frac 1N\sum_{l=1}^N\sin(\theta_l-\theta_i)$ in synchronization is the following. In $\dot\theta_j-\dot\theta_i$, there is the term
\begin{equation*}
    \frac 1N\sum_{l=1}^N\sin(\theta_l-\theta_j)-\frac 1N\sum_{l=1}^N\sin(\theta_l-\theta_i)=-2\sin\left(\frac{\theta_j-\theta_i}{2}\right)\cdot \frac 1N\sum_{l=1}^N \cos\left(\theta_l-\frac{\theta_i+\theta_j}{2}\right).
\end{equation*}
Assuming without loss of generality that $\theta_i<\theta_j\le \theta_i+\pi$, this will be negative if and only if
\begin{equation}\label{eq:imbalance}
\frac 1N\sum_{l=1}^N \cos\left(\theta_l-\frac{\theta_i+\theta_j}{2}\right)>0,
\end{equation}
contributing to $\theta_i$ and $\theta_j$ being pulled towards each other.

If $\{\theta_i\}_{i=1}^N$ is evenly distributed throughout the circle, the relation \eqref{eq:imbalance} would not likely happen for many pairs of $i$ and $j$, and even if the left-hand side of \eqref{eq:imbalance} were positive, it would be small\footnote{The heuristic is that if $\{\theta_i\}_{i=1}^N$ is evenly distributed throughout the circle, then the order parameter $R$ would be small, but $\left|\frac 1N\sum_{l=1}^N \cos\left(\theta_l-\frac{\theta_i+\theta_j}{2}\right)\right|\le R$ since it is the inner product of the centroid $\frac 1N\sum_l e^{\mathrm{i}\theta_l}$ and the unit vector $e^{\mathrm{i}(\theta_i+\theta_j)/2}$.} and the synchronous effects can be ignored. However, if a majority of the oscillators, say $\{\theta_i\}_{i\in \mathcal{A}}$, with $|\mathcal{A}|\ge \lambda N$ and $\lambda >\frac 12$, are concentrated in a sufficiently small arc, say of length $\ell\in (0,\pi)$, this can happen for $i,j\in \mathcal{A}$. In particular, when
\[
i=\operatorname{argmin}_{l\in \mathcal{A}}\theta_l,\quad j=\operatorname{argmax}_{l\in \mathcal{A}}\theta_l,
\]
then \eqref{eq:imbalance} will be guaranteed when
\[
\frac 1N\sum_{l\in \mathcal{A}} \cos\left(\theta_l-\frac{\theta_i+\theta_j}{2}\right)+\frac 1N\sum_{l\notin \mathcal{A}} \cos\left(\theta_l-\frac{\theta_i+\theta_j}{2}\right)\ge \lambda \cos\frac\ell 2-(1-\lambda)>0,
\]
or
\begin{equation}\label{eq:arclen-heuristic}
\ell \in\left(0,2\cos^{-1} \left( \frac{1}{\lambda} - 1 \right)\right)
\end{equation}
Thus, we introduce the following definition.
\begin{definition}\label{def:majority-cluster}
    Fix $\lambda\in (0,1]$ and $\ell\in (0,2\pi)$. Given a vector $\Theta\in \mathbb{R}^N$ and a subset $\mathcal{A}\subset [N]$, we say that $\Theta_\mathcal{A}$ is a \emph{$\lambda$-cluster of arc length $\le \ell$} if $|\mathcal{A}|\ge \lambda N$ and, up to $2\pi$-translations, we have
    \[
    \mathcal{D}\left(\Theta_\mathcal{A}\right)\le \ell.
    \]
\end{definition}
The above heuristic says roughly that whenever $\lambda\in (\frac 12,1]$ and $\ell\in (0,\pi)$ satisfy \eqref{eq:arclen-heuristic}, a $\lambda$-cluster of arclength $\le \ell$ should be stable, possibly under additional assumptions. Theorem \ref{L4.4} given later in Section \ref{sec:3} confirms this heuristic.

\subsection{Synchronization Mechanism III. Quasi-monotonicity of the order parameter}\label{subsec:1.5}

Recall the definition of the order parameter \eqref{eq:R^0}. More generally, given a phase vector $\Theta\in \mathbb{R}^N$, we introduce the \emph{amplitude order parameter} $R=R(\Theta)$ and \emph{phase order parameter} $\phi=\phi(\Theta)$ as
\begin{equation}\label{eq:Rphi}
Re^{\mathrm{i}\phi}\coloneqq \frac 1N \sum_{j=1}^N e^{\mathrm{i}\theta_j}.
\end{equation}
Note that $R$ and $\phi$ are well-defined as long as the right-hand side of \eqref{eq:Rphi} is not zero, and when we simply say \emph{order parameter}, we mean the amplitude order parameter $R$.

Given a solution $\Theta(t)$ to \eqref{A-1}, we denote $R(t)=R(\Theta(t))$ and $\phi(t)=\phi(\Theta(t))$. Then $R$ is a well-defined continuous real-valued function of $t$. On a time interval where $R$ is positive, $R$ is a smooth function of $t$, and there is a smooth section of $\phi$ that is unique up to modulo $2\pi$ shifts.

In some sense, the functional $R$ measures the overall degree of synchronization, with $R$ being close to $1$ or $0$ signifying synchrony or asynchrony, respectively. The functional $\phi$ can be thought of as a representative phase value for the position of a typical particle; it is a strong representative when $R$ is close to $1$ and a weak representative when $R$ is close to $0$.

We divide both sides of \eqref{eq:Rphi} by $e^{\mathrm{i}\theta_i}$ and compare the real and imaginary parts to find 
\begin{equation} \label{B-10}
R\sin(\phi -\theta_i) = \frac{1}{N}\sum_{j=1}^N \sin(\theta_j - \theta_i), \quad  R\cos(\phi -\theta_i) = \frac{1}{N}\sum_{j=1}^N \cos(\theta_j - \theta_i).
\end{equation}
On the other hand, we again divide both sides of \eqref{eq:Rphi} by $e^{\mathrm{i}\phi}$ and compare the real and imaginary parts to obtain
\begin{align}\label{B-12}
R = \frac{1}{N}\sum_{j=1}^N \cos(\theta_j -\phi),\quad 0=\frac{1}{N}\sum_{j=1}^N\sin(\theta_j -\phi).
\end{align}
Also, we have
\begin{align}\label{B-12-1}
R^2 \overset{\eqref{B-12}\mbox{$_1$}}{=} \frac{1}{N}\sum_{j=1}^N R\cos(\theta_j - \phi)
\overset{\eqref{B-10}\mbox{$_2$}}{=}\frac{1}{N^2}\sum_{i,j=1}^N \cos(\theta_j -\theta_i).
\end{align}

We have the following description of the dynamics of $\Theta$ and $R$.
\begin{lemma} \label{L2.3}  Let $(\Theta, \Omega)$ be a global solution to \eqref{B-1}. Then, the following assertions hold.
\begin{enumerate}
\item
The order parameter $R$ satisfies 
\begin{align}\label{B-13}
\dot R = -\frac{1}{N}\sum_{j=1}^N \sin(\theta_j - \phi) \dot \theta_j.
\end{align}
\item
The inertial Kuramoto model can be rewritten using the order parameters $(R, \phi)$ as follows:
\[
m {\ddot \theta}_i + \dot\theta_i = \nu_i - \kappa R \sin(\theta_i - \phi), \quad i \in [N].
\]
\end{enumerate}
\end{lemma}
\begin{proof}
\noindent (i)~We differentiate both sides of \eqref{B-12}$_1$ with respect to $t$ to get the time derivative of $R$:
\[
\dot R =- \frac{1}{N}\sum_{j=1}^N \sin(\theta_j - \phi)(\dot\theta_j -\dot\phi) \overset{\eqref{B-12}\mbox{$_2$}}{=} -\frac{1}{N}\sum_{j=1}^N \sin(\theta_j - \phi) \dot \theta_j.
\]
\vspace{0.2cm}

\noindent (ii)~We use $\eqref{B-10}_1$ to rewrite \eqref{A-1} in mean-field form. 
\end{proof}

From Lemma \ref{L2.3}, we have
\[
\dot R = \frac{\kappa R}{N}\sum_{j=1}^N \sin^2(\theta_j - \phi)+\frac{1}{N}\sum_{j=1}^N \sin(\theta_j - \phi)(m\ddot\theta_j-\nu_j),
\]
and so we can see that the sinusoidal coupling term of \eqref{A-1} tends to increase $R$, while the effect of the natural frequencies and the inertia are uncertain. Another way to see this is that by \eqref{B-12-1}, the potential can be written as
\begin{align*}
P(\Theta) = -\sum_{k=1}^N \nu_k\theta_k + \frac{\kappa N^2}{2} \left(1 - R^2 \right).
\end{align*}
We may imagine that $R$ is a functional that is trying to increase along the dynamics of the nonlinear couplings, while the ``impurities'' caused by $\mathcal{V}$ not being identical over the oscillators and the time-delay effect of the inertia may disrupt the monotonic behavior of $R$. Motivated by the formula for $\dot{R}$, we introduce the mean-square deviation
\begin{equation}\label{eq:Delta}
\Delta\coloneqq\Delta(\Theta)\coloneqq \frac{1}{N} \sum_{k=1}^N \sin^2(\theta_k -\phi),
\end{equation}
which will be used in quantifing the quasi-monotonicity of $R$. Note that the functional $\Delta$ measures the closeness to either a completely synchronized state or a bi-polar configuration, i.e., states where $\theta_i-\theta_j\in \pi \mathbb{Z}$ for all $i,j\in [N]$.

In the case of the first-order model \eqref{A-2} with the same identical natural frequencies $\nu_i=\nu_j$ for all $i,j\in [N]$, we have the identity
\[
\dot{R}=\kappa R\Delta=\frac 1{\kappa R N}\sum_{i=1}^N\dot\theta_i^2\ge 0.
\]
This immediately shows that $R$ is monotonically increasing in $t$, that $R^2$ is a Lyapunov functional, and that it can be used to show that any solution to \eqref{A-2} with identical frequencies must converge. In the second-order model \eqref{A-1} with identical natural frequencies $\nu_i=\nu_j$, there is an energy dissipation formula\footnote{More precisely, \cite{C-H-M} uses the expression $\frac 1N\sum_{i=1}^N\omega_i^2$ instead of $\operatorname{Var}(\Omega)$. Our substitution is harmless by the Galilean transformation \eqref{eq:galilean}.} \cite[Proposition 4.1]{C-H-M}
\begin{equation}\label{eq:energy-dissipation}
\frac{d}{dt}\left(\frac{\kappa(1-R^2)}{2}+\frac m2\operatorname{Var}(\Omega)\right)=-\operatorname{Var}(\Omega).
\end{equation}
See Appendix \ref{app:suppt-2} for the resulting convergence statements.

However, in the nonidentical case, $R(t)^2$ fails to be a Lyapunov functional, but we will show that it can serve as a proxy for a Lyapunov functional of the model \eqref{A-1}. 
Ideally, we would like to have a Lyapunov functional for \eqref{A-1} defined on the torus in the nonidentical case, i.e., we would like to answer Conjecture \ref{conj:lyapunov}. The Lyapunov functional may perhaps be given as a perturbation of $\frac{\kappa(1-R^2)}{2}+\frac m2\operatorname{Var}(\Omega)$. In the presence of nonidentical frequencies and the inertia term, $\dot{R}$ might decrease for certain configurations $(\Theta,\Omega)$. Nevertheless, we will obtain a ``quasi-monotonicity formula'' as follows. We will later use this to control the dynamics of \eqref{A-1}.
\begin{lemma}[Quasi-monotonicity formula for $R$]\label{L2.4}
Let $(\Theta,\Omega)$ be a global solution to \eqref{B-1}. Then, the order parameter $R$ defined as in \eqref{eq:Rphi} satisfies the following differential inequality:
\begin{align} \label{B-14-0}
\begin{aligned}
\dot R(t) \ge& \kappa R(t)\Delta(t)\left(1-e^{-t/m}\right)- \sqrt{\Delta(t)} \mathcal{D}(\Omega^0)e^{-t/m}/2-\sqrt{\Delta(t)}\mathcal{D}(\mathcal{V})(1-e^{-t/m})/2 \\
&-\kappa\sqrt{\Delta(t)} \left(1-e^{-t/m}\right)\left(\mathcal{D}(\Omega^0)te^{-t/m} + m(\mathcal{D}(\mathcal{V})+2\kappa)\left(1-e^{-t/m}\right)^2\right),
\end{aligned}
\end{align}
\color{black}
for $t \ge 0$. Moreover, for $\eta>0$, we have
\begin{equation}\label{B-14-000}
R(t)\ge R^0-\zeta(m,\kappa,\mathcal{V},\Omega^0,\eta),\quad t\le \eta m,
\end{equation}
and
\begin{equation}\label{B-14-00}
\dot R(t) \geq \kappa\sqrt{\Delta(t)}(1-e^{-t/m}) \left( R(t) \sqrt{\Delta(t)} - \xi(m,\kappa,\mathcal{V},\Omega^0,\eta) \right),\quad t \geq \eta m,
\end{equation}
where $\zeta(m,\kappa,\mathcal{V},\Omega^0,\eta)$ and $\xi(m,\kappa,\mathcal{V},\Omega^0,\eta)$ are the dimensionless quantities defined as follows:
\begin{equation}\label{B-14-0-1}
\begin{cases}
\displaystyle \zeta(m,\kappa,\mathcal{V},\Omega^0,\eta)\coloneqq\frac{m(1-e^{-\eta})}{2}\left[\mathcal{D}(\Omega^0)+\mathcal{D}(\mathcal{V})\eta\right]
+m^2\kappa\left(1-e^{-\eta}\right)^3\left[\frac 34  \mathcal{D}(\Omega^0)+  (\mathcal{D}(\mathcal{V})+2\kappa)\eta \right],\\
\displaystyle \xi(m,\kappa,\mathcal{V},\Omega^0,\eta)\coloneqq \left(\mathcal{D}(\mathcal{V})+2\kappa\right)m+\mathcal{D}(\Omega^0)m\max\{1,\eta\}e^{-\max\{1,\eta\}}+\frac{\mathcal{D}(\mathcal{V})}{2\kappa}+\frac{\mathcal{D}(\Omega^0)}{2\kappa}\frac{e^{-\eta}}{1-e^{-\eta}}.
\end{cases}
\end{equation}
\end{lemma}
\begin{proof}
We begin by deriving the differential inequality \eqref{B-14-0} for $R$.
We begin by noting that by the Cauchy-Schwartz inequality,
\begin{equation*}
\displaystyle \frac{1}{N} \sum_{k=1}^N |\sin(\theta_k -\phi)| \leq \sqrt{\frac 1N\sum_{k=1}^N |\sin(\theta_k - \phi)|^2} \stackrel{\eqref{eq:Delta}}{=} \sqrt{\Delta},
\end{equation*}
and by Lemma \ref{L:approxaut} and $\eqref{B-10}_1$,
\begin{align}\label{Ap-1-1}
\begin{aligned}
&\left|\dot{\theta}_i(t)-\omega_i^0 e^{-t/m}-\nu_i\left(1-e^{-t/m}\right)-\kappa R(t)\sin\left(\phi(t)-\theta_i(t)\right)\left(1-e^{-t/m}\right)\right|\\
& \hspace{1cm} \le \kappa \left(1-e^{-t/m}\right)\left(\mathcal{D}(\Omega^0)te^{-t/m} +m(\mathcal{D}(\mathcal{V})+2\kappa)\left(1-e^{-t/m}\right)^2\right).
\end{aligned}
\end{align}
We estimate $\dot R$ in the following manner:
\begin{align}
\begin{aligned} \label{Ap-2}
&\dot R(t)~ \stackrel{\mathclap{\eqref{B-13}}}{=} -\frac{1}{N}\sum_{j=1}^N \sin(\theta_j(t) - \phi(t)) \dot \theta_j(t) ~\stackrel{\mathclap{\eqref{B-12}_2}}{=}-\frac{1}{N}\sum_{j=1}^N \sin(\theta_j(t) - \phi(t)) \left(\dot \theta_j(t)-\omega_c(t)\right)\\
&  \hspace{0.7cm} \stackrel{\mathclap{ \substack{ \mathrm{Lem ~}\ref{L2.1}, \\ \eqref{Ap-1-1} }}}{=} - \frac{1}{N}\sum_{j=1}^N \sin(\theta_j(t) - \phi(t)) \left(\omega^0_j-\omega_c^0\right)e^{-t/m}-\frac{1}{N}\sum_{j=1}^N \sin(\theta_j(t) - \phi(t))\left(\nu_j-\nu_c\right)(1-e^{-t/m}) \\
&  \hspace{0.7cm}  +\kappa R(t)\cdot \frac 1N\sum_{j=1}^N \sin^2(\theta_j(t)-\phi(t))\left(1-e^{-t/m}\right)\\
& \hspace{0.7cm} -\frac{1}{N}\sum_{j=1}^N\left|\sin(\theta_j(t)-\phi(t))\right|\cdot \kappa \left(1-e^{-t/m}\right)  \left(\mathcal{D}(\Omega^0)te^{-t/m} + m(\mathcal{D}(\mathcal{V})+2\kappa)\left(1-e^{-t/m}\right)^2\right)\\
& \hspace{0.7cm}  {\ge}   - \sqrt{\Delta(t)} \sqrt{\operatorname{Var}(\Omega^0)}e^{-t/m}-\sqrt{\Delta(t)}\sqrt{\operatorname{Var}(\mathcal{V})}(1-e^{-t/m}) +\kappa R(t)\Delta(t)\left(1-e^{-t/m}\right)\\
& \hspace{0.7cm}  -\kappa\sqrt{\Delta(t)} \left(1-e^{-t/m}\right)\left(\mathcal{D}(\Omega^0)te^{-t/m} + m(\mathcal{D}(\mathcal{V})+2\kappa)\left(1-e^{-t/m}\right)^2\right)\\
& \hspace{0.7cm}  \ge   - \sqrt{\Delta(t)} \mathcal{D}(\Omega^0)e^{-t/m}/2-\sqrt{\Delta(t)}\mathcal{D}(\mathcal{V})(1-e^{-t/m})/2 +\kappa R(t)\Delta(t)\left(1-e^{-t/m}\right)\\
& \hspace{0.7cm}  -\kappa\sqrt{\Delta(t)} \left(1-e^{-t/m}\right)\left(\mathcal{D}(\Omega^0)te^{-t/m} + m(\mathcal{D}(\mathcal{V})+2\kappa)\left(1-e^{-t/m}\right)^2\right),
\end{aligned}
\end{align}
where in the penultimate inequality, we used the Cauchy-Schwarz inequality, and in the last inequality, we used 
\[ \sqrt{\operatorname{Var}(X)}\le \mathcal{D}(X)/2 \quad \mbox{for $X\in \mathbb{R}^N$}. \]
Next, we consider two cases. \newline

\noindent $\bullet$~Case A (Dynamics of $R$ after initial-layer): For $t \geq \eta m$, we have
\begin{align*}
    \dot R(t)\ge & \kappa \sqrt{\Delta(t)}\left(1-e^{-t/m}\right)\left(R(t)\sqrt{\Delta(t)} - \frac{\mathcal{D}(\Omega^0)}{2\kappa}\cdot\frac{e^{-t/m}}{1-e^{-t/m}}-\frac{\mathcal{D}(\mathcal{V})}{2\kappa}  \right)\\
& -\kappa\sqrt{\Delta(t)} \left(1-e^{-t/m}\right)\left(m\mathcal{D}(\Omega^0)\cdot \frac tm e^{-t/m} + m(\mathcal{D}(\mathcal{V})+2\kappa)\left(1-e^{-t/m}\right)^2\right)\\
\ge & \kappa \sqrt{\Delta(t)}\left(1-e^{-t/m}\right)\left(R(t)\sqrt{\Delta(t)} - \frac{\mathcal{D}(\Omega^0)}{2\kappa}\cdot\frac{e^{-\eta}}{1-e^{-\eta}}-\frac{\mathcal{D}(\mathcal{V})}{2\kappa}  \right)\\
& -\kappa\sqrt{\Delta(t)} \left(1-e^{-t/m}\right)\left(m\mathcal{D}(\Omega^0)\cdot \max\{1,\eta\} e^{-\max\{1,\eta\}} + m(\mathcal{D}(\mathcal{V})+2\kappa)\right)\\
\ge & \kappa \sqrt{\Delta(t)}\left(1-e^{-t/m}\right)\left(R(t)\sqrt{\Delta(t)}-\xi(\eta)\right).
\end{align*}

\vspace{0.2cm}

\noindent $\bullet$~Case B (Dynamics of $R$ in the initial-layer ): For $t \in [0, \eta m]$, we use \eqref{Ap-2} and $\Delta\le 1$ to find
\begin{align*}
R(t) \ge& R^0 + \underbrace{\int_0^t \kappa (1-e^{-\frac{s}{m}}) R(s) \Delta(s) ds}_{\geq 0} 
-  \int_0^{\eta m}\frac{\mathcal{D}(\Omega^0)}{2}e^{-s/m}ds-\int_0^{\eta m}\frac{\mathcal{D}(\mathcal{V})}{2}(1-e^{-s/m})ds \\
&-\int_0^{\eta m}\kappa \mathcal{D}(\Omega^0)se^{-s/m}\left(1-e^{-s/m}\right)ds - \int_0^{\eta m}m\kappa(\mathcal{D}(\mathcal{V})+2\kappa)\left(1-e^{-s/m}\right)^3ds\\
&\ge 
 R^0+0 -  \frac{m\mathcal{D}(\Omega^0)}{2}\left(1-e^{-\eta}\right)-\frac{m\mathcal{D}(\mathcal{V})}{2}\left(\eta-1+e^{-\eta}\right) \\
& -m^2\kappa \mathcal{D}(\Omega^0)\left(\frac 34 -e^{-\eta}-\eta e^{-\eta}+\frac {\eta e^{-2\eta}}{2}+\frac {e^{-2\eta}}{4}\right) \\
& - m^2\kappa (\mathcal{D}(\mathcal{V})+2\kappa)\left(-\frac{11}6+\eta+3e^{-\eta}-\frac {3e^{-2\eta}}{2}+\frac  {e^{-3\eta}}3\right)\\
&\ge 
 R^0-  \left(1-e^{-\eta}\right)\left[\frac{m\mathcal{D}(\Omega^0)}{2}+\frac{m\mathcal{D}(\mathcal{V})\eta}{2}\right]
-\left(1-e^{-\eta}\right)^3\left[\frac{3}{4} m^2\kappa \mathcal{D}(\Omega^0) + m^2\kappa (\mathcal{D}(\mathcal{V})+2\kappa)\eta \right]\\
& \ge R^0 - \zeta(\eta),
\end{align*}
where in the penultimate inequality we used the following calculus inequalities:
\begin{align*}
&x-1+e^{-x}\le x\left(1-e^{-x}\right),~ \frac 34-e^{-x}-xe^{-x}+\frac x2 e^{-2x}+\frac 14 e^{-2x}\le \frac 34 \left(1-e^{-x}\right)^3,\\
&-\frac {11}6+x+3e^{-x}-\frac 32 e^{-2x}+\frac 13 e^{-3x}\le x\left(1-e^{-x}\right)^3,\quad x\ge 0.
\end{align*}

The inequality \eqref{B-14-00} follows from \eqref{B-14-0}.
\end{proof}
\begin{remark}
Below, we provide several comments of Lemma \ref{L2.4}.
\begin{enumerate}
\item
In order to describe the relaxation dynamics at later times, it is convenient to have a positive lower bound on $R$, as $R$ is the strength of the nonlinear interaction $\frac 1N \sum_j \sin(\theta_j-\theta_i)$ in \eqref{A-1}. In fact, as can be seen above, the nonlinear interactions tend to increase $R$, so we have a positive feedback: a larger $R$ leads to greater nonlinear interactions which in turn enlarges $R$. Unfortunately, given uniformly random initial data on $\mathbb{S}^1$, the expected value of $R^2$ is $1/N$, as can be seen from  \eqref{B-12-1}. In other words, most initial data have $R$ on the order of $1/\sqrt{N}$, which is extremely small in the large particle limit $N\to\infty$. Perhaps the main difficulty of the complete synchronization problem for \eqref{A-1} is to prove or disprove how synchrony ``erupts'' from the almost disordered state $R\approx 0$, producing a value of $R$ above some positive universal constant.
\vspace{0.1cm}
\item
For notational simplicity, we suppress dependence of $\zeta$ and $\xi$ on $\kappa, m, {\mathcal V}$ and $\Omega^0$:
\[
\zeta(\eta)\coloneqq  \zeta(m, \kappa, {\mathcal V}, \Omega^0, \eta),\quad \xi(\eta)\coloneqq  \xi(m, \kappa, {\mathcal V}, \Omega^0, \eta).
\]
Note that $\zeta(\eta)$ and $\xi(\eta)$ are independent of $\Theta^0$, and 
\[
\lim_{\eta\rightarrow 0}\zeta(\eta)=0,\quad\lim_{\eta\rightarrow \infty}\zeta(\eta)=\infty,\quad\lim_{\eta\rightarrow 0}\xi(\eta)=\infty,
\]
and
\[
\xi(\infty)\coloneqq \xi(m,\kappa,\mathcal{V},\Omega^0,\infty)\coloneqq\lim_{\eta\rightarrow \infty}\xi(\eta)=m\mathcal{D}(\mathcal{V})+2m\kappa+\frac{\mathcal{D}(\mathcal{V})}{2\kappa}.
\]
The framework \eqref{C-1} says that $\zeta(\eta)$ and $\xi(\eta)$ are small.
\vspace{0.1cm}
\item
The dominant terms in \eqref{B-14-0} as $t\to \infty$ are
\[
\kappa R(t)\Delta(t)-\left(\mathcal{D}(\mathcal{V})/2+m\kappa \mathcal{D}(\mathcal{V})+2m\kappa^2\right)\sqrt{\Delta(t)}.
\]
The first term $\kappa R(t)\Delta(t)$ is due to the sinusoidal couplings and thus act to increase $R(t)$, while the second term
\[
\left(\mathcal{D}(\mathcal{V})/2+m\kappa \mathcal{D}(\mathcal{V})+2m\kappa^2\right)\sqrt{\Delta(t)},
\]
is the maximal rate at which the linear terms $\nu_i$ and the inertial term $m\ddot\theta_i$ may conspire to decrease the order parameter. This heuristic is true only for large $t$, in which case we have \eqref{B-14-00}, and for small $t$ we only have the crude bound of \eqref{B-14-000}. Our strategy in Section \ref{sec:5} is as follows: first, for small time $t\le \eta m$, we can control the amount of fluctuations of $R(t)$, namely $\zeta(\eta)$, by requiring $m$ to be small; this shields us from the possible initial ``adversarial attack'' to diminish the amount of synchronization. This is the idea of Lemma \ref{L4.1}.

Next, for a large time $t\ge \eta m$, we have two scenarios: either $\dot R\ge 0$, which is good since this means the system is synchronizing, or we have $\dot{R}<0$. By the above heuristic, this roughly means
\begin{align*}
\Delta(t)\le \left(\mathcal{D}(\mathcal{V})/2\kappa +m \mathcal{D}(\mathcal{V})+2m\kappa\right)^2/R(t)^2.
\end{align*}
 Thus, having a lower bound on $R(t)$, we can make $\Delta(t)$ as small as we wish by making $\mathcal{D}(\mathcal{V})/\kappa$ and $m\kappa$ small (i.e., under the framework of \eqref{eq:framework-small-param}). This is quantified in Lemma \ref{L4.2}: there exists a time at which we have a lower bound on $R\ge c>0$ and an upper bound on $\Delta$. Roughly, this means that $\Theta$ is close to a bipolar state with $\frac{(1+c)N}{2}$ oscillators at $\phi$ and $\frac{(1-c)N}{2}$ oscillators at $\phi+\pi$, so that $\Theta$ has roughly $\frac{(1+c)N}{2}$ oscillators (a majority) concentrated around $\phi$; this argument is quantified in Lemma \ref{L4.3}. This allows us to use Synchronization Mechanism II, the stability of majority clusters, discussed in Subsection \ref{subsec:1.4}.
\end{enumerate}
\end{remark}
In the rest of subsection, we provide some simulations of the dynamics of the order parameter and the mean squared deviation while varying the dimensionless quantities\footnote{These are the quantities invariant under the time dilatation symmetry \eqref{eq:dilatation}.} $m\kappa$, $\mathcal{D}(\mathcal{V})/\kappa$, and $\mathcal{D}(\Omega^0)/\kappa$.

We use the fourth-order Runge-Kutta method with a time step of $0.01$ and set $N=50$. For each numeric simulation, we fix some prescribed value for $m$ and $\kappa$, and we take initial phase data uniformly distributed on the interval $[0,2\pi]$, initial frequency data uniformly distributed on the interval $\left[-\frac{\mathcal{D}(\Omega^0)}{2},\frac{\mathcal{D}(\Omega^0)}{2}\right]$, and intrinsic frequency data uniformly distributed on the interval $\left[-\frac{\mathcal{D}(\mathcal{V})}{2},\frac{\mathcal{D}(\mathcal{V})}{2}\right]$.

In the first experiment, whose results are displayed in Figure \ref{Fig1-1}, we take $m\kappa$ and $\mathcal{D}(\Omega^0)/\kappa$ constant and vary $\mathcal{D}(\mathcal{V})/\kappa$. When $\kappa=1$ and $m=1$ (blue line), $\mathcal{D}(\mathcal{V})/\kappa$ is too large and the system fails to achieve synchronization, signified by the fact that not only do the order parameter $R$ and the mean squared deviation $\Delta$ fail to converge to a fixed value, but also $R$ frequents a neighborhood of zero while $\Delta$ fluctuates near high values. But, as soon as $\mathcal{D}(\mathcal{V})/\kappa \le \frac 12$, not only does the solution exhibit asymptotic phase-locking but also both $R$ and $\Delta$ converge to values close to 1 and 0, respectively. This suggests that the smallness of $\mathcal{D}(\mathcal{V})/\kappa$ is a favorable environment for asymptotic phase-locking.

	\begin{figure}[htbp]
	\centering
		\begin{subfigure}[b]{0.47\textwidth}
		\includegraphics[width=\textwidth]{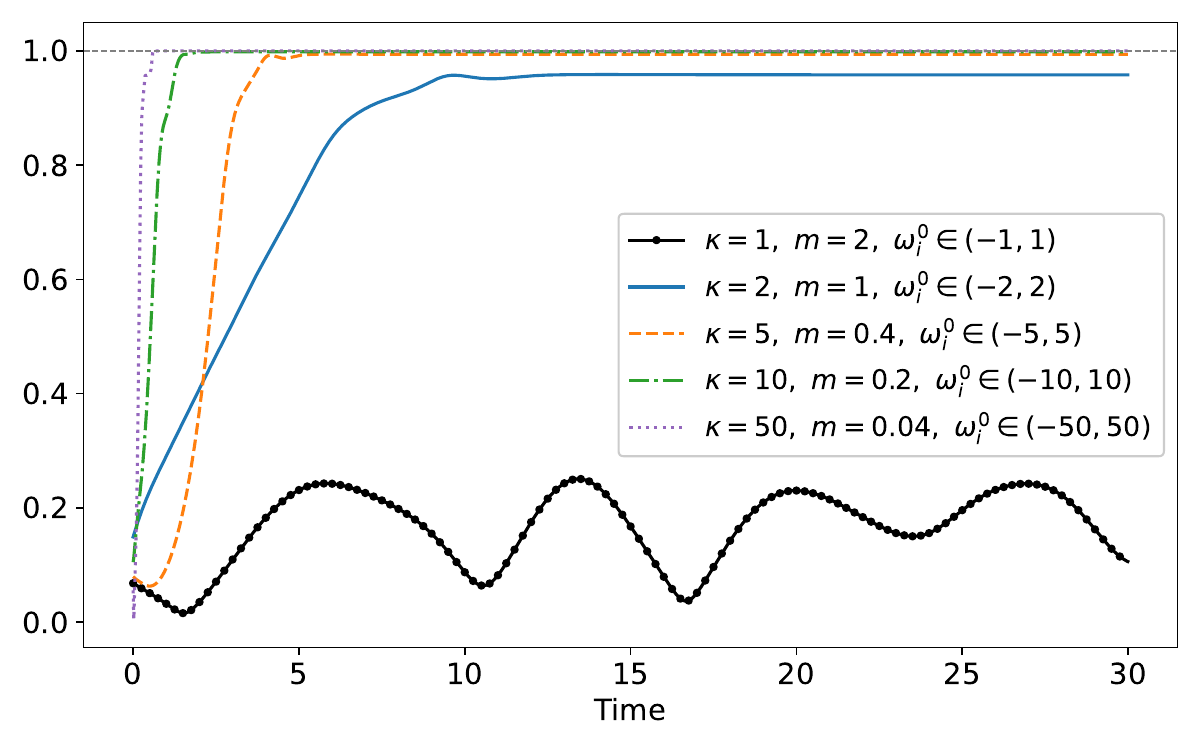}
		\caption{Temporal evolution of $R$.}
	\end{subfigure}
	\vspace{0.1cm}
	\begin{subfigure}[b]{0.47\textwidth}
		\includegraphics[width=\textwidth]{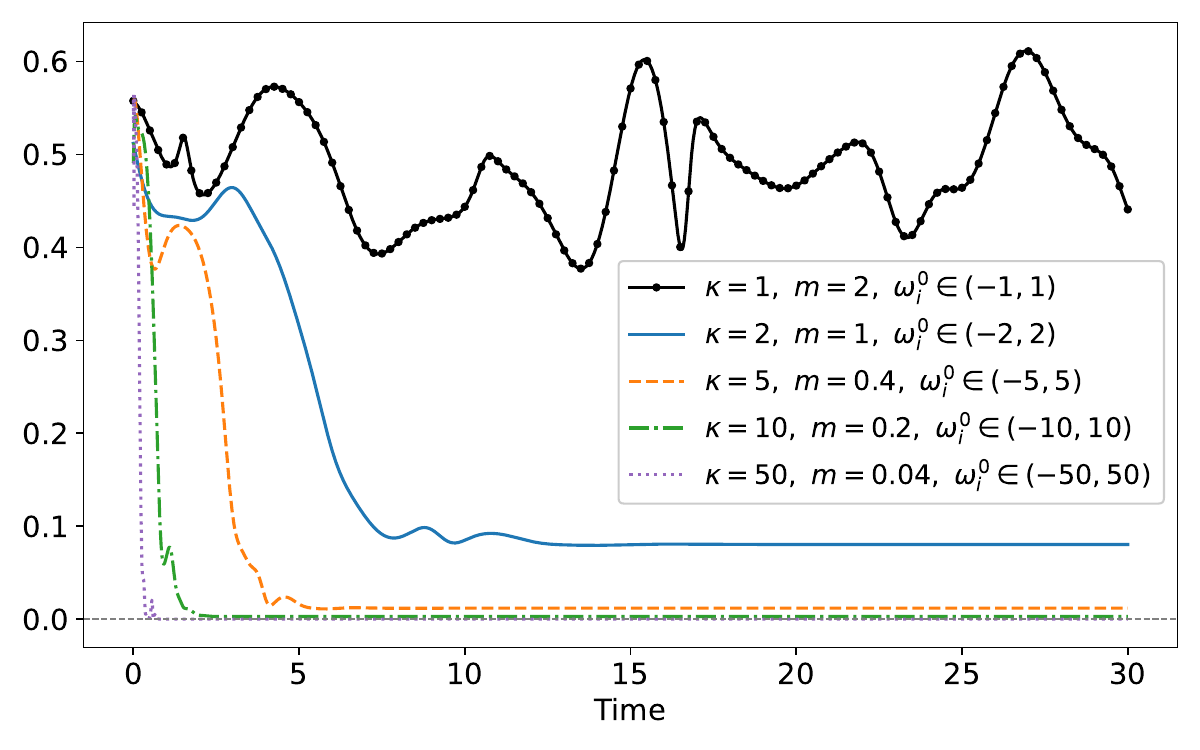}
		\caption{Temporal evolution of $\Delta$.}
	\end{subfigure}
	\caption{(Color online) Simulations of $R$ and $\Delta$ varying ${\mathcal D}({\mathcal V})/\kappa$ and ${\mathcal D}(\Omega^0)$. We used a single set of intrinsic frequency data uniformly distributed over the interval $[-1,1]$ for all simulations. }
	\label{Fig1-1}
	\end{figure}

When $\mathcal{D}(\mathcal{V})/\kappa \le \frac 12$, the limiting order parameter is close to $1$. Thus, it appears that the limiting configuration, which is a phase-locked state, is contained in a quarter circle. However, when $\mathcal{D}(\mathcal{V})/\kappa<1$, there is a unique phase-locked state for \eqref{A-1} and \eqref{A-2} in the quarter-circle; see \cite{C-H-J-K} and Remark \ref{rem:pls} (3). This phase-locked state may be computed from the equations
\[
R=\frac 1N\sum_{i=1}^N\cos(\theta_i-\phi),\quad \sin(\theta_i-\phi)=\frac{\nu_i-\nu_c}{\kappa R}, \quad \forall i\in [N].
\]
From the zeroth order approximation $R=1+o(1)$ (in the sense as $\mathcal{D}(\mathcal{V})/\kappa\to 0$), we have 
\[ \sin(\theta_i-\phi)=(1+o(1))\frac{\nu_i-\nu_c}{\kappa}, \]
and
\[
R =\frac 1N\sum_{i=1}^N\sqrt{1-(1+o(1))\frac{(\nu_i-\nu_c)^2}{\kappa^2}}=\frac 1N\sum_{i=1}^N \left(1-\left(\frac 12+o(1)\right)\frac{(\nu_i-\nu_c)^2}{\kappa^2}\right) =1-\left(\frac 12+o(1)\right)\frac{\operatorname{Var}(\mathcal{V})}{\kappa^2}.
\]
Indeed, the values of $\frac{1-R(30)}{(\operatorname{Var}(\mathcal{V})/\kappa)^2}$ from the second to the fifth experiments in Figure \ref{Fig1-1} are 0.5676, 0.5089, 0.5022, 0.5001, respectively, and are close to 0.5.

What about the other factors? In Figure \ref{Fig1-2}, we keep $\mathcal{D}(\mathcal{V})/\kappa$ and $\mathcal{D}(\Omega^0)/\kappa$ constant and vary $m\kappa$. In this case, the same conclusion holds, i.e., the order parameter $R$ converges to a value close to $1$ while the mean-squared deviation $\Delta$ converges to a value close to $0$; the only difference is the time scale at which this happens, which is multiplicatively delayed proportional to $m$.\footnote{This is the scale of the time delayed interaction.} This suggests that as long as $\mathcal{D}(\mathcal{V})/\kappa$ is small, variations in $m\kappa$ do not present material differences to the asymptotic dynamics.

	\begin{figure}[htbp]
	\centering
		\begin{subfigure}[b]{0.47\textwidth}
		\includegraphics[width=\textwidth]{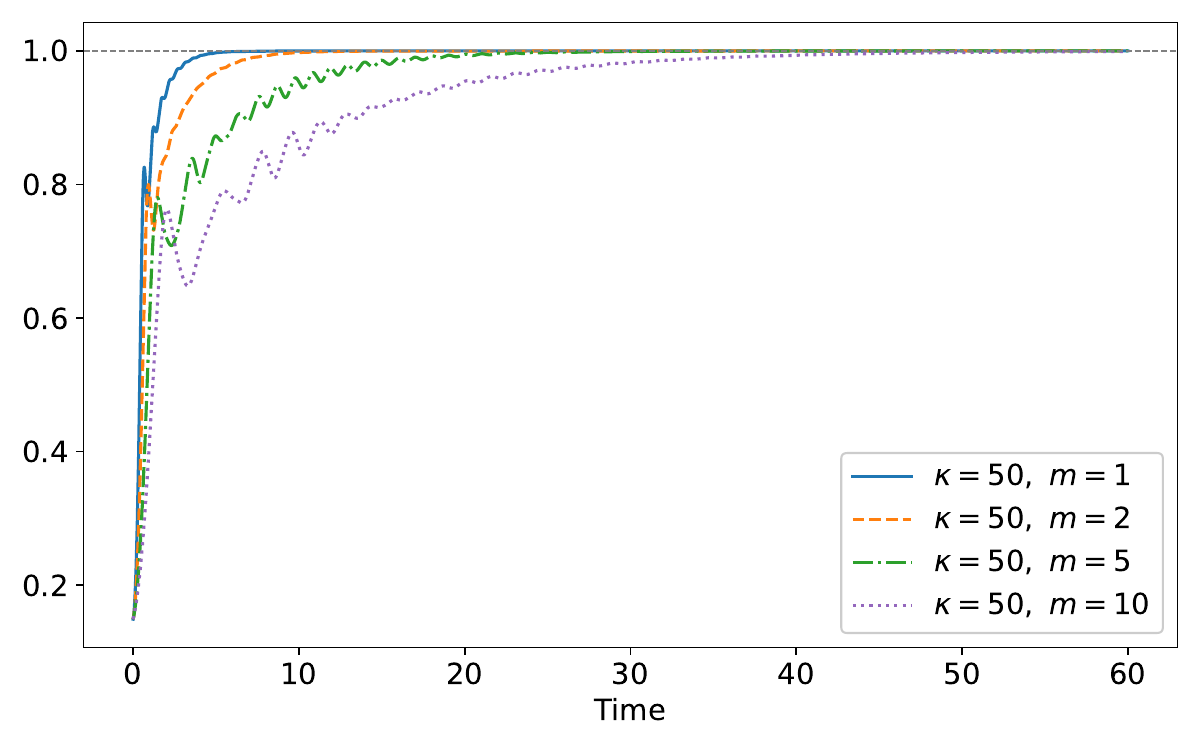}
		\caption{Temporal evolution of $R$.}
	\end{subfigure}
	\vspace{0.1cm}
	\begin{subfigure}[b]{0.47\textwidth}
		\includegraphics[width=\textwidth]{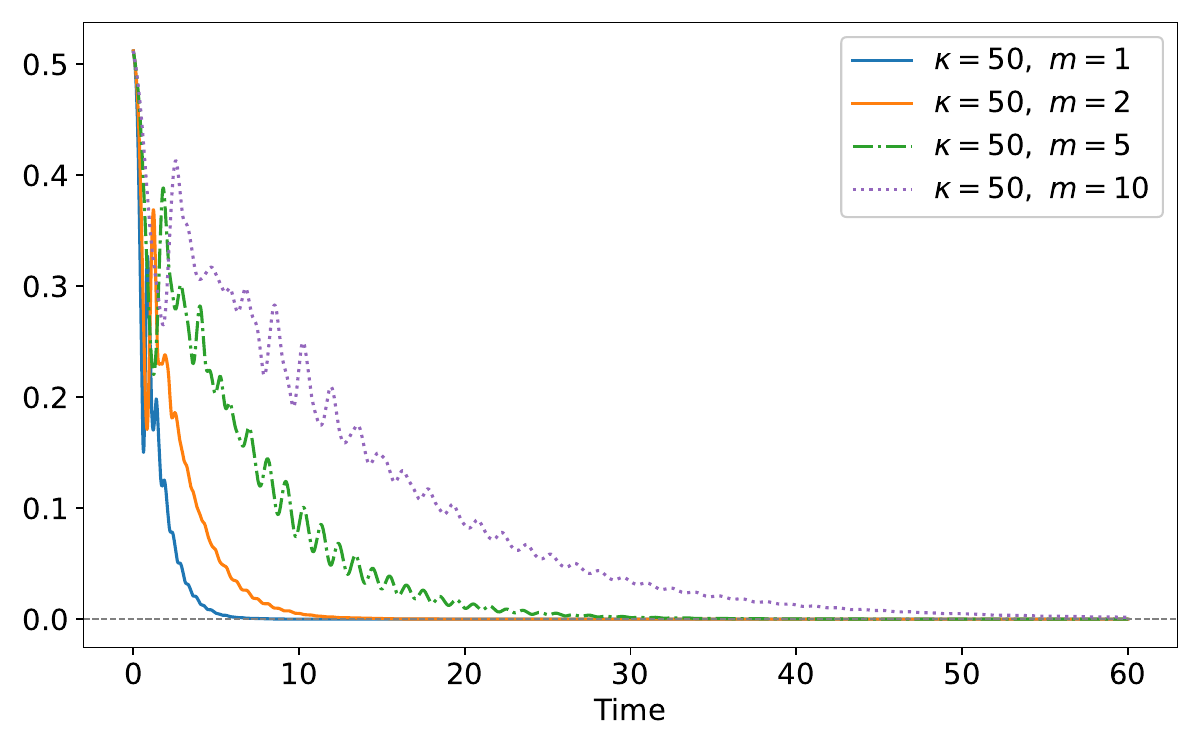}
		\caption{Temporal evolution of $\Delta$.}
	\end{subfigure}
	\caption{(Color online) Simulations of $R$ and $\Delta$ varying $m\kappa$. We used a single set of initial frequency data and natural frequencies, uniformly distributed over the interval $[-1,1]$, for all simulations, respectively.} 
	\label{Fig1-2}
	\end{figure}

In Figure \ref{Fig1-3}, we keep $\mathcal{D}(\mathcal{V})/\kappa$ and $m\kappa$ constant and vary $\mathcal{D}(\Omega^0)/\kappa$. In this case, the same conclusion holds with the additively delayed time scale proportional to $m\mathcal{D}(W^0)/\kappa$.\footnote{This is the time required to recover from a hypothetical adversarial attack on $\Omega^0$.} Again, this suggests that as long as $\mathcal{D}(\mathcal{V})/\kappa$ is small, the magnitude of $\mathcal{D}(\Omega^0)/\kappa$ does not significantly affect the asymptotic dynamics.

 	\begin{figure}[htbp]
	\centering
		\begin{subfigure}[b]{0.47\textwidth}
		\includegraphics[width=\textwidth]{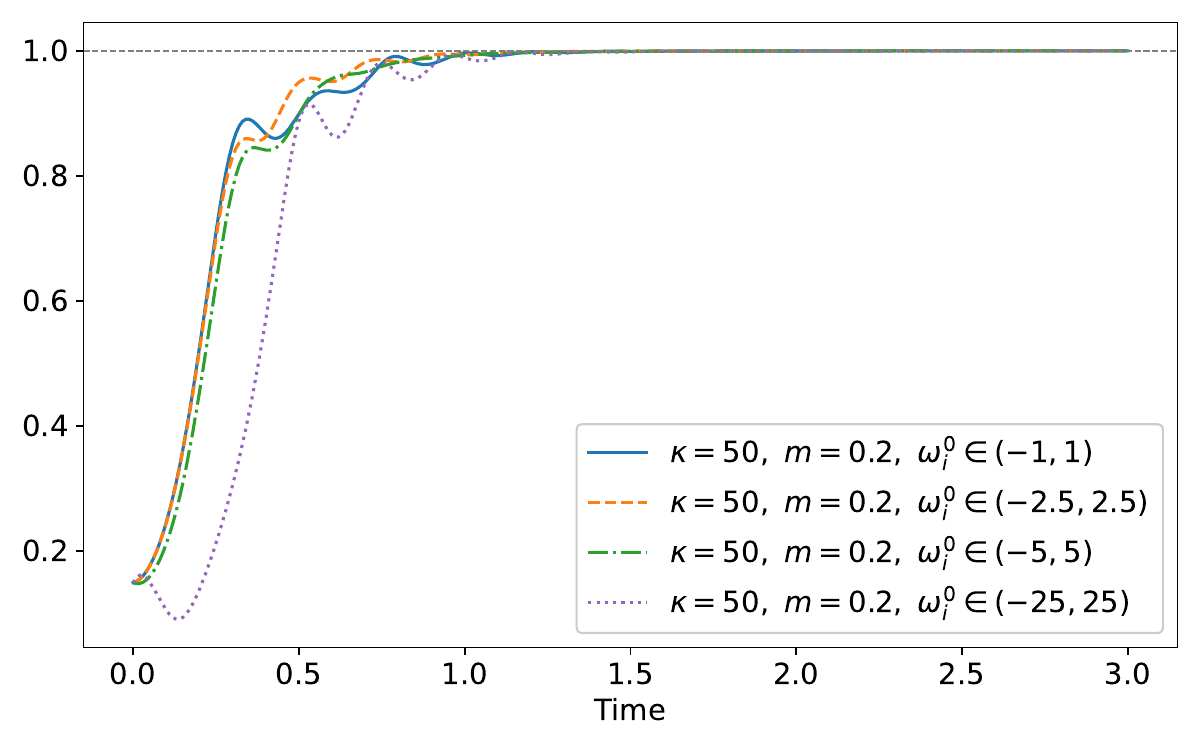}
		\caption{Temporal evolution of $R$.}
	\end{subfigure}
	\vspace{0.1cm}
	\begin{subfigure}[b]{0.47\textwidth}
		\includegraphics[width=\textwidth]{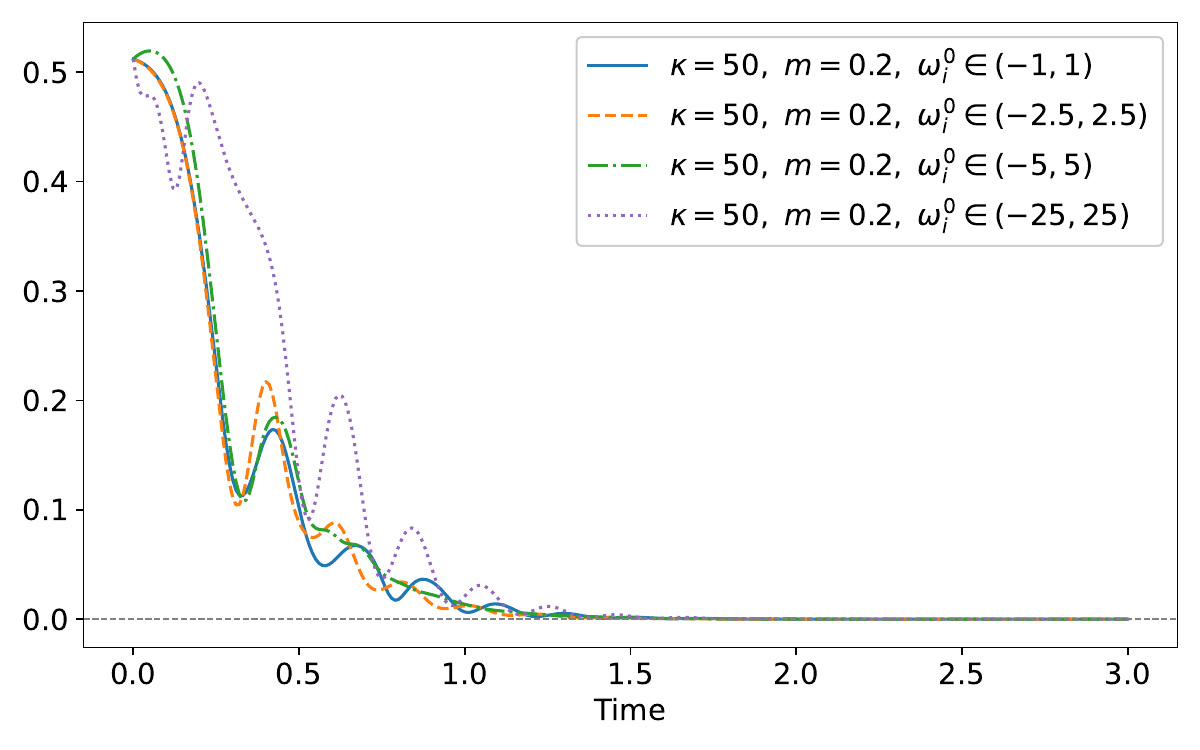}
		\caption{Temporal evolution of $\Delta$.}
	\end{subfigure}
	\caption{(Color online) Simulations of $R$ and $\Delta$ varying ${\mathcal D}(\Omega^0)/\kappa$. We used a single set of intrinsic frequency data uniformly distributed over the interval $[-1,1]$ for all simulations.} 
	\label{Fig1-3}
	\end{figure}

In summary, asymptotic phase-locking appears to occur if $\mathcal{D}(\mathcal{V})/\kappa$ is small, while the $m\kappa$ is the multiplicative time scale of the synchronization, and $\mathcal{D}(\Omega^0)/\kappa$ acts as an additive delay for synchronization to happen. We collect our observations into the following conjecture.
\begin{conjecture}\label{conj:R}
\,
    \begin{enumerate}
        \item (Weak form) For any $\varepsilon>0$, there exists $\delta=\delta(\varepsilon)$ such that if $\mathcal{D}(\mathcal{V})/\kappa<\delta$, then for generic initial data $\Theta^0$, we have
    \[
    \liminf_{t\to\infty}R(t)\ge 1-\varepsilon.
    \]
    \item (Strong form) For any $\varepsilon \in (0,\frac 12)$, there exists a sufficiently small constant $c_\varepsilon>0$ such that $\operatorname{Var}(\mathcal{V})/\kappa<c_\varepsilon$, then for generic initial data $\Theta^0$, we have
    \[
    1-\left(\frac 12+\varepsilon\right)\frac{\operatorname{Var}(\mathcal{V})}{\kappa^2} \le \liminf_{t\to\infty} R(t)\le \limsup_{t\to\infty} R(t)\le 1-\left(\frac 12-\varepsilon\right)\frac{\operatorname{Var}(\mathcal{V})}{\kappa^2}.
    \]
    If, in addition, $\mathcal{D}(\mathcal{V})/\kappa<c_\varepsilon$, then $\Theta(t)$ converges to the unique phase-locked state of \eqref{A-2} confined in the quarter circle \cite{C-H-J-K}.
  \end{enumerate}
  \end{conjecture}
In the following, we present partial progress toward {\bf Conjecture \ref{conj:R}} (1), without providing a full proof. A quantitative analysis of the $C^\infty(0,\infty)$ convergence, together with the main result of this paper, leads to a complete proof, which we postpone to a subsequent paper.
\begin{theorem}
Let $\Theta^0\in \mathbb{R}^N$ be such that $R^0>0$, and let $\varepsilon>0$. Let $\Theta(t)$ be the solution to the Cauchy problem \eqref{A-1}. Then, there exist sufficiently small numbers $a,b,c>0$ depending only on $\Theta^0$ and $\varepsilon$ such that if the initial data $\Omega^0$ and system parameters $\kappa$, $m$, and $\mathcal{V}$ satisfy
\[
D(\mathcal{V})/\kappa< a,\quad D(\Omega^0)/\kappa< b,\quad m\kappa <c,
\]
then asymptotic phase-locking occurs for $\Theta(t)$ with the following lower bounds for $R$:
\begin{align*}
\lim_{t\to\infty}R(t)>
\begin{cases}
    1-\varepsilon&\mathrm{if~}N=2,\\
    1-\frac 2N-\varepsilon&\mathrm{if~}N=3\mathrm{~or~}\eqref{eq:init-nonidentical}\mathrm{~holds},\\
    R^0-\varepsilon&\mathrm{otherwise},
\end{cases}
\end{align*}
where \eqref{eq:init-nonidentical} is the following condition:
\begin{equation}\label{eq:init-nonidentical}
\theta_i^0\not\equiv \theta_j^0 \mod 2\pi,\quad \forall i\neq j\in [N].
\end{equation}
\end{theorem}

\subsection{Piecing three mechanisms together}\label{subsec:piecing}
In this subsection, we summarize our strategy for the proof of Theorem \ref{simplemainthm}, which is to combine three synchronization mechanisms of subsections \ref{subsec:1.6}, \ref{subsec:1.4}, and \ref{subsec:1.5} together. \newline

First, we briefly delineate our sufficient framework in terms of parameters and initial data. Compared to the simple conditions \eqref{B-15-1} for the Kuramoto model without inertia, our framework will be described by four free parameters $\lambda, \ell, \eta, \delta$. More precisely, the first two parameters, $\lambda$ and $\ell$, are the size and diameter of a majority phase cluster that will emerge from the given initial configuration in finite time. The third parameter $\eta$ is responsible for the range of the initial layer time zone $[0,\eta m]$. The last parameter $\delta$ is a lower bound for the ratio $R(t)/ R^0$ in the initial layer time zone. The ranges of the above parameters can be summarized as follows. 
\[ \ell\in(0,\pi), \quad  \lambda\in(0.5,1], \quad \eta\in(0,\infty), \quad  \delta\in(0,1). \]
  Now, we are ready to present our sufficient framework $\eqref{C-1}$:~Let $(\Theta^0, \Omega^0)$ and ${\mathcal V} = \{\nu_i \}$ be given initial data and a set of natural frequencies. Then, we assume that the parameters and initial data satisfy the following set of conditions $({\mathcal F})$:
\begin{equation}  \label{C-1}
\begin{cases}
\displaystyle ({\mathcal F}_1):~~R^0>0, \quad \zeta(\eta)
\leq (1-\delta)R^0. \\
\displaystyle  ({\mathcal F}_2):~
\delta R^0\ge \lambda+(1-\lambda)\cos\frac{\ell}{2} \quad \mathrm{or}\quad 2\lambda + \left( \frac{\xi(\eta)}{\delta R^0} \right)^2\frac{1}{1-\cos(\ell/2)} \leq 1+ \delta R^0. \\
\displaystyle  ({\mathcal F}_3):~~ \xi(\eta)<\sin \frac{\ell}{2} \left( \lambda \cos\frac{\ell}{2} - (1-\lambda) \right). \\
\displaystyle  ({\mathcal F}_4):~~
\frac{\mathcal{D}(\mathcal{V})}{\kappa} +4m\kappa+2m\mathcal{D}(\mathcal{V})< \frac{(2\lambda-1)^{3/2}}{\sqrt{2\lambda}}\frac{2-\lambda}{\sqrt{\lambda/2}+(1-\lambda)}.
\end{cases}
\end{equation}
Here, $\zeta(\eta)=\zeta(m,\kappa,\mathcal{V},\Omega^0,\eta)$ and $\xi(\eta)=\xi(m,\kappa,\mathcal{V},\Omega^0,\eta)$ are dimensionless quantities defined in \eqref{B-14-0-1}.

Note that condition $({\mathcal F}_3)$ necessitates $\lambda\cos\frac{\ell}{2}-(1-\lambda)>0$. Framework $({\mathcal F})$ in $\eqref{C-1}$ is significantly different from those in the literature: framework \eqref{C-1} is the first to apply to generic initial data (the condition $R^0 > 0$ is satisfied for Lebesgue a.e.~initial position data $\Theta^0$ in the position state space ${\mathbb R}^N$), not imposing any restriction on the cardinality or diameter of initial data.
\begin{theorem}\label{T3.1}
Suppose the conditions $({\mathcal F}_1) - ({\mathcal F}_3)$ hold, and let $(\Theta, \Omega)$ be a global solution to the Cauchy problem \eqref{B-1}. Then the following assertions hold. \newline
\begin{enumerate}
\item\emph{(Asymptotic phase-locking)}:~There exists a constant state $\Theta^{\infty} \in {\mathbb R}^N$ such that 
\[
\lim_{t\to\infty } \Big( \| \Theta(t) - \nu_c t {\bf 1}_{[N]} - {\Theta}^\infty \|_{\infty} + \| \Omega(t) - \nu_c {\bf 1}_{[N]} \|_{\infty}  \Big) = 0.
\]
\item \emph{(Finite-time emergence and persistence of a majority cluster)}:~There exist a nonnegative time $t_* > 0$ and a subset $\mathcal{A}$ with $|\mathcal{A}|\ge \lambda N$ such that the majority cluster $\Theta_\mathcal{A} \coloneqq ( \theta_i)_{i\in \mathcal{A}}$ is confined modulo $2\pi$, after some time $t_*$, in an arc whose length is  less than or equal to $ \ell < \pi$; this means that there are integers $k_i $ for $i\in \mathcal{A}$ such that
\[
\mathcal{D}\left((\theta_{i}(t) - 2k_{i}\pi)_{i\in\mathcal{A}}\right)<\ell,\quad t\ge t_*.
\]
There is a unique maximal (with respect to inclusion) such $\mathcal{A}$.
\vspace{0.1cm}
\item  \emph{(Linear arrangement of majority cluster)}:~If in addition \eqref{C-1}-(${\mathcal F}_4$) holds, then the majority cluster $\Theta_\mathcal{A}$ is arranged according to its natural frequencies: there is a constant $c$ depending only on $\frac{\mathcal{D}(\mathcal{V})}{\kappa}$, $m\kappa$, and $\lambda$ such that for any $i,j\in \mathcal{A}$ with $\nu_i\ge \nu_j$, we have
\[
\frac{\nu_i-\nu_j}{\kappa}\le \lim_{t\to\infty}(\theta_i(t)-\theta_j(t)-2(k_i-k_j)\pi)\le c\frac{\nu_i-\nu_j}{\kappa}.
\]
\end{enumerate}
\end{theorem}
\begin{proof}
We will prove statement (2) of Theorem \ref{T3.1} first in Section \ref{sec:5}. Then, in Section \ref{sec:5}, we will derive statement (1) of Theorem \ref{T3.1} from statement (2) using Proposition \ref{P2.1}. On the other hand, the condition \eqref{C-1}-(${\mathcal F}_4$) enables us to tell an additional structural property, called linear arrangement, of the majority cluster described in (2) of Theorem \ref{T3.1}. See Theorem \ref{L4.4}.
\end{proof}

%
%
%
%
\section{Partial phase-locking of majority clusters} \label{sec:3}
\setcounter{equation}{0}
In this section, we establish a version of this partial phase-locking result for the inertial Kuramoto model \eqref{A-1}.  We remark that a similar result has been established in \cite{H-J-K} for the inertial Kuramoto model \eqref{A-1} (see Appendix \ref{app:suppt}). We establish a stronger version of this theorem in this paper (see Corollary \ref{cor:partialphaselocking-initial}).  Theorem \ref{thm:1stpartial} in Section \ref{sec:3} says this is true for the first-order model \eqref{A-2}. It tells us that not only these majority (i.e., $\lambda>\frac 12$) clusters are stable, but they can also control other certain oscillators as well.  In what follows, we describe partial phase-locking in \eqref{A-1} due to majority clusters. For this, we begin by defining the function $f_\lambda:\bbr\rightarrow\bbr$  as
\begin{equation}\label{eq:fgamma}
f_\lambda(\theta) =2\sin\frac{\theta}{2}\left(\lambda\cos\frac{\theta}{2}-(1-\lambda)\right)=\lambda\sin\theta-2(1-\lambda)\sin\frac{\theta}{2},\quad \theta\in \mathbb{R}, \quad \lambda\in \Big (\frac 12, 1 \Big ].
\end{equation}
Then $f_\lambda$ has the following properties.
\begin{lemma}[{\cite[Lemma 4.2]{H-R}}]\label{lem:phi-def}
The function $f_\lambda$ defined in \eqref{eq:fgamma} satisfies the following properties.
\begin{enumerate}
\item
The function $f_\lambda$ has zeros at $\theta=0$ and $\theta=2\cos^{-1} \Big( \frac{1-\lambda}{\lambda} \Big)$, and is positive on the interval $\Big (0,2\cos^{-1}\frac{1-\lambda}{\lambda} \Big)$.
\item
On the interval $(0,2\cos^{-1}\frac{1-\lambda}{\lambda})$, the function $f_\lambda$ is strictly concave and attains its maximum at the unique zero $\theta_*=2\cos^{-1} \Big( \frac{1-\lambda+\sqrt{(1-\lambda)^2+8\lambda^2}}{4\lambda} \Big)$ of $f_{ \lambda}(\theta_*)=0$ in $(0,2\cos^{-1}\frac{1-\lambda}{\lambda})$. 
\end{enumerate}
\end{lemma}
\vspace{0.2cm}

\noindent Thus, for $\delta\in (0,f_\lambda (\theta_*))$, the equation $f_\lambda (\theta)=\delta$ has two zeros $\phi_1=\phi_1(\lambda, \delta)$ and $\phi_2=\phi_2(\lambda, \delta)$ in $(0,2\cos^{-1}\frac{1-\lambda}{\lambda})$, with the ordering
\begin{equation}\label{eq:phi1phi2-ordering}
0<\phi_1<\theta_*<\phi_2<2\cos^{-1} \left( \frac{1-\lambda}{\lambda} \right).
\end{equation}
These are the angles that the majority clusters will form.
\begin{definition}\label{def:phi1phi2}
    For $\lambda\in (1/2,1]$ and $\delta\in \mathbb{R}$ such that
    \begin{align*}
    0<\delta &< \max_{\theta \in (0,2\cos^{-1}( \frac{1}{\lambda} - 1 ) )}f_\lambda(\theta)=f_\lambda(\theta_*)\\
    &=\frac{\left(-3(1-\lambda)+\sqrt{9\lambda^2-2\lambda+1}\right)\sqrt{3\lambda^2+2\lambda-1-(1-\lambda)\sqrt{9\lambda^2-2\lambda+1}}}{4\sqrt{2}\lambda},
    \end{align*}
    we denote by $\phi_1(\lambda,\delta)$ the smaller root and by $\phi_2(\lambda,\delta)$ the larger root among the two distinct roots of the following trigonometric equation in $\theta$:
\[
\lambda\sin\theta-2(1-\lambda)\sin\frac{\theta}{2} = \delta, \quad \theta\in\left(0,2\cos^{-1}\frac{1-\lambda}{\lambda} \right).
\]
\end{definition}
In the next lemma, we collect some facts.
\begin{lemma}[{\cite[Lemma 4.3]{H-R}}]\label{f-estimates}
Let $\theta_*, \phi_1$ and $\phi_2$ be defined as above. Then, the following estimates hold.
\begin{align*}
& (i)~\phi_1< \frac{3\pi\delta}{4(2\lambda-1)}, \quad  \cos^{-1} \left( \frac{1-\lambda}{\lambda} \right) \le\theta_*. \\
& (ii)~f_\lambda\left(\cos^{-1}\frac{1-\lambda}{\lambda} \right)=\frac{(2\lambda-1)^{3/2}}{\sqrt{2\lambda}}\frac{2-\lambda}{\sqrt{\lambda/2}+(1-\lambda)}.
\end{align*}
\end{lemma}

We state our partial-locking theorem for \eqref{A-1} as follows.

\begin{theorem}\label{L4.4}
Suppose that the free real parameters $\lambda, \ell>0$ and index set ${\mathcal A} \subset [N]$ satisfy
\begin{equation}\label{eq:gammaellrange}
\frac{1}{2} < \lambda \leq 1, \quad \ell \in\left(0,2\cos^{-1} \Big( \frac{1}{\lambda} - 1 \Big) \right), \quad |\mathcal{A}|\ge \lambda N,
\end{equation}
and that the system parameters and the free real parameter $\eta>0$ satisfy the following variant of \eqref{C-1}-$({\mathcal F}_3)$ for the index set $\mathcal{A}\subset [N]$:
\begin{align}\label{eq:xi-partial}
\begin{aligned}
\xi(m,\kappa,\mathcal{V}_\mathcal{A},\Omega^0_\mathcal{A},\eta)&=m\mathcal{D}(\mathcal{V}_\mathcal{A})+2m\kappa+\frac{\mathcal{D}(\mathcal{V}_\mathcal{A})}{2\kappa} \\
&+\mathcal{D}(\Omega^0_\mathcal{A})m\max\{1,\eta\}e^{-\max\{1,\eta\}}+\frac{\mathcal{D}(\Omega^0_\mathcal{A})}{2\kappa}\frac{e^{-\eta}}{1-e^{-\eta}}\\
&<\frac\lambda 2\sin \ell-(1-\lambda)\sin\frac{\ell}{2},
\end{aligned}
\end{align}
and let $(\Theta,\Omega)$ be a global solution to \eqref{B-1}. Assume there exists a time $t_1 \geq \eta m$ such that the subensemble $\Theta_\mathcal{A}=(\theta_\ell)_{\ell\in \mathcal{A}}$ satisfies
\begin{align*}
\mathcal{D}(\Theta_\mathcal{A}(t_1))\le \ell.
\end{align*}
Then, the following assertions hold.
\begin{enumerate}
\item (Stability of the majority cluster): One has
\begin{align*}
\sup_{t\ge t_1}\mathcal{D}(\Theta_\mathcal{A}(t))\le \ell,
\end{align*}
and
\begin{align*}
\begin{aligned}
 \limsup_{t\rightarrow\infty}\mathcal{D}(\Theta_\mathcal{A}(t)) &\le \phi_1\left(\lambda, 2m\mathcal{D}(\mathcal{V}_\mathcal{A})+4m\kappa+\frac{\mathcal{D}(\mathcal{V}_\mathcal{A})}{\kappa}\right) \\
 &\stackrel{\mathrm{Lemma~}\ref{f-estimates}(i)}{<}\frac{3\pi}{4(2\lambda-1)}\left(2m\mathcal{D}(\mathcal{V}_\mathcal{A})+4m\kappa+\frac{\mathcal{D}(\mathcal{V}_\mathcal{A})}{\kappa}\right).
\end{aligned}
\end{align*}
\item (Partial linear arrangement) If we assume in addition that
\begin{equation}\label{additional0}
2m\mathcal{D}(\mathcal{V}_\mathcal{A})+4m\kappa+\frac{\mathcal{D}(\mathcal{V}_\mathcal{A})}{\kappa}< \frac{(2\lambda-1)^{3/2}}{\sqrt{2\lambda}}\frac{2-\lambda}{\sqrt{\lambda/2}+(1-\lambda)},
\end{equation}
then the oscillators of $\Theta_\mathcal{A}$ becomes linearly ordered according to their natural frequencies: for $i,j\in \mathcal{A}$, with $\nu_i\ge \nu_j$,
\begin{equation}\label{eq:A-linear-ordered}
\hspace{1cm} \frac{\nu_i-\nu_j}{\kappa}\le \liminf_{t\rightarrow\infty}[\theta_i(t)-\theta_j(t)] \le \limsup_{t\rightarrow\infty}[\theta_i(t)-\theta_j(t)]  \le \frac{\phi_1}{2 \sin(\phi_1/2)(\lambda\cos \phi_1 -(1-\lambda)) }\frac{\nu_i-\nu_j}{\kappa},
\end{equation}
where $\phi_1=\phi_1\left(\lambda, 2m\mathcal{D}(\mathcal{V}_\mathcal{A})+4m\kappa+\frac{\mathcal{D}(\mathcal{V}_\mathcal{A})}{\kappa}\right)$.
\end{enumerate}
Now we assume that there is an index set $\mathcal{B}\subset [N]$ with $\mathcal{B}\supset \mathcal{A}$ satisfying the following variant of \eqref{eq:xi-partial}:
\begin{equation}\label{eq:xi-partial-B}
\xi(m,\kappa,\mathcal{V}_\mathcal{B},\Omega^0_\mathcal{B},\infty)=m\mathcal{D}(\mathcal{V}_\mathcal{B})+2m\kappa+\frac{\mathcal{D}(\mathcal{V}_\mathcal{B})}{2\kappa}<\frac\lambda 2\sin \ell-(1-\lambda)\sin\frac{\ell}{2}.
\end{equation}
Then, the following statements hold.
\begin{enumerate}[resume]
\item (The majority cluster $\mathcal{A}$ confines $\mathcal{B}$) The ensemble $\Theta_\mathcal{B}$ is partially phase-locked:
\[
\sup_{t\ge 0}\mathcal{D}(\Theta_\mathcal{B}(t))<\infty.
\]
In particular, if $\mathcal{B}=[N]$, then asymptotic phase-locking occurs.

\item (Uniqueness of the maximal majority cluster) There is a unique index set $\mathcal{A}_\mathrm{max}\subset [N]$ with
\[
\mathcal{A}\subset \mathcal{A}_\mathrm{max}\subset \mathcal{B}
\]
possessing the following properties (a) and (b):
\begin{enumerate}
    \item (The ensemble $\Theta_{\mathcal{A}_\mathrm{max}}$ forms a cluster) By possibly replacing $\theta_i$ by $\theta_i-2k_i\pi$ for a suitable integer $k_i\in \mathbb{Z}$ over all $i\in \mathcal{A}_\mathrm{max}\setminus\mathcal{A}$, we have
    \begin{equation}\label{eq:max-cluster}
    \limsup_{t\rightarrow\infty}\mathcal{D}(\Theta_{\mathcal{A}_\mathrm{max}}(t))\le \phi_1\left(\lambda, 2m\mathcal{D}(\mathcal{V}_\mathcal{B})+4m\kappa+\frac{\mathcal{D}(\mathcal{V}_\mathcal{B})}{\kappa}\right).
    \end{equation}
    \item (Maximality and quantitative separation)
Any enlargement of $\mathcal{A}_\mathrm{max}$ fails to form a cluster: if $\mathcal{A}_\mathrm{max}\subsetneq \mathcal{B}$, then
    \begin{equation}\label{eq:cluster-maximality}
    \liminf_{t\to\infty}\min_{\substack{i\in \mathcal{B}\setminus\mathcal{A}_\mathrm{max}\\ k\in \mathbb{Z}}}\mathcal{D}(\Theta_{\mathcal{A}_\mathrm{max}}(t)\cup \{\theta_i(t)-2\pi k\})\ge \phi_2\left(\lambda, 2m\mathcal{D}(\mathcal{V}_\mathcal{B})+4m\kappa+\frac{\mathcal{D}(\mathcal{V}_\mathcal{B})}{\kappa}\right),
    \end{equation}
    and we have the following separation estimate:
    \begin{align}
    \begin{aligned} \label{eq:cluster-separation}
 & \liminf_{t\to\infty}\min_{\substack{i\in \mathcal{B}\setminus \mathcal{A}_{\mathrm{max}}\\ j\in \mathcal{A}_{\mathrm{max}}\\ k\in \mathbb{Z}}}|\theta_i(t)-2\pi k-\theta_j(t)| \\
 & \hspace{1cm} \ge \phi_2\left(\lambda, 2m\mathcal{D}(\mathcal{V}_\mathcal{B})+4m\kappa+\frac{\mathcal{D}(\mathcal{V}_\mathcal{B})}{\kappa}\right)-\phi_1\left(\lambda, 2m\mathcal{D}(\mathcal{V}_\mathcal{B})+4m\kappa+\frac{\mathcal{D}(\mathcal{V}_\mathcal{B})}{\kappa}\right).
\end{aligned}
\end{align}
\end{enumerate}
Again, under additional conditions, we have linear arrangement of $\Theta_{\mathcal{A}_\mathrm{max}}$:
\begin{enumerate}[resume]
    \item (Linear arrangement) If we assume in addition that
\begin{equation*}
2m\mathcal{D}(\mathcal{V}_\mathcal{B})+4m\kappa+\frac{\mathcal{D}(\mathcal{V}_\mathcal{B})}{\kappa}< \frac{(2\lambda-1)^{3/2}}{\sqrt{2\lambda}}\frac{2-\lambda}{\sqrt{\lambda/2}+(1-\lambda)},
\end{equation*}
then the oscillators of $\Theta_{\mathcal{A}_\mathrm{max}}$ becomes, after suitable $2\pi$-translations, linearly ordered according to their natural frequencies: for $i,j\in \mathcal{A}_\mathrm{max}$, with $\nu_i\ge \nu_j$,
\[
\frac{\nu_i-\nu_j}{\kappa}\le \liminf_{t\rightarrow\infty}[\theta_i(t)-\theta_j(t)] \le \limsup_{t\rightarrow\infty}[\theta_i(t)-\theta_j(t)] \le \frac{\phi_1}{2\sin(\phi_1/2)\left(\lambda\cos\phi_1-(1-\lambda)\right)}\frac{\nu_i-\nu_j}{\kappa},
\]
where $\phi_1=\phi_1\left(\lambda,2m\mathcal{D}(\mathcal{V}_\mathcal{B})+4m\kappa+\frac{\mathcal{D}(\mathcal{V}_\mathcal{B})}{\kappa}\right)$.
\end{enumerate}
\end{enumerate}
\end{theorem}
\begin{proof}
Since the proofs are very lengthy, we leave them in the following subsections. 
\end{proof}
\begin{remark}
Below, we comment on the contents of serval assertions appearing in Theorem \ref{L4.4}. We refer to Theorem \ref{thm:1stpartial} in Appendix \ref{app:suppt-1} for a version of Theorem \ref{L4.4} in the simpler case of the first-order model \eqref{A-2}. The assertions in Theorem \ref{L4.4} can be rephrased as follows. 
\begin{enumerate}
    \item Once a majority $\mathcal{A}$ of the phase oscillators is concentrated (modulo $2\pi$) in an arc of sufficiently small length $\ell$ at some finite time $t_1\ge \eta m$ bounded away from $0$, they must always stay in an arc of length at most $\ell$ after that time if the coupling strength is sufficiently large compared to $\mathcal{D}(\mathcal{V})$, $\mathcal{D}(\Omega^0)$, and $1/m$. Thus, $\mathcal{A}$ becomes a stable majority cluster.
    \vspace{0.1cm}
    \item With stronger assumptions on the smallness of the normalized natural frequency diameter $\mathcal{D}(\mathcal{V}_\mathcal{A})/\kappa$ and normalized inertia $m\kappa$, the members of the majority cluster $\mathcal{A}$ eventually rearrange themselves according to the linear order of their natural frequencies $\nu_i$.
       \vspace{0.1cm}
    \item The movement of other members in $\mathcal{B}\setminus\mathcal{A}$ is heavily restricted by the majority cluster $\mathcal{A}$, since they become part of the majority cluster $\mathcal{A}$ if they cross paths with them.
       \vspace{0.1cm}
    \item The members of $\mathcal{B}$ which join the majority cluster $\mathcal{A}$ attract each other as well. This allows us to identify a unique maximal majority cluster $\mathcal{A}_\mathrm{max}$ with $\mathcal{A}\subset \mathcal{A}_\mathrm{max}\subset \mathcal{B}$, from which $\mathcal{B}\setminus \mathcal{A}_\mathrm{max}$ distances itself.\footnote{It is tempting to view $\Theta_{\mathcal{A}_\mathrm{max}}$ as repelling $\Theta_{\mathcal{B}\setminus\mathcal{A}_\mathrm{max}}$, but this is not the case. What is happening is that $\Theta_{\mathcal{A}_\mathrm{max}}$ is doing its best to include $\Theta_{\mathcal{B}\setminus \mathcal{A}_\mathrm{max}}$, but $\Theta_{[N]\setminus \mathcal{A}_\mathrm{max}}$ is on the opposite side of the circle, engaging in a tug-of-war with $\Theta_{\mathcal{A}_\mathrm{max}}$, pulling $\Theta_{\mathcal{B}\setminus \mathcal{A}_{\mathrm{max}}}$ away from $\Theta_{\mathcal{A}_{\mathrm{max}}}$.}
\end{enumerate}
Simply put, a concentrated majority cluster is stable and attractive.
\end{remark}
As an application of Theorem \ref{L4.4}, we use the finite speed of \eqref{B-1} to show the stability of majority clusters for initial data as follows.
\begin{corollary}\label{cor:partialphaselocking-initial}
    Suppose that the free real parameters $\lambda,\ell$ satisfy
\[
\frac{1}{2} < \lambda \le 1, \quad \ell \in\left(0,2\cos^{-1} \left( \frac{1}{\lambda} - 1 \right) \right),
\]
and that the system parameters, free real parameter $\eta>0$, and index set $\mathcal{A}\subset [N]$ satisfy $|\mathcal{A}|\ge \lambda N$, \eqref{eq:xi-partial}, and
\begin{align*}
\mathcal{D}(\Theta_\mathcal{A}^0)\le \ell-m(1-e^{-\eta})\mathcal{D}(\Omega_\mathcal{A})-(\eta m-m+me^{-\eta})(\mathcal{D}(\mathcal{V}_\mathcal{A})+2\kappa),
\end{align*}
and let $(\Theta,\Omega)$ be a global solution to \eqref{B-1}. Then, the conclusion of Theorem \ref{L4.4} holds with $t_1=\eta m$.
\end{corollary}
\begin{proof}
    By Lemma \ref{L2.2} (2), we have that for $i,j\in \mathcal{A}$ and $t\ge 0$,
    \begin{align*}
        |\dot{\theta}_i-\dot{\theta}_j|\le e^{-t/m}|\omega_i^0-\omega_j^0|+(1-e^{-t/m})(|\nu_i-\nu_j|+2\kappa)\le e^{-t/m}\mathcal{D}(\Omega_\mathcal{A})+(1-e^{-t/m})(\mathcal{D}(\mathcal{V}_\mathcal{A})+2\kappa)
    \end{align*}
    and hence, we have
    \begin{align*}
        \mathcal{D}(\Theta_\mathcal{A}(\eta m))-\mathcal{D}(\Theta_\mathcal{A}^0)&\le \int_0^{\eta m} \left(e^{-t/m}\mathcal{D}(\Omega_\mathcal{A})+(1-e^{-t/m})(\mathcal{D}(\mathcal{V}_\mathcal{A})+2\kappa)\right)dt\\
        &=m(1-e^{-\eta})\mathcal{D}(\Omega_\mathcal{A})+(\eta m-m+me^{-\eta})(\mathcal{D}(\mathcal{V}_\mathcal{A})+2\kappa).
    \end{align*}
By the condition of corollary, we have
    \[
    \mathcal{D}(\Theta_\mathcal{A}(\eta m))\le \ell.
    \]
 Then, we may apply Theorem \ref{L4.4} with $t_1=\eta m$ to get the desired estimate. 
\end{proof}

\begin{remark}
In general, even though the subensemble $\Theta_\mathcal{A}$, viewed as particles on the unit circle $\mathbb{S}^1$, may lie on a small arc on the circle, it need not be the case when they are viewed as particles on the real line $\mathbb{R}$. See Figure \ref{Fig2} for a demonstration of this phenomenon. Nevertheless, owing to the $2\pi\mathbb{Z}$-translation invariance of \eqref{A-1}, Theorem \ref{L4.4} may be employed harmlessly for our purposes of proving partial or asymptotic phase-locking whenever we have that  $\Theta_\mathcal{A}$ lies on a small arc on the circle.
	\begin{figure}[htbp]
	\centering
		\begin{subfigure}[b]{0.4\textwidth}
		\includegraphics[width=\textwidth]{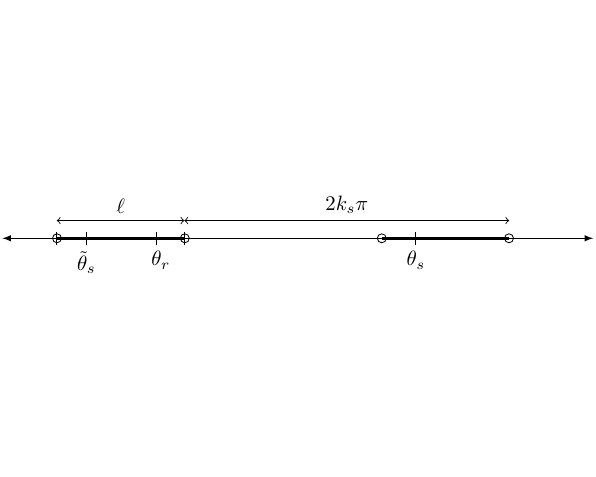}
		\caption{Geometric configuration on $\mathbb{R}^1$.}
	\end{subfigure}
	\hspace{0.3cm}
	\begin{subfigure}[b]{0.4\textwidth}
		\includegraphics[width=\textwidth]{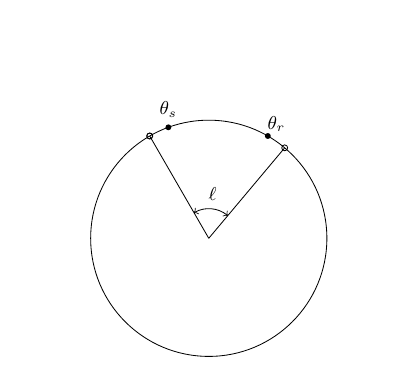}
		\caption{Geometric configuration on $\mathbb{T}^1$.}
	\end{subfigure}
	\caption{Schematic descriptions for virtual rearrangement; $\theta_s$ and $\theta_r$ are on the same arc on ${\mathbb T}^1$ but may not be in the same on ${\mathbb R}^1$. Considering $2\pi k_s$ translation (for some $k_s$) enable us to regard $\tilde \theta_s\coloneqq \theta_s -2 k_s\pi$ as $\theta_s$ in our estimations. }
	\label{Fig2}
	\end{figure}
\end{remark}	
We are now ready to prove several assertions in Theorem \ref{L4.4} one by one.
In the proof, we will employ the Sturm--Picone comparison principle. The precise statements and proofs of the relevant lemmas will be provided in Appendix \ref{app:sturm-picone}. Assuming these lemmas, we proceed with the proof of Theorem \ref{L4.4}.
\subsection{Proof of the first assertion in Theorem \ref{L4.4} (Stability of $\Theta_\mathcal{A}$)}
In this subsection, we show the first assertion that the majority cluster formed at $t = t_1 \ge \eta m$ persists afterwards. \newline

\noindent Recall that there exists an index set $\mathcal{A}$ with $|\mathcal{A}|\ge \lambda N$ such that the corresponding subensemble $\Theta_\mathcal{A}= ( \theta_i )_{i\in \mathcal{A}}$ satisfies
\[
\mathcal{D}(\Theta_\mathcal{A}(t_1))\le \ell,
\]
for some $t_1 \geq \eta m$. Next, we show that
\begin{equation} \label{Bp-1}
{\mathcal D}(\Theta_\mathcal{A}(t)) \le \ell,\quad \forall t\ge t_1,
\end{equation}
and
\begin{equation}\label{Bp-2}
\limsup_{t\to\infty}\mathcal{D}(\Theta_\mathcal{A}(t))\le \phi_1\left(\lambda,2m\mathcal{D}(\mathcal{V}_\mathcal{A})+4m\kappa+\frac{\mathcal{D}(\mathcal{V}_\mathcal{A})}{\kappa}\right).
\end{equation}
By definition of $f_\lambda$, condition \eqref{eq:xi-partial} becomes
\begin{align*}
2\xi(m,\kappa,\mathcal{V}_\mathcal{A},\Omega^0_\mathcal{A},\eta)<f_\lambda(\ell),
\end{align*}
and since $f_\lambda$ has its maximum in $\Big [0,2\cos^{-1}\frac{1-\lambda}{\lambda} \Big]$ at $\theta_*$, we also have 
\[
2\xi(m,\kappa,\mathcal{V}_\mathcal{A},\Omega^0_\mathcal{A},\eta)<
f_\lambda(\theta_*).
\] 
By Lemma \ref{lem:phi-def}, the equation $f_\lambda(\theta)=2\xi(m,\kappa,\mathcal{V}_\mathcal{A},\Omega^0_\mathcal{A},\eta)$ has two zeros $\theta=\phi_1(\eta)$ and $\theta=\phi_2(\eta)$ in the interval $(0,2\cos^{-1}\frac{1-\lambda}{\lambda})$, which, by the strict concavity of $f_\lambda$, have the ordering
\[
0<\phi_1(\eta)<\ell<\phi_2(\eta)<2\cos^{-1} \left( \frac{1-\lambda}{\lambda} \right).
\]
Choose any $\psi_1,\psi_2\in\bbr$ with
\begin{equation}\label{eq:psi1-psi2-choice}
\phi_1(\eta)<\psi_1\le \ell\le \psi_2<\phi_2(\eta),\quad \psi_1<\psi_2.
\end{equation}
Again by the strict concavity of $f_\lambda$, there exists a positive constant $c>0$ such that
\begin{equation}\label{eq:psi-compact}
2\xi(m,\kappa,\mathcal{V}_\mathcal{A},\Omega^0_\mathcal{A},\eta)+c<f_\lambda(\theta),\quad \theta\in [\psi_1,\psi_2].
\end{equation}

We will show that whenever $\mathcal{D}(\Theta_\mathcal{A})\in [\psi_1,\psi_2]$ on the time interval $[t_1,\infty)$, it has a negative upper Dini (time) derivative: $D^+\mathcal{D}(\Theta_\mathcal{A})<0$. Note that we may compute the upper Dini derivative as follows:
\begin{equation}\label{eq:dini}
D^+\mathcal{D}(\Theta_\mathcal{A})=\max_{\substack{i,j\in \mathcal{A}\\ \theta_i=\max_{k\in \mathcal{A}} \theta_k\\ \theta_j=\min_{k\in \mathcal{A}} \theta_k}}\left(\dot\theta_i-\dot\theta_j\right)
\end{equation}

\noindent By Invoking Lemma \ref{L:approxaut}, for $i,j\in \mathcal{A}$ we have
\begin{align*}
& \dot\theta_i(t)-\dot\theta_j(t) \le   \frac{\kappa\left(1-e^{-t/m}\right)}{N}\sum_{l\in \mathcal{A}}\left(\sin(\theta_l(t)-\theta_i(t))-\sin(\theta_l(t)-\theta_j(t))\right)\\
& \hspace{0.7cm} +(\omega^0_i-\omega^0_j) e^{-t/m} + (\nu_i-\nu_j) (1 -e^{-t/m}) \\
&  \hspace{0.7cm} + 2\kappa \left(\mathcal{D}(\Omega^0_\mathcal{A})te^{-t/m}\left(1-e^{-t/m}\right)+m (\mathcal{D}(\mathcal{V}_\mathcal{A})+2\kappa)\left(1-e^{-t/m}\right)^3\right)\\
&  \hspace{0.7cm} + \frac{N-|\mathcal{A}|}{N}2\kappa\left|\sin\frac{\theta_i(t) -\theta_j(t)}{2}\right|\left(1-e^{-t/m}\right)\\
&  \hspace{0.7cm} \le   -\frac{2\kappa\left(1-e^{-t/m}\right)}{N}\sin\frac{\theta_i(t)-\theta_j(t)}{2}
\sum_{l\in \mathcal{A}}\cos\left(\theta_l(t)-\frac{\theta_i(t)+\theta_j(t)}{2}\right)\\
&  \hspace{0.7cm} +\mathcal{D}(\Omega^0_\mathcal{A}) e^{-t/m} + \mathcal{D}(\mathcal{V}_\mathcal{A}) (1 -e^{-t/m}) \\
&  \hspace{0.7cm}+ 2\kappa \left(\mathcal{D}(\Omega^0_\mathcal{A})te^{-t/m}\left(1-e^{-t/m}\right)+m (\mathcal{D}(\mathcal{V}_\mathcal{A})+2\kappa)\left(1-e^{-t/m}\right)^3\right)\\
& \hspace{0.7cm} +\frac{N-|\mathcal{A}|}{N}2\kappa\left|\sin\frac{\theta_i(t) -\theta_j(t)}{2}\right|\left(1-e^{-t/m}\right).
\end{align*}
If $t\ge t_1$ is such that $\mathcal{D}(\Theta_\mathcal{A}(t))\le \pi$, and $i,j\in \mathcal{A}$ are so that 
\[ \theta_i(t)=\max_{k\in \mathcal{A}}\theta_k(t) \quad \mbox{and} \quad \theta_j(t)=\min_{k\in \mathcal{A}}\theta_k(t), \]
then we have
\[
\cos\left(\theta_l(t)-\frac{\theta_i(t)+\theta_j(t)}{2}\right)\ge \cos\frac{\theta_i(t)-\theta_j(t)}{2},\quad l\in \mathcal{A},
\]
so that, continuing the above estimate and $|\mathcal{A}|\ge \lambda N$,
\begin{align*}
    \dot\theta_i(t)-\dot\theta_j(t) \le &  -2\lambda\kappa\left(1-e^{-t/m}\right)\sin\frac{\theta_i(t)-\theta_j(t)}{2}
\cos\frac{\theta_i(t)-\theta_j(t)}{2}\\
&+\mathcal{D}(\Omega^0_\mathcal{A}) e^{-t/m} + \mathcal{D}(\mathcal{V}_\mathcal{A}) (1 -e^{-t/m}) \\
&+ 2\kappa \left(\mathcal{D}(\Omega^0_\mathcal{A})te^{-t/m}\left(1-e^{-t/m}\right)+m (\mathcal{D}(\mathcal{V}_\mathcal{A})+2\kappa)\left(1-e^{-t/m}\right)^3\right)\\
&+(1-\lambda)2\kappa\sin\frac{\theta_i(t) -\theta_j(t)}{2}\left(1-e^{-t/m}\right)\\
\le &  -2\kappa\left(1-e^{-t/m}\right)\sin\frac{\theta_i(t)-\theta_j(t)}{2}
\left(\lambda \cos\frac{\theta_i(t)-\theta_j(t)}{2}-(1-\lambda)\right)\\
&+\mathcal{D}(\Omega^0_\mathcal{A}) e^{-t/m} + \mathcal{D}(\mathcal{V}_\mathcal{A}) (1 -e^{-t/m}) \\
&+ 2\kappa \left(\mathcal{D}(\Omega^0_\mathcal{A})te^{-t/m}\left(1-e^{-t/m}\right)+m (\mathcal{D}(\mathcal{V}_\mathcal{A})+2\kappa)\left(1-e^{-t/m}\right)^3\right).
\end{align*}
Again, by Invoking \eqref{eq:dini}, we have
\begin{align*}
    \left.D^+\right|_t\mathcal{D}(\Theta_\mathcal{A})\le&-2\kappa\left(1-e^{-t/m}\right)\sin\frac{\mathcal{D}(\Theta_\mathcal{A}(t))}{2}
\left(\lambda \cos\frac{\mathcal{D}(\Theta_\mathcal{A}(t))}{2}-(1-\lambda)\right)\\
&+\mathcal{D}(\Omega^0_\mathcal{A}) e^{-t/m} + \mathcal{D}(\mathcal{V}_\mathcal{A}) (1 -e^{-t/m}) \\
&+  2\kappa \left(\mathcal{D}(\Omega^0_\mathcal{A})te^{-t/m}\left(1-e^{-t/m}\right)+m (\mathcal{D}(\mathcal{V}_\mathcal{A})+2\kappa)\left(1-e^{-t/m}\right)^3\right)\\
\le&-\kappa \left(1-e^{-t/m}\right)\left(f_\lambda(\mathcal{D}(\Theta_\mathcal{A}(t)))-2\xi(m,\kappa,\mathcal{V}_\mathcal{A},\Omega^0_\mathcal{A},\eta)\right).
\end{align*}
Now, if time $t_0\ge t_1$ is such that 
\[ \mathcal{D}(\Theta_\mathcal{A}(t_0))\in [\psi_1,\psi_2], \]
where $\psi_1,\psi_2$ are as in \eqref{eq:psi1-psi2-choice}, then, by \eqref{eq:psi-compact}, we have
\[
\left.D^+\right|_{t_0}\mathcal{D}(\Theta_\mathcal{A})\le -\kappa \left(1-e^{-t/m}\right)c<0.
\]
By a standard exit-time argument, we can easily establish that
\[
D(\Theta_\mathcal{A}(t))\le \ell,\quad \forall t\ge t_1,
\]
i.e., we have \eqref{Bp-1}, and we can establish that there is a finite time $T\ge t_1$ with $T\le t_1+\frac{\psi_2-\psi_1}{\kappa \left(1-e^{-t/m}\right)c}$  such that
\[
D(\Theta_\mathcal{A}(t))\le\psi_1,\quad \forall t\ge T.
\]
Since $\psi_1\in (\phi_1(\lambda, 2\xi(m,\kappa,\mathcal{V}_\mathcal{A},\Omega^0_\mathcal{A},\eta)),\ell)$ was arbitrary, we conclude that
\begin{align}\label{Bp-7}
\limsup_{t\rightarrow\infty}D(\Theta_\mathcal{A}(t))\le \phi_1(\lambda, 2\xi(m,\kappa,\mathcal{V}_\mathcal{A},\Omega^0_\mathcal{A},\eta)).
\end{align}
\noindent So far, we have shown that under the assumptions of Theorem \ref{L4.4}, we have \eqref{Bp-1} and \eqref{Bp-7}. Let $\eta'\ge \eta$ be arbitrary. By \eqref{Bp-1}, we have
\[
\mathcal{D}(\Theta_\mathcal{A}(t_1'))\le \ell
\]
for $t_1'=\max\{t_1,\eta' m\}\ge \eta' m$, and since the map $\eta \mapsto \xi(\eta)$ is decreasing, we have
\[
\xi(m,\kappa,\mathcal{V}_\mathcal{A},\Omega^0_\mathcal{A},\eta')\le \xi(m,\kappa,\mathcal{V}_\mathcal{A},\Omega^0_\mathcal{A},\eta)<\frac{\lambda}{2}\sin\ell-(1-\lambda)\sin\frac{\ell}{2}.
\]
So the assumptions of Theorem \ref{L4.4} are satisfied with $\eta$ replaced by $\eta'$ and $t_1$ replaced by $t_1'$, and we have \eqref{Bp-7} for $\eta'$:
\[
\limsup_{t\rightarrow\infty}\mathcal{D}(\Theta_\mathcal{A}(t))\le \phi_1(\lambda, 2\xi(m,\kappa,\mathcal{V}_\mathcal{A},\Omega^0_\mathcal{A},\eta')).
\]
Since $\eta'\ge \eta$ was arbitrary, take the limit $\eta'\to\infty$ to obtain
\[
\limsup_{t\rightarrow\infty}\mathcal{D}(\Theta_\mathcal{A}(t))\le \phi_1(\lambda, 2\xi(m,\kappa,\mathcal{V}_\mathcal{A},\Omega^0_\mathcal{A},\infty)),
\]
which is \eqref{Bp-2}. This completes the proof of the first assertion.

\subsection{Proof of the second assertion in Theorem \ref{L4.4} (Linear arrangement of $\Theta_{{\mathcal A}}$)}
Now we assume in addition that the following relations \eqref{additional0} hold:
\[
2m\mathcal{D}(\mathcal{V}_\mathcal{A})+4m\kappa+\frac{\mathcal{D}(\mathcal{V}_\mathcal{A})}{\kappa}< \frac{(2\lambda-1)^{3/2}}{\sqrt{2\lambda}}\frac{2-\lambda}{\sqrt{\lambda/2}+(1-\lambda)}.
\]
We claim that the following relations holds:~ for $i,j\in \mathcal{A}$, with $\nu_i\ge \nu_j$, we have
\[ 
 \frac{\nu_i-\nu_j}{\kappa} \le \liminf_{t\rightarrow\infty}[\theta_i(t)-\theta_j(t)]\le \limsup_{t\rightarrow\infty}[\theta_i(t)-\theta_j(t)] \le \frac{\phi_1}{2\sin(\phi_1/2)\left(\lambda\cos\phi_1-(1-\lambda)\right)}\frac{\nu_i-\nu_j}{\kappa},
\]
where $\phi_1 = \phi_1\left(\lambda, 2m\mathcal{D}(\mathcal{V}_\mathcal{A})+4m\kappa+\frac{\mathcal{D}(\mathcal{V}_\mathcal{A})}{\kappa}\right)$. \newline

Recall  \eqref{Bp-2}:
\begin{align}\tag{\ref{Bp-2}}
\limsup_{t\to\infty}\mathcal{D}(\Theta_\mathcal{A}(t))\le \phi_1\left(\lambda,2m\mathcal{D}(\mathcal{V}_\mathcal{A})+4m\kappa+\frac{\mathcal{D}(\mathcal{V}_\mathcal{A})}{\kappa}\right).
\end{align}
Since we have $\phi_1\le \theta_*$ by \eqref{eq:phi1phi2-ordering}, $\cos^{-1}\frac{1-\lambda}{\lambda}\le\theta_*$ by Lemma \ref{f-estimates}, and since $f_\lambda$ is strictly increasing on $[0,\theta_*]$, we have the following equivalence:
\begin{align}
&\phi_1\left(\lambda,2m\mathcal{D}(\mathcal{V}_\mathcal{A})+4m\kappa+\frac{\mathcal{D}(\mathcal{V}_\mathcal{A})}{\kappa}\right)< \cos^{-1} \Big( \frac{1-\lambda}{\lambda} \Big)\label{eq:phi1-small-linear-arrangement}\\
& \hspace{0.5cm} \Longleftrightarrow \quad  f_\lambda\left(\phi_1\left(\lambda,2m\mathcal{D}(\mathcal{V}_\mathcal{A})+4m\kappa+\frac{\mathcal{D}(\mathcal{V}_\mathcal{A})}{\kappa}\right)\right)< f_\lambda\left(\cos^{-1}\left(\frac{1-\lambda}{\lambda} \right) \right)\nonumber\\
& \hspace{0.5cm}\Longleftrightarrow \quad  2m\mathcal{D}(\mathcal{V}_\mathcal{A})+4m\kappa+\frac{\mathcal{D}(\mathcal{V}_\mathcal{A})}{\kappa}< \frac{(2\lambda-1)^{3/2}}{\sqrt{2\lambda}}\frac{2-\lambda}{\sqrt{\lambda/2}+(1-\lambda)} \quad (\because \mathrm{Lemma~\ref{f-estimates}~(ii)})\nonumber\\
& \hspace{0.5cm} \Longleftrightarrow \quad \eqref{additional0},\nonumber
\end{align}
so every statement here, in particular the first statement, is true. Let
\begin{equation}\label{eq:psi-choice}
\psi_1\in\left(\phi_1\left(\lambda,2m\mathcal{D}(\mathcal{V}_\mathcal{A})+4m\kappa+\frac{\mathcal{D}(\mathcal{V}_\mathcal{A})}{\kappa}\right),\cos^{-1}\frac{1-\lambda}{\lambda}\right)
\end{equation}
be arbitrary. By \eqref{Bp-2}, there exists a finite time $T\ge t_1$ such that
\begin{equation}\label{trapped}
D(\Theta_\mathcal{A}(t))\le \psi_1,\quad \forall~ t\ge T,
\end{equation}

Let $i,j\in \mathcal{A}$ and $t\ge T$. Because we are looking in the long-term with the a priori guarantee that $\theta_i$ and $\theta_j$ are always contained in an $\ell$-arc for times $t\ge t_1$, we need not be restrained by the myopic approach of the previous subsection, i.e., estimating $\dot\Theta_\mathcal{A}(t)$ from $\Theta_\mathcal{A}(t)$ and $\dot\Theta_\mathcal{A}^0$. Instead, we may employ the second-order ODE directly:
\begin{align}\label{individualdiff}
\begin{aligned}
m(\ddot \theta_i-\ddot\theta_j)+\dot{\theta}_i-\dot{\theta}_j&=\nu_i-\nu_j+\frac{\kappa}{N}\sum_{k=1}^N\left[\sin(\theta_k-\theta_i)-\sin(\theta_k-\theta_j)\right]\\
&= \nu_i-\nu_j-2 \kappa \sin \Big( \frac{\theta_i-\theta_j}{2} \Big) \cdot\frac{1}{N}\sum_{k=1}^N \cos\left(\theta_k-\frac{\theta_i+\theta_j}{2}\right),
\end{aligned}
\end{align}
where
\begin{equation}\label{crude-mf0}
-1\le \frac{1}{N}\sum_{k=1}^N \cos\left(\theta_k-\frac{\theta_i+\theta_j}{2}\right)\le 1
\end{equation}
and (recalling the assumption $t\ge T$)
\begin{align}\label{crude-mf}
\begin{aligned}
\frac{1}{N}\sum_{k=1}^N \cos\left(\theta_k-\frac{\theta_i+\theta_j}{2}\right) &= \frac{1}{N}\sum_{k\in\mathcal{A}} \cos\left(\theta_k-\frac{\theta_i+\theta_j}{2}\right)+\frac{1}{N}\sum_{k\in [N]\backslash\mathcal{A}} \cos\left(\theta_k-\frac{\theta_i+\theta_j}{2}\right)\\
& \ge\frac{|\mathcal{A}|}{N}\cos \mathcal{D}(\Theta_\mathcal{A})-\frac{N-|\mathcal{A}|}{N} \ge \lambda\cos \mathcal{D}(\Theta_\mathcal{A})-(1-\lambda).
\end{aligned}
\end{align}
The penultimate inequality uses the fact that since 
\[ \mathcal{D}(\Theta_\mathcal{A}(t))\stackrel{\mathclap{\eqref{trapped}}}{\le} \psi_1<\cos^{-1}\frac{1-\lambda}{\lambda}\le \frac \pi 2, \]
we have 
\[ \cos\left(\theta_k-(\theta_i+\theta_j)/2\right)\ge \cos \mathcal{D}(\Theta_\mathcal{A}) \quad \forall~i,j,k\in \mathcal{A}. \]
Note that since we are not assuming that $\theta_i$ and $\theta_j$ are extremal in $\Theta_\mathcal{A}$ as before, we cannot say that 
\[ \left|\theta_k-\frac{\theta_i+\theta_j}{2}\right|\le \frac{D(\Theta_\mathcal{A})}{2}, \]
but we can only say that 
\[ \left|\theta_k-\frac{\theta_i+\theta_j}{2}\right|\le D(\Theta_\mathcal{A}). \]
The final inequality uses the fact that $1+\cos\mathcal{D}(\Theta_\mathcal{A})\ge 0$. \newline

Note that our assumption \eqref{additional0}, which implies \eqref{eq:phi1-small-linear-arrangement}. This allows the choice of $\psi_1$ as in \eqref{eq:psi-choice} and it tells us that \eqref{crude-mf} is positive:
\begin{equation}\label{crude-mf2}
\frac{1}{N}\sum_{k=1}^N \cos\left(\theta_k-\frac{\theta_i+\theta_j}{2}\right)~\stackrel{\mathclap{\eqref{crude-mf}}}{\ge} \lambda\cos D(\Theta_\mathcal{A})-(1-\lambda)\stackrel{\mathclap{\eqref{eq:psi-choice}}}{\ge} \lambda\cos \psi_1-(1-\lambda)~\stackrel{\mathclap{\eqref{eq:psi-choice}}}{>}~0,\quad\forall t\ge T.
\end{equation}
With the estimates \eqref{crude-mf0} and \eqref{crude-mf2} on the mean-field term at our disposal, we are now ready to prove the desired statement. Next, we consider two cases: $\nu_i = \nu_j$ or $\nu_i > \nu_j$. \newline

\noindent$\diamond$ Case A ($\nu_i=\nu_j$): In this case, \eqref{individualdiff} becomes
\[
m(\ddot\theta_i-\ddot\theta_j)+\dot{\theta}_i-\dot{\theta}_j=- \frac{2\kappa}{N}\sum_{k=1}^N \cos\left(\theta_k-\frac{\theta_i+\theta_j}{2}\right)\cdot\sin\frac{\theta_i-\theta_j}{2},\quad t\ge T.
\]
Thus, for a time interval $J\subset [T,\infty)$ on which $\theta_i-\theta_j> 0$, we have 
\[
\theta_i-\theta_j\in \left(0,\frac{\pi}{2}\right) \quad \mbox{and} \quad m(\ddot\theta_i-\ddot\theta_j)+\dot{\theta}_i-\dot{\theta}_j\stackrel{\eqref{crude-mf0}}{\ge}-2\kappa \sin \frac{\theta_i-\theta_j}{2}> - \kappa(\theta_i-\theta_j).
\]
Likewise, for a time interval $J\subset [T,\infty)$ on which $\theta_i-\theta_j< 0$, we have 
\[
\theta_i-\theta_j\in \left(-\frac{\pi}{2},0\right) \quad \mbox{and} \quad  m(\ddot\theta_i-\ddot\theta_j)+\dot{\theta}_i-\dot{\theta}_j\stackrel{\eqref{crude-mf0}}{\le}-2\kappa \sin \frac{\theta_i-\theta_j}{2}< - \kappa(\theta_i-\theta_j).
\]
Note also that $m\kappa <\frac 14$, which follows from \eqref{eq:xi-partial}:
\[
2m\kappa\le \xi(m,\kappa,\mathcal{V}_\mathcal{A},\Omega^0_\mathcal{A},\eta)\stackrel{\eqref{eq:xi-partial}}{<}\frac\lambda 2\sin \ell-(1-\lambda)\sin\frac{\ell}{2}\le \frac 12.
\]
Therefore, $y=\theta_i-\theta_j$ satisfies the hypothesis of statement (3) of Lemma \ref{lem:sturm-picone}, with $I=(T,\infty)$, $a=m$, $b=1$, and $c=\kappa$. It follows that $\theta_i-\theta_j$ cannot change sign twice in $(T,\infty)$, so that there exists a time $t_2\in (T,\infty)$ so that either $\theta_i-\theta_j>0$ on $(t_2,\infty)$, $\theta_i-\theta_j=0$ on $(t_2,\infty)$, or $\theta_i-\theta_j<0$ on $(t_2,\infty)$. There is nothing to prove if $\theta_i-\theta_j=0$ on $(t_2,\infty)$ (in fact, the particle exchange symmetry \eqref{eq:particle-exchange} with the transposition $\pi=(i,j)$ along with the time-autonomy and the uniqueness of solutions to \eqref{B-1} implies that $\theta_i=\theta_j$ for all $t\ge 0$). Thus, by switching $i$ and $j$ if necessary, we may assume without loss of generality that $\theta_i-\theta_j> 0$ on $(t_2,\infty)$.
Using \eqref{individualdiff} and \eqref{crude-mf2}, we have for $t\ge T$,
\begin{align*}
0&\ge m(\ddot\theta_i-\ddot\theta_j)+\dot{\theta}_i-\dot{\theta}_j+2\kappa (\lambda\cos\psi_1-(1-\lambda))\sin\frac{\theta_i-\theta_j}{2}\\
&\ge m(\ddot\theta_i-\ddot\theta_j)+\dot{\theta}_i-\dot{\theta}_j+2\kappa (\lambda\cos\psi_1-(1-\lambda))\cdot\frac{1}{\pi}(\theta_i-\theta_j)~(\because 0\le \theta_i-\theta_j< \pi).
\end{align*}
Note also that, from \eqref{trapped}, we have the uniform boundedness of $y=\theta_i-\theta_j$; from Lemma \ref{L2.2}, we have the uniform boundedness of $\dot{y}=\dot\theta_i-\dot\theta_j$; and from the defining equation \eqref{A-1}, we have the uniform boundedness of $\ddot{y}=\ddot\theta_i-\ddot\theta_j$. Therefore, by the Barbalat-type Lemma \ref{lem:barbalat}, we have
\[
\lim_{t\to\infty}\left(\theta_i(t)-\theta_j(t)\right)=0,\quad \lim_{t\to\infty}\left(\dot\theta_i(t)-\dot\theta_j(t)\right)=0,
\]
as desired. \newline

\noindent$\diamond$ Case B $(\nu_i>\nu_j)$: Recall \eqref{trapped}, which implies 
\[ |\theta_i(t)-\theta_j(t)|\le \psi_1 \quad \mbox{for $t\ge T$}. \]
By \eqref{individualdiff} and \eqref{crude-mf0}, we have
\[
m(\ddot \theta_i-\ddot\theta_j)+\dot{\theta}_i-\dot{\theta}_j\ge \nu_i-\nu_j-2\kappa \sin\left(\frac{\theta_i-\theta_j}{2}\right),\quad \mathrm{whenever~} 0\le\theta_i-\theta_j\le \psi_1.
\]
Recall that $\psi_1\in (0,\frac{\pi}{2})$ and $|\nu_i-\nu_j|<\kappa$. Taking the tangent line to the graph of $y=\sin x$ at $x=\sin^{-1}\frac{\nu_i-\nu_j}{2\kappa}$, and using the concavity of $\sin x$ for $x\in [0,\frac{\pi}{2}]$, we have
\[
\sin x\le \sqrt{1-\left(\frac{\nu_i-\nu_j}{2\kappa}\right)^2}\left(x-\sin^{-1}\frac{\nu_i-\nu_j}{2\kappa}\right)+\frac{\nu_i-\nu_j}{2\kappa},\quad x\in \left[0,\frac{\pi}{2}\right].
\]
By substituting $x=\frac{\theta_i-\theta_j}{2}$, we have
\begin{align*}
m(\ddot \theta_i-\ddot\theta_j)+(\dot{\theta}_i-\dot{\theta}_j)+\kappa \sqrt{1-\left(\frac{\nu_i-\nu_j}{2\kappa}\right)^2}\left(\theta_i-\theta_j-2\sin^{-1}\frac{\nu_i-\nu_j}{2\kappa}\right) \ge 0,\\
 \mathrm{whenever~} 0\le\theta_i-\theta_j\le \psi_1.
\end{align*}
On the other hand, by \eqref{individualdiff} and \eqref{crude-mf2}, we have
\[
m(\ddot \theta_i-\ddot\theta_j)+\dot{\theta}_i-\dot{\theta}_j\ge \nu_i-\nu_j>0,\quad \mathrm{whenever~} -\psi_1\le\theta_i-\theta_j\le 0.
\]
Thus, if we set $M=\frac{(\nu_i-\nu_j)/\kappa}{\sqrt{1-\left(\frac{\nu_i-\nu_j}{2\kappa}\right)^2}}$, we have
\[
m(\ddot \theta_i-\ddot\theta_j)+(\dot{\theta}_i-\dot{\theta}_j)+\kappa \sqrt{1-\left(\frac{\nu_i-\nu_j}{2\kappa}\right)^2}\max\left\{-M,\theta_i-\theta_j-2\sin^{-1}\frac{\nu_i-\nu_j}{2\kappa}\right\} \ge 0,\quad t>T.
\]
We use Lemma \ref{lem:cutoff-comparison} with 
\[ a=m, \quad b=1, \quad c=\kappa \sqrt{1-\left(\frac{\nu_i-\nu_j}{2\kappa}\right)^2} \quad \mbox{and} \quad y=\theta_i-\theta_j-2\sin^{-1}\frac{\nu_i-\nu_j}{2\kappa}, \]
and $4ac<4m\kappa\le 1$ to get 
\[
\liminf_{t\to\infty}(\theta_i(t)-\theta_j(t))\ge 2\sin^{-1} \Big( \frac{\nu_i-\nu_j}{2\kappa} \Big).
\]
Finally, we use $2\sin^{-1}\frac{x}{2}>x$ for $x\in (0,\frac{\pi}{2}]$ to verify the first inequality of \eqref{eq:A-linear-ordered}. A fortiori, there is a time $T_1>T$ such that 
\[ \theta_i(t)-\theta_j(t)>0 \quad \mbox{for $t\ge T_1$}. \]
On the time interval $[T_1,\infty)$, we use \eqref{individualdiff} and \eqref{crude-mf2} to get 
\begin{align*}
m(\ddot\theta_i-\ddot\theta_j)+\dot{\theta}_i-\dot{\theta}_j&\stackrel{\mathclap{ \eqref{individualdiff}, \eqref{crude-mf2}}}{\le} \nu_i-\nu_j-2\kappa (\lambda\cos \psi_1 -(1-\lambda)) \sin\frac{\theta_i-\theta_j}{2}\\
&\le \nu_i-\nu_j-\frac{2\kappa\sin(\psi_1/2)}{\psi_1} (\lambda\cos \psi_1 -(1-\lambda)) (\theta_i-\theta_j),
\end{align*}
where we used the inequality:
\[ \sin x\ge \frac{\sin(\psi_1/2)}{\psi_1/2}x \quad \mbox{for $x\in [0,\frac {\psi_1} 2]$}. \]
Again, we use Lemma \ref{lem:cutoff-comparison} with
\begin{align*}
\begin{aligned}
& a=m,~ b=1,~ c=\frac{2\kappa\sin(\psi_1/2)}{\psi_1} (\lambda\cos \psi_1 -(1-\lambda)), \quad 4ac\le 4m\kappa\le 1, \\
& y=\frac{\psi_1(\nu_i-\nu_j)}{2\kappa \sin(\psi_1/2)(\lambda\cos \psi_1 -(1-\lambda)) }-\theta_i+\theta_j,\quad M=\infty,
\end{aligned}
\end{align*}
to obtain
\[
\limsup_{t\to\infty}(\theta_i(t)-\theta_j(t))\le \frac{\psi_1}{2 \sin(\psi_1/2)(\lambda\cos \psi_1 -(1-\lambda)) }\cdot\frac{\nu_i-\nu_j}{\kappa}.
\]
Recalling that $\psi_1\in(\phi_1,\cos^{-1}\frac{1-\lambda}{\lambda})$ was arbitrary, we take $\psi_1\to \phi_1+$ to conclude
\[
\limsup_{t\rightarrow\infty}(\theta_i(t)-\theta_j(t))\le \frac{\phi_1}{2 \sin(\phi_1/2)(\lambda\cos \phi_1 -(1-\lambda)) }\cdot\frac{\nu_i-\nu_j}{\kappa}.
\]
This completes the proof of the second assertion of Theorem \ref{L4.4}.

\subsection{Proof of third assertion in Theorem \ref{L4.4} ($\Theta_\mathcal{A}$ confines $\Theta_\mathcal{B}$)}
Now we assume there is an index set $\mathcal{B}\subset [N]$ with $\mathcal{B}\supset \mathcal{A}$, that satisfies the following variant of \eqref{eq:xi-partial}:
\[\tag{\ref{eq:xi-partial-B}}
\xi(m,\kappa,\mathcal{V}_\mathcal{B},\Omega^0_\mathcal{B},\infty)=m\mathcal{D}(\mathcal{V}_\mathcal{B})+2m\kappa+\frac{\mathcal{D}(\mathcal{V}_\mathcal{B})}{2\kappa}<\frac\lambda 2\sin \ell-(1-\lambda)\sin\frac{\ell}{2}.
\]
By continuity of $\xi$ in $\eta$, we may find a time $t_2\ge t_1$ such that
\begin{equation}\label{eq:xi-partial-B-finite}
\xi(m,\kappa,\mathcal{V}_\mathcal{B},\Omega^0_\mathcal{B},t_2/m)<\frac\lambda 2\sin \ell-(1-\lambda)\sin\frac{\ell}{2}.
\end{equation}
We are to show that
\[
\sup_{t\ge 0}\mathcal{D}(\Theta_\mathcal{B}(t))<\infty.
\]
It is enough to show that, for each $i\in \mathcal{B}\setminus\mathcal{A}$, if we let $k_i\in \mathbb{Z}$ be such that
\[
\theta_i(t_2)-2k_{i}\pi\in \left[\min\Theta_ \mathcal{A}(t_2),\min \Theta_\mathcal{A}(t_2)+2\pi\right),
\]
then one of the following assertions holds:
\begin{enumerate}
\item We have
\[
\theta_i(t)-2k_{i}\pi\in \left[\min \Theta_\mathcal{A}(t),\min \Theta_\mathcal{A}(t)+\ell\right],\quad t\ge t_2.
\]
\item There exists a time $t_{2,i}'\in [t_2,\infty)$ such that
\[
\theta_i(t)-2k_{i}\pi\in \left(\min\Theta_\mathcal{A}(t)+\ell,\min\Theta_\mathcal{A}(t)+2\pi\right),\quad t_2\le t< t_{2,i}',
\]
and either
\[
\theta_i(t)-2k_{i}\pi\in \left[\min\Theta_\mathcal{A}(t),\min\Theta_\mathcal{A}(t)+\ell\right],\quad t\ge t_{2,i}',
\]
or
\[
\theta_i(t)-2k_{i}\pi\in \left[\min\Theta_\mathcal{A}(t)+2\pi,\min\Theta_\mathcal{A}(t)+\ell+2\pi\right],\quad t\ge t_{2,i}'.
\]
\item We have
\[
\theta_i(t)-2k_{i}\pi\in \left(\min\Theta_\mathcal{A}(t)+\ell,\min\Theta_\mathcal{A}(t)+2\pi\right),\quad t\ge t_2.
\]
\end{enumerate}

To prove this, we begin by observing that if there exists a time $t_2'\ge t_2$ and a wave number $k'_i\in \mathbb{Z}$ such that
\begin{equation*}
\theta_i(t_2')-2k'_i\pi\in \left[\min\Theta_\mathcal{A}(t_2'),\min\Theta_\mathcal{A}(t_2')+\ell\right],
\end{equation*}
then by applying statement (1) of Theorem \ref{L4.4} with $\mathcal{A}$ replaced by $\mathcal{A}\cup\{i\}$, $\theta_i$ replaced by $\theta_i-2k'_i\pi$, and time $t_1$ replaced by $t_2'$, and recalling that
\[
\theta_j(t'_2)\in \left[\min\Theta_\mathcal{A}(t_2'),\min\Theta_\mathcal{A}(t_2')+\ell\right],\quad j\in \mathcal{A},
\]
we have that $\mathcal{D}(\Theta_{\mathcal{A}}(t)\cup\{\theta_i(t)-2k'_i\pi\})\le \ell$ for $t\ge t'_2$, and a fortiori
\[
\theta_i(t)-2k'_i\pi\in \left[\min\Theta_\mathcal{A}(t),\min\Theta_\mathcal{A}(t)+\ell\right],\quad t\ge t'_2.
\]

We now divide into three cases, each giving the corresponding part of the trichotomy.
\begin{enumerate}
\item If
\[
\theta_i(t_2)-2k_{i}\pi\in \left[\min\Theta_\mathcal{A}(t_2),\min\Theta_\mathcal{A}(t_2)+\ell\right],
\]
then the stated result follows from our above observation with $t'_2=t_2$.
\vspace{0.1cm}
\item If
\[
\theta_i(t_2)-2k_{i}\pi\in \left(\min\Theta_\mathcal{A}(t_2)+\ell,\min\Theta_\mathcal{A}(t_2)+2\pi\right),
\]
we define
\[
t_{2,i}'\coloneqq \sup\left\{t\ge t_2:\theta_i(\tau)-2k_{i}\pi\in \left(\min\Theta_\mathcal{A}(\tau)+\ell,\min\Theta_\mathcal{A}(\tau)+2\pi\right)~\forall \tau\in [t_2,t]\right\}.
\]
If in addition $t'_{2,i}<\infty$, we must have 
\begin{align*}
\begin{aligned}
& \mbox{either}~\theta_i(t)-2k_{i}\pi\in \left[\min\Theta_\mathcal{A}(t),\min\Theta_\mathcal{A}(t)+\ell\right],\quad t\ge t'_{2,i}, \\
& \mbox{or} \quad \theta_i(t)-2k_{i}\pi\in \left[\min\Theta_\mathcal{A}(t)+2\pi,\min\Theta_\mathcal{A}(t)+\ell+2\pi\right],\quad t\ge t'_{2,i}.
\end{aligned}
\end{align*}
according to whether 
\[ \theta_i(t'_{2,i})-2k_i\pi=\min\Theta_\mathcal{A}(t'_{2,i})+\ell \quad \mbox{or} \quad \theta_i(t'_{2,i})-2k_i\pi=\min\Theta_\mathcal{A}(t'_{2,i})+2\pi \]
by applying the above observation with $t'_2=t'_{2,i}$.
\vspace{0.1cm}
\item The remaining case is that
\[
\theta_i(t_2)-2k_{i}\pi\in \left(\min\Theta_\mathcal{A}(t_2)+\ell,\min\Theta_\mathcal{A}(t_2)+2\pi\right) \quad \mbox{and} \quad t'_{2,i}=\infty.
\]
By definition, this means that
\[
\theta_i(t)-2k_{i}\pi\in \left(\min\Theta_\mathcal{A}(t)+\ell,\min\Theta_\mathcal{A}(t)+2\pi\right),\quad t\ge t'_{2,i},
\]
as desired.
\end{enumerate}

\subsection{Proof of the fourth assertion in Theorem \ref{L4.4} (Uniqueness of maximal majority cluster)}

The idea behind the existence of a maximal majority cluster is that if there are two majority clusters, then they must merge into a single majority cluster.
\begin{lemma}[Merging of two majority clusters]\label{L4.5}
Suppose that the initial data, system parameters, free parameters $\lambda, \beta, \eta>0$, $t_1\ge \eta m$, and index sets $\mathcal{A}\subset\mathcal{B}\subset[N]$ satisfy \eqref{eq:gammaellrange}, \eqref{eq:xi-partial-B}, and
\[
|\mathcal{A}|\ge \lambda N,\quad \mathcal{D}(\Theta_\mathcal{A}(t_1))\le \ell,
\]
and let $(\Theta,\Omega)$ be a global solution to \eqref{B-1}. If $\tilde{\mathcal{A}}\subset\mathcal{B}$ is such that $|\tilde{\mathcal{A}}|\ge \lambda N$ and there is a time $t_3\ge \eta m$ with
    \[
    \mathcal{D}(\Theta_{\tilde{\mathcal{A}}}(t_3))<\phi_2(\lambda, 2\xi(m,\kappa,\mathcal{V}_\mathcal{\tilde{\mathcal{A}}},\Omega^0_{\tilde{\mathcal{A}}},t_3/m)),
    \]
  then, we have
    \begin{equation}\label{eq:union-is-trapped}
    \limsup_{t\to\infty}\mathcal{D}(\Theta_{\mathcal{A}\cup\tilde{\mathcal{A}}}(t))\le \phi_1(\lambda, 2\xi(m,\kappa,\mathcal{V}_\mathcal{\mathcal{A}\cup\tilde{\mathcal{A}}},\Omega^0_{\mathcal{A}\cup\tilde{\mathcal{A}}},\infty)).
    \end{equation}
\end{lemma}
\begin{proof}
    We first begin with observation that
    \begin{equation}\label{eq:alpha-prime-stable}
    \limsup_{t\to\infty}\mathcal{D}(\Theta_{\tilde{\mathcal{A}}}(t))\le \phi_1(\lambda, 2\xi(m,\kappa,\mathcal{V}_\mathcal{\tilde{\mathcal{A}}},\Omega^0_{\tilde{\mathcal{A}}},\infty))<\ell.
    \end{equation}
    Indeed, we choose $\ell'>0$ such that
    \[
    \max\{\mathcal{D}(\Theta_{\tilde{\mathcal{A}}}(t_3)),\phi_1(\lambda, 2\xi(m,\kappa,\mathcal{V}_\mathcal{\tilde{\mathcal{A}}},\Omega^0_{\tilde{\mathcal{A}}},\eta'))\}<\ell'<\phi_2(\lambda, 2\xi(m,\kappa,\mathcal{V}_\mathcal{\tilde{\mathcal{A}}},\Omega^0_{\tilde{\mathcal{A}}},\eta')),
    \]
    where $\eta'\coloneqq t_3/m$.
    By definition of $\phi_1,\phi_2$ and the concavity of $f_\lambda$, we have \eqref{eq:xi-partial} for $\tilde{\mathcal{A}}$, $\eta'$, and $\ell'$:
    \[
    \xi(m,\kappa,\mathcal{V}_\mathcal{\tilde{\mathcal{A}}},\Omega^0_{\tilde{\mathcal{A}}},\eta')<\frac 12 f_\lambda(\ell')=\frac\lambda 2\sin \ell'-(1-\lambda)\sin\frac{\ell'}{2}.
    \]
Thus, by the first assertion of Theorem \ref{L4.4}, we have \eqref{eq:alpha-prime-stable}. Therefore, invoking statement (1) of Theorem \ref{L4.4} and \eqref{eq:alpha-prime-stable}, we may find a time $t_4\ge \max\{t_1,t_2,t_3\}$ (recall $t_2$ is chosen to satisfy \eqref{eq:xi-partial-B-finite}) such that
\begin{equation}\label{eq:alpha-tilde-alpha-individually-trapped}
\sup_{t\ge t_4}\mathcal{D}(\Theta_\mathcal{A}(t))\le \ell,\quad \sup_{t\ge t_4}\mathcal{D}(\Theta_{\tilde{\mathcal{A}}}(t))\le \ell.
\end{equation}
Since $|\mathcal{A}|,|\tilde{\mathcal{A}}|\ge \lambda N>\frac N2$, it must be that $\mathcal{A}\cap\tilde{\mathcal{A}}\neq \emptyset$, so we have
\[
\sup_{t\ge t_4}\mathcal{D}(\Theta_{\mathcal{A}\cup\tilde{\mathcal{A}}}(t))\le 2\ell.
\]
By invoking statement (1) of Theorem \ref{L4.4} once more, while noting that 
\[ |\mathcal{A}\cup\tilde{\mathcal{A}}|\ge |\mathcal{A}|\ge \lambda N, \]
we have that, in order to prove \eqref{eq:union-is-trapped}, it is enough to show that there is a time $t\ge t_4$ such that
\[
\mathcal{D}(\Theta_{\mathcal{A}\cup\tilde{\mathcal{A}}}(t))\le \ell.
\]
By the standard exit-time argument, it is enough to show that there is a positive number $\epsilon>0$ such that, if $t\ge t_4$ is a time such that 
\[ \mathcal{D}(\Theta_{\mathcal{A}\cup\tilde{\mathcal{A}}})\in [\ell,2\ell], \]
then we have
\[ D^+\mathcal{D}(\Theta_{\mathcal{A}\cup\tilde{\mathcal{A}}})<-\varepsilon. \]
Let $t\ge t_4$ be such that $\mathcal{D}(\Theta_{\mathcal{A}\cup\tilde{\mathcal{A}}})\in [\ell,2\ell]$. Recalling \eqref{eq:dini} and we compute the Dini derivative as follows:
\begin{equation}\label{eq:dini-union}
D^+\mathcal{D}(\Theta_\mathcal{A})=\max_{\substack{i,j\in \mathcal{A}\cup\tilde{\mathcal{A}}\\\theta_i=\max_{k\in {\mathcal{A}\cup\tilde{\mathcal{A}}}} \theta_k \\ \theta_j=\min_{k\in \mathcal{A}\cup\tilde{\mathcal{A}}} \theta_k}}\left(\dot\theta_i-\dot\theta_j\right)
\end{equation}
By \eqref{eq:alpha-tilde-alpha-individually-trapped} and $\mathcal{A}\cap\tilde{\mathcal{A}}\neq \emptyset$, we have the following two cases:
\begin{align*}
\begin{aligned} 
& \mbox{either}~~\min_{i\in \mathcal{A}}\theta_i(t)\le \min_{i\in \tilde{\mathcal{A}}}\theta_i(t)\le \max_{i\in \mathcal{A}}\theta_i(t)\le \max_{i\in \tilde{\mathcal{A}}}\theta_i(t), \\
& \mbox{or}~~
\min_{i\in \tilde{\mathcal{A}}}\theta_i(t)\le\min_{i\in \mathcal{A}}\theta_i(t)\le \max_{i\in \tilde{\mathcal{A}}}\theta_i(t)\le  \max_{i\in \mathcal{A}}\theta_i(t).
\end{aligned}
\end{align*}
By symmetry, we may assume that the former case holds. First, we set 
\[ \alpha\coloneqq\mathcal{D}(\Theta_{\mathcal{A}\cup\tilde{\mathcal{A}}}(t)) \quad 
\mathcal{C}\coloneqq \left\{j\in \mathcal{A}\cup\tilde{\mathcal{A}}:\min_{i\in \tilde{\mathcal{A}}}\theta_i(t)\le \theta_j(t)\le\max_{i\in \mathcal{A}}\theta_i(t)\right\},\quad \mu\coloneqq \frac{|\mathcal{C}|}{N}.
\]
Note that $\mathcal{A}\cap\tilde{\mathcal{A}}\subset\mathcal{C}$, so that
\begin{equation}\label{eq:union-bound}
|\mathcal{A}\cup\tilde{\mathcal{A}}|=|\mathcal{A}|+|\tilde{\mathcal{A}}|-|\mathcal{A}\cap\tilde{\mathcal{A}}|\ge |\mathcal{A}|+|\tilde{\mathcal{A}}|-|\mathcal{C}|\ge (2\lambda-\mu)N.
\end{equation}
Since 
\[ |\mathcal{A}\cup\tilde{\mathcal{A}}|\le N \quad \mbox{and} \quad |\mathcal{C}|\le|\mathcal{A}|\le \lambda N, \]
we have the admissibility range
\begin{equation}\label{eq:mu-range}
\mu\in[ 2\lambda -1,\lambda].
\end{equation}
By invoking Lemma \ref{L:approxaut}, for indices $i,j\in \mathcal{A}\cup\tilde{\mathcal{A}}$ such that 
\[ \theta_i(t)=\max_{k\in \mathcal{A}\cup\tilde{\mathcal{A}}}\theta_k(t) \quad \mbox{and} \quad \theta_j(t)=\min_{k\in \mathcal{A}\cup\tilde{\mathcal{A}}}\theta_k(t), \] 
we have
\begin{align}\label{eq:union-dini-estimate}
\begin{aligned}
& \dot\theta_i(t)-\dot\theta_j(t)  \le   \frac{\kappa\left(1-e^{-t/m}\right)}{N}\sum_{l\in \mathcal{A}\cup\tilde{\mathcal{A}}}\left(\sin(\theta_l(t)-\theta_i(t))-\sin(\theta_l(t)-\theta_j(t))\right)\\
& \hspace{1cm} +(\omega^0_i-\omega^0_j) e^{-t/m} + (\nu_i-\nu_j) (1 -e^{-t/m}) \\
& \hspace{1cm}  + 2\kappa \left(\mathcal{D}(\Omega^0_{\mathcal{A}\cup\tilde{\mathcal{A}}})te^{-t/m}\left(1-e^{-t/m}\right)+m (\mathcal{D}(\mathcal{V}_{\mathcal{A}\cup\tilde{\mathcal{A}}})+2\kappa)\left(1-e^{-t/m}\right)^3\right)\\
& \hspace{1cm} + \frac{N-|\mathcal{A}\cup\tilde{\mathcal{A}}
|}{N}2\kappa\sin\frac{\theta_i(t) -\theta_j(t)}{2}\left(1-e^{-t/m}\right)\\
& \hspace{1cm}  \stackrel{\mathclap{\eqref{eq:union-bound}}}{\le}   -\frac{2\kappa\left(1-e^{-t/m}\right)}{N}\sin\frac{\theta_i(t)-\theta_j(t)}{2}
\sum_{l\in \mathcal{A}\cup\tilde{\mathcal{A}}}\cos\left(\theta_l(t)-\frac{\theta_i(t)+\theta_j(t)}{2}\right)\\
& \hspace{1cm} +\mathcal{D}(\Omega^0_{\mathcal{A}\cup\tilde{\mathcal{A}}}) e^{-t/m} + \mathcal{D}(\mathcal{V}_{\mathcal{A}\cup\tilde{\mathcal{A}}}) (1 -e^{-t/m}) \\
& \hspace{1cm}  + 2\kappa \left(\mathcal{D}(\Omega^0_{\mathcal{A}\cup\tilde{\mathcal{A}}})te^{-t/m}\left(1-e^{-t/m}\right)+m (\mathcal{D}(\mathcal{V}_{\mathcal{A}\cup\tilde{\mathcal{A}}})+2\kappa)\left(1-e^{-t/m}\right)^3\right)\\
& \hspace{1cm}  +2\left(1-2\lambda+\mu\right)\kappa\sin\frac{\theta_i(t) -\theta_j(t)}{2}\left(1-e^{-t/m}\right)\\
&  \hspace{1cm}  \le  -2\kappa \left(1-e^{-t/m}\right)\sin\frac{\theta_i(t) -\theta_j(t)}{2}\left(\frac 1N\sum_{l\in \mathcal{A}\cup\tilde{\mathcal{A}}}\cos\left(\theta_l(t)-\frac{\theta_i(t)+\theta_j(t)}{2}\right)-(1-2\lambda+\mu)\right) \\
& \hspace{1cm}  +2\kappa \left(1-e^{-t/m}\right) \xi(m,\kappa,\mathcal{V}_{\mathcal{A}\cup\tilde{\mathcal{A}}},\Omega^0_{\mathcal{A}\cup\tilde{\mathcal{A}}},t_4/m).
\end{aligned}
\end{align}
Note that for $l\in \mathcal{C}$, we have
\[
\theta_i(t)-\ell\le \min_{k\in \tilde{\mathcal{A}}}\theta_k(t)-\ell\le \min_{k\in \tilde{\mathcal{A}}}\theta_k(t)\le \theta_l(t)\le \max_{k\in \mathcal{A}}\theta_k(t)\le \min_{k\in \mathcal{A}}\theta_k(t)+\ell=\theta_j(t)+\ell,
\]
so that
\[
-\ell+\frac \alpha 2\le -\ell+\frac{\theta_i(t)-\theta_j(t)}{2}\le \theta_l(t)-\frac{\theta_i(t)+\theta_j(t)}{2}\le \ell-\frac{\theta_i(t)-\theta_j(t)}{2}=\ell-\frac \alpha 2.
\]
Hence, we have
\[
\cos\left(\theta_l(t)-\frac{\theta_i(t)+\theta_j(t)}{2}\right)\ge \cos\left(\ell-\frac\alpha 2\right),\quad l\in \mathcal{C}.
\]
On the other hand, we have the following straightforward bound
\[
\cos\left(\theta_l(t)-\frac{\theta_i(t)+\theta_j(t)}{2}\right)\ge \cos\left(\frac{\theta_i(t)-\theta_j(t)}2\right)=\cos \frac\alpha 2,\quad l\in (\mathcal{A}\cup\tilde{\mathcal{A}})\setminus\mathcal{C}
\]
to derive
\begin{align*}
&\frac 1N\sum_{l\in \mathcal{A}\cup\tilde{\mathcal{A}}}\cos\left(\theta_l(t)-\frac{\theta_i(t)+\theta_j(t)}{2}\right)\\
& \hspace{0.5cm} \ge \frac 1N\sum_{l\in \mathcal{C}}\cos\left(\theta_l(t)-\frac{\theta_i(t)+\theta_j(t)}{2}\right)+\frac 1N\sum_{l\in (\mathcal{A}\cup\tilde{\mathcal{A}})\setminus\mathcal{C}}\cos\left(\theta_l(t)-\frac{\theta_i(t)+\theta_j(t)}{2}\right)\\
& \hspace{0.5cm} \ge \frac{|\mathcal{C}|}N\cos\left(\ell-\frac\alpha 2\right)+\frac{|\mathcal{A}\cup\tilde{\mathcal{A}}|-|\mathcal{C}|}N \cos \frac\alpha 2\\
& \hspace{0.5cm} \ge \mu\cos\left(\ell-\frac\alpha 2\right)+2(\lambda-\mu)\cos \frac\alpha 2.
\end{align*}
Thus, we continue the estimate of \eqref{eq:union-dini-estimate} to obtain
\begin{align}\label{eq:union-dini-estimate-2}
\begin{aligned}
\dot\theta_i(t)-\dot\theta_j(t) &\le -2\kappa \left(1-e^{-t/m}\right)\sin\frac\alpha 2\left(\mu\cos\left(\ell-\frac\alpha 2\right)+2(\lambda-\mu)\cos \frac\alpha 2-(1-2\lambda+\mu)\right)\\
&+2\kappa \left(1-e^{-t/m}\right)\xi(m,\kappa,\mathcal{V}_{\mathcal{A}\cup\tilde{\mathcal{A}}},\Omega^0_{\mathcal{A}\cup\tilde{\mathcal{A}}},t_4/m).
\end{aligned}
\end{align}
We claim that
\begin{align}
\begin{aligned} \label{eq:union-collapse-key-estimate}
& \sin\frac\alpha 2\left(\mu\cos\left(\ell-\frac\alpha 2\right)+2(\lambda-\mu)\cos \frac\alpha 2-(1-2\lambda+\mu)\right) \\
& \hspace{2cm} \ge \sin\frac\ell 2\left(\lambda \cos \frac \ell 2-(1-\lambda)\right)=\frac 12 f_\lambda(\ell).
\end{aligned}
\end{align}
For the moment, we postpone the proof of this estimate to the end of this proof. \newline

Suppose \eqref{eq:union-collapse-key-estimate} holds. Then, it follows from \eqref{eq:union-dini-estimate-2} that
\[
\dot\theta_i(t)-\dot\theta_j(t)\le-\kappa \left(1-e^{-t/m}\right)\left(f_\lambda(\ell)-2\xi(m,\kappa,\mathcal{V}_{\mathcal{A}\cup\tilde{\mathcal{A}}},\Omega^0_{\mathcal{A}\cup\tilde{\mathcal{A}}},t_4/m)\right)<0,
\]
which is negative since
\[
f_\lambda(\ell)~\stackrel{\mathclap{\eqref{eq:xi-partial-B-finite}}}{\ge} ~2\xi(\mathcal{B},t_2/m)\ge2\xi(\mathcal{B},t_4/m)\ge 2\xi(m,\kappa,\mathcal{V}_{\mathcal{A}\cup\tilde{\mathcal{A}}},\Omega^0_{\mathcal{A}\cup\tilde{\mathcal{A}}},t_4/m),
\]
where $ \xi(\mathcal{B},s) = \xi(m,\kappa,\mathcal{V}_\mathcal{B},\Omega^0_\mathcal{B},s)$ for $s=t_2/m,t_4/m$.
Invoking \eqref{eq:dini-union}, we have
\begin{align*}
    \left.D^+\right|_t\mathcal{D}(\Theta_\mathcal{A})\le&-\kappa \left(1-e^{-t_4/m}\right)\left(f_\lambda(\ell)-2\xi(m,\kappa,\mathcal{V}_{\mathcal{A}\cup\tilde{\mathcal{A}}},\Omega^0_{\mathcal{A}\cup\tilde{\mathcal{A}}},t_4/m)\right)<0.
\end{align*}
Thus, we have shown that whenever $\mathcal{D}(\Theta_{\mathcal{A}\cup\tilde{\mathcal{A}}})\in [\ell,2\ell]$ on the time interval $[t_4,\infty)$, it has negative upper Dini (time) derivative: $D^+\mathcal{D}(\Theta_{\mathcal{A}\cup\tilde{\mathcal{A}}})<-\epsilon$ for some fixed $\epsilon$:
\[
\epsilon\coloneqq \kappa \left(1-e^{-t_4/m}\right)\left(f_\lambda(\ell)-2\xi(m,\kappa,\mathcal{V}_{\mathcal{A}\cup\tilde{\mathcal{A}}},\Omega^0_{\mathcal{A}\cup\tilde{\mathcal{A}}},t_4/m)\right)>0.
\]
As per the aforementioned argument, this verifies \eqref{eq:union-is-trapped}. Now, we return to the proof of $\eqref{eq:union-collapse-key-estimate}$. \newline

\noindent {\it Proof of \eqref{eq:union-collapse-key-estimate}}: Recall the constraints given by \eqref{eq:gammaellrange}, $\mathcal{D}(\Theta_{\mathcal{A}\cup\tilde{\mathcal{A}}}(t))\in [\ell,2\ell]$, and \eqref{eq:mu-range}:\footnote{Because the inequality \eqref{eq:union-collapse-key-estimate} that we are trying to prove is not strict, we are allowed to take the closure of the conditions, i.e., we can make every open interval closed.}
\[
\lambda\in [1/2,1],\quad \ell\in [0,2\cos^{-1}(1/\lambda-1)], \quad \alpha\in [\ell,2\ell],\quad \mu\in [2\lambda-1,\lambda].
\]
These constraints are equivalent to
\[
\ell\in [0,\pi],\quad \alpha\in [\ell,2\ell],\quad \lambda\in \left[\frac{1}{1+\cos(\ell)/2},1\right],\quad \mu\in [2\lambda-1,\lambda],
\]
where we choose the variables in the order of $\ell$, $\alpha$, $\lambda$, and then $\mu$, within the above conditions. Note that, for each fixed choice of $\ell$ and $\alpha$, the inequality \eqref{eq:union-collapse-key-estimate} is linear in $\lambda$ and $\mu$, whose domain
\[
\left\{(\lambda,\mu)\in \mathbb{R}^2:\lambda\in \left[\frac{1}{1+\cos(\ell/2)},1\right],\quad \mu\in [2\lambda-1,\lambda]\right\}
\]
forms a closed solid triangle in $\mathbb{R}^2$ with vertices
\[
(\lambda,\mu)=\left(\frac{1}{1+\cos(\ell/2)},~\frac{1-\cos(\ell/2)}{1+\cos(\ell/2)}\right), \quad \left(\frac{1}{1+\cos(\ell/2)},~\frac{1}{1+\cos(\ell/2)}\right), \quad \left(1,~1\right).
\]
Thus, it is enough to prove \eqref{eq:union-collapse-key-estimate} at these extreme points; the corresponding inequalities are
\begin{equation}\label{eq:casework-1}
\sin\frac\alpha 2\frac{(1-\cos(\ell/2))\cos\left(\ell-\frac\alpha 2\right)+2\cos(\ell/2)\cos \frac\alpha 2}{1+\cos(\ell/2)}\ge 0,
\end{equation}
\begin{equation}\label{eq:casework-2}
\sin\frac\alpha 2\left(\frac{1+\cos\left(\ell-\frac\alpha 2\right)}{1+\cos(\ell/2)}-1\right)\ge 0,
\end{equation}
and
\begin{equation}\label{eq:casework-3}
\sin\frac\alpha 2\cos\left(\ell-\frac\alpha 2\right)\ge \sin\frac\ell 2 \cos \frac \ell 2,
\end{equation}
respectively. \newline

\noindent $\bullet$ Case A.1 (verification of \eqref{eq:casework-1}): Since $\sin \frac\alpha 2\ge 0$, it is enough to show, for $\ell\in [0,\pi]$ and $\alpha\in [\ell,2\ell]$,
\[
(1-\cos(\ell/2))\cos\left(\ell-\frac\alpha 2\right)+2\cos(\ell/2)\cos \frac\alpha 2\ge 0.
\]
Observe that
\begin{align*}
&(1-\cos\frac \ell 2)\cos\left(\ell-\frac\alpha 2\right)+2\cos\frac \ell 2 \cos \frac\alpha 2\\
& \hspace{0.5cm} =(1-\cos\frac\ell 2)\cos\left(\frac \ell 2-\frac{\alpha-\ell} 2\right)+2\cos\frac \ell 2 \cos \left(\frac \ell 2+\frac{\alpha-\ell} 2\right)\\
& \hspace{0.5cm} =(1+\cos\frac\ell 2)\cos\frac \ell 2\cos\frac{\alpha-\ell} 2+(1-3\cos\frac\ell 2)\sin\frac \ell 2\sin\frac{\alpha-\ell} 2.
\end{align*}
Since $\frac \ell 2,\frac{\alpha-\ell}{2}\in [0,\frac\pi 2]$, this is nonnegative if $1-3\cos\frac\ell 2\ge 0$. \newline

\noindent If $1-3\cos\frac\ell 2<0$, then the $\alpha$-derivative is
\[
-(1+\cos\frac\ell 2)\cos\frac \ell 2\sin\frac{\alpha-\ell} 2+(1-3\cos\frac\ell 2)\sin\frac \ell 2\cos\frac{\alpha-\ell} 2
\]
and is nonpositive. So, for fixed $\ell\in [0,\pi]$, the given expression is minimized at $\alpha=2\ell$, at which it becomes
\begin{align*}
&(1+\cos\frac\ell 2)\cos^2\frac \ell 2+(1-3\cos\frac\ell 2)\sin^2\frac \ell 2\\
& \hspace{0.5cm} =(1+\cos\frac\ell 2)\cos^2\frac \ell 2+(1-3\cos\frac\ell 2)(1-\cos^2\frac \ell 2) =(1+\cos\frac\ell 2)\left(\cos^2\frac \ell 2+(1-3\cos\frac\ell 2)(1-\cos\frac\ell 2)\right)\\
& \hspace{0.5cm} =(1+\cos\frac\ell 2)\left(1-4\cos\frac\ell 2+4\cos^2\frac\ell 2\right) =(1+\cos\frac\ell 2)\left(1-2\cos\frac\ell 2\right)^2\ge 0.
\end{align*}
This proves \eqref{eq:casework-1}. \newline

\noindent $\bullet$ Case A.2 (verification of \eqref{eq:casework-2}):~Since $\sin\frac\alpha 2\ge 0$, it is enough to show
\[
\cos\left(\ell-\frac\alpha 2\right)-\cos\frac\ell 2\ge 0.
\]
This follows from the fact that $\cos$ is decreasing on $[0,\frac\pi 2]$ and
\[
0\le \ell-\frac\alpha 2\le \frac\ell 2\le \frac\pi 2.
\]

\noindent $\bullet$ Case A.3 (verification of \eqref{eq:casework-3}):  The last estimate \eqref{eq:casework-3} follows from $\alpha-\ell\in [0,\pi]$ and
\[
\sin\frac{\alpha}{2}\cos\left( \ell - \frac{\alpha}{2} \right) =\frac{1}{2}(\sin \ell + \sin(\alpha - \ell) ) \ge \frac{1}{2}\sin \ell = \sin\frac{\ell}{2} \cos\frac{\ell}{2}.
\]
This finally completes the proof of Lemma \ref{L4.5}.
\end{proof}
Now we are ready to proceed with the proof of statement (4) of Theorem \ref{L4.4}. Consider the set 
\begin{align*}
{\mathfrak S}\coloneqq\left\{\tilde{\mathcal{A}}\subset\mathcal{B}:|\tilde{\mathcal{A}}|\ge \lambda N\mathrm{~and~}\exists t\ge \eta m\exists \vec{k}\in \mathbb{Z}^{\tilde{\mathcal{A}}}~\mathcal{D}(\Theta_{\tilde{\mathcal{A}}}(t)-2\pi \vec{k})<\phi_2(\lambda, 2\xi(m,\kappa,\mathcal{V}_\mathcal{\tilde{\mathcal{A}}},\Omega^0_{\tilde{\mathcal{A}}},t/m))\right\}.
\end{align*}
Then Lemma \ref{L4.5} tells us that whenever $\tilde{\mathcal{A}}_1,\tilde{\mathcal{A}}_2\in {\mathfrak S}$, then $\tilde{\mathcal{A}}_1\cup\tilde{\mathcal{A}}_2\in {\mathfrak S}$. On the other hand, $\mathcal{A}\in {\mathfrak S}$ by assumption. Therefore, ${\mathfrak S}$ contains an element $\mathcal{A}_{\mathrm{max}}$ that is maximal with respect to inclusion, i.e., we have $\mathcal{A}_{\mathrm{max}}\in {\mathfrak S}$, and whenever $\tilde{\mathcal{A}}\in {\mathfrak S}$ we have $\tilde{\mathcal{A}}\subset \mathcal{A}_{\mathrm{max}}$. Of course, we have $\mathcal{A}\subset \mathcal{A}_{\mathrm{max}}\subset \mathcal{B}$. Next, we show that the index set $\mathcal{A}_\mathrm{max}$ that we constructed has properties (4)(a) and (4)(b). \newline

\noindent  $\bullet$~(Verification of property (4)(a)):~From $\mathcal{A}_{\mathrm{max}}\in {\mathfrak S}$ and statement (1) of Theorem \ref{L4.4}, we have that for some $\vec{k}\in \mathbb{Z}^{\mathcal{A}_{\mathrm{max}}}$,
\[
\limsup_{t\to\infty}\mathcal{D}(\Theta_{\mathcal{A}_{\mathrm{max}}}(t)-2\pi \vec{k}) \le \phi_1(\lambda, 2\xi(m,\kappa,\mathcal{V}_{\mathcal{A}_\mathrm{max}},\Omega^0_{\mathcal{A}_\mathrm{max}},\infty)) \le \phi_1(\lambda, 2\xi(m,\kappa,\mathcal{V}_{\mathcal{B}},\Omega^0_{\mathcal{B}},\infty)).
\]
Since $\mathcal{D}(\Theta_\mathcal{A}(t))\le \ell<\pi$ for $t\ge t_1$, we can see that the vector $\vec{k}\in \mathbb{Z}^{\mathcal{A}_{\mathrm{max}}}$ must have constant $\mathcal{A}$-entries, and thus, by adding some integer multiple of the all-ones vector to $\vec{k}$, we may assume the $\mathcal{A}$-entries of $\vec{k}$ are zero. This proves that $\mathcal{A}_\mathrm{max}$ has property (4)(a). \newline

\vspace{0.2cm}

\noindent  $\bullet$~(Verification of property (4)(b)):~If $\mathcal{A}_\mathrm{max}=\mathcal{B}$, then there is nothing to prove, so  we may assume $\mathcal{A}_\mathrm{max}\subsetneq \mathcal{B}$. For any $i\in \mathcal{B}\setminus\mathcal{A}_\mathrm{max}$, by construction of $\mathcal{A}_\mathrm{max}$ as the maximal element of ${\mathfrak S}$ up to inclusion, we have $\mathcal{A}_\mathrm{max}\cup\{i\}\notin {\mathfrak S}$, which, by definition of ${\mathfrak S}$, tells us that for all $t\ge \eta m$, we have
\[
\min_{k\in \mathbb{Z}}\mathcal{D}(\Theta_{\mathcal{A}_\mathrm{max}}(t)\cup \{\theta_i(t)-2\pi k\})\ge \phi_2(\lambda, 2\xi(m,\kappa, \mathcal{V}_{\mathcal{A}_\mathrm{max}\cup\{i\}},\Omega^0_{\mathcal{A}_\mathrm{max}\cup\{i\}},t/m)).
\]
Thus we have
\begin{align*}
&\liminf_{t\to\infty}\min_{k\in \mathbb{Z}}\mathcal{D}(\Theta_{\mathcal{A}_\mathrm{max}}(t)\cup \{\theta_i(t)-2\pi k\}) \\
& \hspace{1cm} \ge \phi_2(\lambda, 2\xi(m,\kappa, \mathcal{V}_{\mathcal{A}_\mathrm{max}\cup\{i\}},\Omega^0_{\mathcal{A}_\mathrm{max}\cup\{i\}},\infty)) \ge \phi_2(\lambda, 2\xi(m,\kappa, \mathcal{V}_{\mathcal{B}},\Omega^0_{\mathcal{B}},\infty)).
\end{align*}
We take the minimum over the finite set $\mathcal{B}\setminus \mathcal{A}_\mathrm{max}$ to see
    \[
    \liminf_{t\to\infty}\min_{\substack{i\in \mathcal{B}\setminus\mathcal{A}_\mathrm{max}\\ k\in \mathbb{Z}}}\mathcal{D}(\Theta_{\mathcal{A}_\mathrm{max}}(t)\cup \{\theta_i(t)-2\pi k\})\ge \phi_2(\lambda, 2\xi(m,\kappa, \mathcal{V}_{\mathcal{B}},\Omega^0_{\mathcal{B}},\infty)).
    \]
Now, we have a triangle inequality:
\[
\min_{\substack{i\in \mathcal{B}\setminus \mathcal{A}_{\mathrm{max}}\\ j\in \mathcal{A}_{\mathrm{max}}\\ k\in \mathbb{Z}}}|\theta_i(t)-2\pi k-\theta_j(t)|\ge \min_{\substack{i\in \mathcal{B}\setminus\mathcal{A}_\mathrm{max}\\ k\in \mathbb{Z}}}\mathcal{D}(\Theta_{\mathcal{A}_\mathrm{max}}(t)\cup \{\theta_i(t)-2\pi k\}) - \mathcal{D}(\Theta_{\mathcal{A}_\mathrm{max}}(t)),
\]
to which, if we apply \eqref{eq:cluster-maximality} and \eqref{eq:max-cluster}:
    \[
    \limsup_{t\rightarrow\infty}\mathcal{D}(\Theta_{\mathcal{A}_\mathrm{max}}(t))\le \phi_1\left(\lambda, 2\xi(m,\kappa, \mathcal{V}_{\mathcal{B}},\Omega^0_{\mathcal{B}},\infty)\right),
    \]
 the we obtain \eqref{eq:cluster-separation}:
 \[
  \liminf_{t\to\infty}\min_{\substack{i\in \mathcal{B}\setminus \mathcal{A}_{\mathrm{max}}\\ j\in \mathcal{A}_{\mathrm{max}}\\ k\in \mathbb{Z}}}|\theta_i(t)-2\pi k-\theta_j(t)| \ge \phi_2\left(\lambda, 2\xi(m,\kappa, \mathcal{V}_{\mathcal{B}},\Omega^0_{\mathcal{B}},\infty)\right)-\phi_1\left(\lambda, 2\xi(m,\kappa, \mathcal{V}_{\mathcal{B}},\Omega^0_{\mathcal{B}},\infty)\right).
\]
\vspace{0.2cm}
\noindent  $\bullet$~(Verification of property (4)(c)):~Now we show that $\mathcal{A}_\mathrm{max}$ is the unique index set with these properties. Assume $\mathcal{A}_\mathrm{max}'\subset \mathcal{B}$ satisfies $\mathcal{A}_\mathrm{max}'\supset \mathcal{A}$ and has properties (4)(a) and (4)(b). By property (4)(a), we have (after suitable $2\pi$-translations)
\begin{align*}
    \limsup_{t\rightarrow\infty}\mathcal{D}(\Theta_{\mathcal{A}_\mathrm{max}'}(t))\stackrel{\mathclap{\eqref{eq:max-cluster}}}{\le} \phi_1\left(\lambda, 2\xi(m,\kappa, \mathcal{V}_{\mathcal{B}},\Omega^0_{\mathcal{B}},\infty)\right)< \phi_2\left(\lambda, 2\xi(m,\kappa, \mathcal{V}_{\mathcal{A}_\mathrm{max}'},\Omega^0_{\mathcal{A}_\mathrm{max}'},\infty)\right)
\end{align*}
from which it follows that $\mathcal{A}_\mathrm{max}'\in {\mathfrak S}$. Hence, by the maximality of $\mathcal{A}_\mathrm{max}$, we have
\[ \mathcal{A}_\mathrm{max}'\subset \mathcal{A}_\mathrm{max}. \]
But for any $i\in \mathcal{A}_\mathrm{max}\setminus\mathcal{A}_\mathrm{max}'$, we have
\[
\phi_1(\lambda, 2\xi(m,\kappa, \mathcal{V}_{\mathcal{B}},\Omega^0_{\mathcal{B}},\infty))\stackrel{\mathclap{\eqref{eq:max-cluster}}}{\ge}\limsup_{t\rightarrow\infty}\mathcal{D}(\Theta_{\mathcal{A}_\mathrm{max}}(t))\ge \liminf_{t\to\infty}\mathcal{D}(\Theta_{\mathcal{A}_\mathrm{max}'\cup\{i\}}(t)) \stackrel{\mathclap{\eqref{eq:cluster-maximality}}}{\ge} \phi_2(\lambda, 2\xi(m,\kappa, \mathcal{V}_{\mathcal{B}},\Omega^0_{\mathcal{B}},\infty)),
\]
which gives a contradiction. Thus, the condition $i\in \mathcal{A}_\mathrm{max}\setminus\mathcal{A}_\mathrm{max}'$ must be vacuous, i.e., 
\[ \mathcal{A}_\mathrm{max}\setminus\mathcal{A}_\mathrm{max}'=\emptyset. \]
Therefore $\mathcal{A}_\mathrm{max}'=\mathcal{A}_\mathrm{max}$, proving the uniqueness of $\mathcal{A}_\mathrm{max}$. Assertion (4)(c) for $\mathcal{A}_\mathrm{max}$ follows directly from the assertion (2).

%
%
%
%
\section{Proof of Theorem \ref{T3.1}} \label{sec:5}
\setcounter{equation}{0} 
In this section, we provide a proof of Theorem \ref{T3.1}. Since the proof involves a series of technical estimates, we briefly outline a proof strategy in the following three stages. Let $\delta \in (0, 1)$ and $\eta>0$ be free parameters. 
\begin{itemize}
\item Stage A (Strict positivity of $R$ in an initial layer):~We show, via Lemma \ref{L2.4} (which in turn is based on the finite propagation speed guarantee of Lemma \ref{L2.2}), that the order parameter $R$ stays strictly positive in the initial layer $[0,\eta m]$. That is, when $m$ is small enough to satisfy \eqref{C-1}-(${\mathcal F}_1$), we have
\[ R(t) \geq \delta R^0, \quad \forall~t \in [0, \eta m]. \]
 For details, we refer to Lemma \ref{L4.1}.
 \vspace{0.2cm}
\item
Stage B (Finite time condensation of a majority cluster):~We show, using the quasi-monotonicity of the order parameter, that after some finite time, a majority cluster forms spontaneously from the oscillators. See Synchronization Mechanism III (subsection \ref{subsec:1.5}) for a heuristic explanation.
\begin{itemize}
\item Stage B.1 (Finite-time condensation): with $\delta\in (0,1)$ as above, we show that there exists a positive time $t_e \in [\eta m, \infty)$ such that
\[ R(t)\ge \delta R^0>0,~ \forall t\in[0,t_e]\quad \mbox{and}\quad \Delta(t_e) \leq \left( \frac{\xi(\eta)}{\delta R^0} \right)^2,
\]
where $\Delta$ is the functional defined in \eqref{eq:Delta} and $\xi=\xi(\eta)$ is the constant defined in \eqref{B-14-0-1}. For details, we refer to Lemma \ref{L4.2}. 
\vspace{0.2cm}
\item Stage B.2 (Condensation gives a majority cluster): The above condition on $R(t_e)$ and $\Delta(t_e)$, in conjunction with condition \eqref{C-1}-(${\mathcal F}_2$) on free parameters $\lambda > \frac{1}{2}$ and $\ell
$, necessitates the existence of a majority cluster $\Theta_\mathcal{A}=(\theta_i)_{i\in \mathcal{A}}$ and integers $k_i$, $i\in \mathcal{A}$, such that
\[
|\mathcal{A}| \geq \lambda N \quad \mbox{and} \quad \mathcal{D}\left((\theta_i(t_e) - 2k_i\pi)_{i\in \mathcal{A}}\right)<\ell.
\]
For details, we refer to Lemma \ref{L4.3}. 
\end{itemize}
\vspace{0.2cm}
\item Stage C (Stability of the majority cluster and relaxation):~We show, using Synchronization Mechanism II (subsection \ref{subsec:1.4}), that the majority cluster $\Theta_\mathcal{A}$ of arclength $<\ell$ formed at $t= t_*$ is maintained for the rest of time, i.e.,
\[ 
\mathcal{D}\left((\theta_i(t_e) - 2k_i\pi)_{i\in \mathcal{A}}\right)<\ell,\quad t\geq t_*.
\]
By Synchronization Mechanism I (subsection \ref{subsec:1.6}), that asymptotic phase-locking occurs. More precisely, we use Theorem \ref{L4.4}, which is applicable since we are in the regime of \eqref{C-1}-(${\mathcal F}_3$).
\end{itemize}

\subsection{Stage A(Initial layer)} \label{sec:5.1}
As described after \eqref{eq:omega-decomposition} in Subsection \ref{sec:2.1}, the initial layer $[0,\eta m]$ is a short time period during which we are exposed to a potential adversarial attack. By \eqref{C-1}-(i), $m$ is small enough that the order parameter does not change much during this time. The following lemma quantifies this.
\begin{lemma}{(Behavior of $R$ in the initial layer)} \label{L4.1}
Suppose that the parameters $\kappa>0$ and $m>0$, initial data $\Theta^0$ and $\Omega^0$, intrinsic frequencies $\mathcal{V}$, and additional parameters $\eta>0$ and $\delta\in (0,1)$ satisfy \eqref{C-1}-(${\mathcal F}_1$). Let $\Theta$ be a global solution to \eqref{B-1}. Then $R$ satisfies
\[
R(t) \geq \delta R^0 > 0, \quad \forall~ t \in [0, \eta m].
\]
\end{lemma}
\begin{proof}  For $t \in [0, \eta m]$, it follows from \eqref{B-14-000} and \eqref{C-1}-(${\mathcal F}_1$) that
\[
R(t)\ge R^0-\zeta(\eta)\ge R^0 - (1-\delta) R^0 = \delta R^0.
\]
\end{proof}

\subsection{Stage B (Condensation of a majority cluster)}\label{subsec:condensation}

Next, we study the dynamics of $R$ in the time interval $[\eta m, \infty)$ after the initial layer.  In this time zone, it follows from \eqref{B-14-00} of Lemma \ref{L2.4} that 
\begin{equation} \label{D-12-1}
\dot R \geq \kappa\sqrt{\Delta} \left(1-e^{-t/m}\right) \left( R \sqrt{\Delta} - \xi(\eta) \right),\quad t \geq \eta m.
\end{equation}
Note that for $t>0$,
\[
\kappa\sqrt{\Delta} \left(1-e^{-t/m}\right) \left( R \sqrt{\Delta} - \xi(\eta) \right) > 0 \quad \Longleftrightarrow \quad  \Delta(t) > \left( \frac{\xi(\eta)}{R(t)}  \right)^2.   \]
Thus,  as long as $R$ and $\Delta$ satisfy 
\begin{equation}\label{eq:monotone-condition}
R(t) \geq \delta R^0  \quad \mbox{and} \quad \Delta(t) > \left(  \frac{\xi(\eta)}{ \delta R^0 } \right)^2,
\end{equation}
we have $\Delta(t)  > \left(  \frac{\xi(\eta)}{ R(t) }  \right)^2$ and hence ${\dot R}(t) > 0$.

As $R$ is bounded in $[0,1]$, the monotonic behavior of $R$ must halt either in finite time or as $t\to\infty$. So, heuristically, condition \eqref{eq:monotone-condition} should be untenable in the long run. The following lemma tells us that indeed condition \eqref{eq:monotone-condition} fails to hold, in fact in finite time, with the second part of \eqref{eq:monotone-condition} being the point of failure due to the monotonicity of $R$.

\begin{lemma}\label{L4.2}
Suppose the parameters $\kappa,m,\Theta^0,\Omega^0,\mathcal{V},\eta,\delta$ satisfy \eqref{C-1}-(${\mathcal F}_1$),
and let $(\Theta, \Omega)$ be a global solution to \eqref{B-1}.  Then, there exists a time $t_0 \in[\eta m,\infty)$ such that
\begin{equation}\label{D-18}
R(t)\ge \delta R^0 >0,~ \forall t\in[0,t_0],\quad \mbox{and}\quad \Delta(t_0) \leq \left( \frac{\xi(\eta)}{\delta R^0} \right)^2.
\end{equation}
\end{lemma}
\begin{proof}
It follows from \eqref{C-1}-(${\mathcal F}_1$) and Lemma \ref{L4.1} that 
\[ 
(1-\delta) R^0 \geq \zeta(\eta) \quad \mbox{and} \quad   R(t) \geq \delta R^0, \quad \forall~t \in [0, \eta m].
\]
Now, we consider  two cases:
\begin{equation}\label{D-20}
\mbox{either} \quad \Delta(\eta m) \leq   \left( \frac{\xi(\eta)}{\delta R^0} \right)^2 \quad \mbox{or} \quad  \Delta(\eta m)>  \left( \frac{\xi(\eta)}{\delta R^0} \right)^2.
\end{equation}
For the former case of \eqref{D-20}, we simply take 
\[ t_0 = \eta m. \]
This choice satisfies all the desired estimates of \eqref{D-18}.  \newline

\vspace{0.2cm}

\noindent For the latter case of \eqref{D-20}, we use $R(\eta m) \geq \delta R^0$ to see
\begin{equation} \label{D-20-1}
 \Delta(\eta m)>  \left( \frac{\xi(\eta)}{\delta R^0} \right)^2 \geq   \left( \frac{\xi(\eta)}{R(\eta m)} \right)^2.
 \end{equation}
Now we define a set ${\mathcal T}_1$ and its supremum:
\begin{align}
\begin{aligned} \label{D-21}
& {\mathcal T}_1 \coloneqq \left\{ t>\eta m\,:\,R(s)>0,\quad \Delta(s) > \left( \frac{\xi(\eta)}{R(s)} \right)^2,\quad \eta m\leq s\leq t \right\}, \quad   t_1^\infty \coloneqq \sup {\mathcal T}_1.
\end{aligned}
\end{align}
Note that 
\[ t \in {\mathcal T}_1 \quad \Longleftrightarrow \quad {\dot R}(s) > 0, \quad \forall~s \in [\eta m, t]. \]
By \eqref{D-20-1} and the continuity of $\Delta$, one has 
\[  \eta m < t_1^{\infty} \leq \infty. \]
Thus,
\begin{align}\label{D-22}
\dot R(t) > 0,\quad t\in(\eta m, t_1^\infty).
\end{align}
Now we consider two cases:
\[
\mbox{either}~~t_1^\infty <\infty\quad \mbox{or}\quad t_1^\infty = \infty.
\]
Either case will lead to the existence of a $t_0$ satisfying \eqref{D-18}. \newline

\noindent $\diamond$~Case A: Suppose $t_1^\infty<\infty$.  Then, by the definition of $t_1^\infty$, we have
\begin{equation} \label{D-22-1}
 \Delta(t_1^{\infty})  = \left( \frac{\xi(\eta)}{ R(t_1^\infty)} \right)^2 \quad \mbox{and} \quad  {\dot R}(t) > 0, \quad \forall~t \in (\eta m, t_1^\infty). 
\end{equation}
In particular, we have
\begin{align}\label{D-23}
R(t_1^\infty)> R(t) \geq R(\eta m) \geq \delta R^0, \quad t \in [\eta m, t_1^\infty).
\end{align}
We claim that the following choice satisfies \eqref{D-18}:
\[ t_0 \coloneqq t_1^{\infty}. \]
Indeed, the first relation in \eqref{D-18} holds due to \eqref{D-22} and \eqref{D-23}. To check the second relation, observe that by \eqref{D-22-1} and \eqref{D-23}, one has 
\begin{align*}
\Delta (t_1^\infty) =\left( \frac{ \xi(\eta)}{R(t_1^\infty)}\right)^2 < \left( \frac{\xi(\eta)}{ \delta R^0}\right)^2.
\end{align*}

\vspace{0.5cm}

\noindent $\diamond$ Case B:~ Suppose $t_1^\infty=\infty$. In this case, we have
\begin{equation*}
\dot R(t) > 0, \quad t \in [\eta m, \infty).
\end{equation*}
This yields
\begin{equation*}
R(t) > R(\eta m) \geq \delta R^0, \quad t \in [\eta m, \infty)
\end{equation*}
and
\begin{equation}\label{eq:Rfraclim}
\lim_{t\to\infty}\left(\frac{\xi(\eta)}{R(t)}-\frac{\xi(\eta)}{\delta R^0}\right)<0.
\end{equation}
On the other hand, we have, for $t \in(\eta m,\infty),$
\begin{equation*} 
\dot R(t) \stackrel{\eqref{D-12-1}}{\ge}  \kappa R(t) (1-e^{-\eta}) \sqrt{\Delta} \left( \sqrt{\Delta(t)} - \frac{\xi(\eta)}{ R(t)} \right)  \stackrel{\eqref{D-21}}{>}\kappa\left(1-e^{-\eta}\right)\xi(\eta) \left( \sqrt{\Delta(t)} - \frac{\xi(\eta)}{ R(t)} \right) > 0.
\end{equation*}
Integrating both sides from $t=\eta m$ to $t=\infty$ and using the boundedness of $R\in[0,1]$,
\begin{align*}
1\geq  R(\infty) - R(\eta m) = \int_{\eta m}^\infty \dot R(s) ds > \kappa \left(1-e^{-\eta}\right)\xi(\eta) \int_{\eta m}^\infty \underbrace{\left( \sqrt{\Delta(s)} - \frac{\xi(\eta)}{R(s)} \right)}_{\geq 0} ds\ge 0.
\end{align*}
This implies 
\[
\liminf_{t\to\infty}\left(\sqrt{\Delta(t)}-\frac{\xi(\eta)}{R(t)}\right)=0,
\]
which, in conjunction with \eqref{eq:Rfraclim}, yields
\[
\liminf_{t\to\infty}\sqrt{\Delta(t)}<\frac{\xi(\eta)}{\delta R^0}.
\]
Thus, we may choose $t_0\ge \eta m$ such that
$\sqrt{\Delta(t_0)}<\frac{\xi(\eta)}{\delta R^0}$, which satisfies the desired estimate \eqref{D-18}.
\end{proof}

Lemma \ref{L4.2} furnishes a positive lower bound on $R$ and an upper bound on $\Delta$. As our framework \eqref{C-1} makes $\xi(\eta)$ small, it makes $\Delta$ very small. Now we show that if $R$ is bounded away from $0$ and $\Delta$ is extremely small, then a majority of the oscillators is trapped in the vicinity of the order phase parameter $\phi$.
Below, $\lceil x \rceil$ is the ceiling function, i.e., the smallest integer bigger than or equal to $x$.

\begin{lemma}{(Sufficient condensation implies existence of a majority cluster)}\label{L4.3} 
Suppose $\Theta\in \mathbb{R}^N$, $\lambda\in\left(0.5,1 \right]$, $\beta\in\left(0,\pi/2 \right)$ are such that either
\begin{equation}\label{D-29-prim}
    R\ge \lambda+(1-\lambda)\cos\beta
\end{equation}
or
\begin{align}\label{D-29}
2\lambda +   \frac{\Delta}{1-\cos\beta} \leq 1+ R,
\end{align}
where $R$ and $\Delta$ are defined as in \eqref{eq:Rphi} and \eqref{eq:Delta}. Then, at least $\lceil \lambda N \rceil$ particles of $\Theta$ are trapped modulo $2\pi$ in an arc whose length is less than $2\beta$. 
\end{lemma}
\begin{proof}
First, we decompose the whole index set $[N]$ into two disjoint subsets:
\begin{gather*}
\mathcal{A}\coloneqq \{ i\in[N]\,:\,\theta_i\in (\phi-\beta,\phi+\beta) \quad \mbox{mod}\,2\pi\},\quad \mathcal{B}\coloneqq[N]\setminus\mathcal{A}.
\end{gather*}
where $\phi$ is defined as in \eqref{eq:Rphi}.
It suffices to show that 
\begin{equation} \label{NNN-4}
|\mathcal{A}|\geq \lambda N.
\end{equation}
We have either \eqref{D-29-prim} or \eqref{D-29}. In case we have \eqref{D-29-prim}, we have
\[
\lambda+(1-\lambda)\cos\beta \stackrel{\mathclap{\eqref{D-29-prim}}}{\le} R=\frac 1N \sum_{i\in \mathcal{A}}\cos(\theta_i-\phi)+\frac 1N\sum_{i\in \mathcal{B}}\cos(\theta_i-\phi) \le \frac{|\mathcal{A}|}{N}+\frac{N-|\mathcal{A}|}{N}\cos\beta
\]
from which we have \eqref{NNN-4}. Otherwise, we have \eqref{D-29}. Note that, for each $i\in\mathcal{B}$,
\begin{align}
\begin{aligned} \label{D-30-1}
& \cos(\theta_i - \phi) \leq \cos \beta \quad \mbox{and} \\
& \sin^2(\theta_i - \phi) = (1 -\cos(\theta_i - \phi))(1 +\cos(\theta_i - \phi))\geq (1 -\cos\beta)(1 +\cos(\theta_i - \phi)).
\end{aligned}
\end{align}
Now, we use \eqref{D-30-1} and $\eqref{B-12}_1$ to get 
\begin{align}
\begin{aligned} \label{D-31}
\Delta &\geq \frac{1}{N}\sum_{i\in\mathcal{B}}\sin^2(\theta_i - \phi) \geq (1 -\cos\beta)  \frac{1}{N} \sum_{i\in\mathcal{B}} \Big (1 +\cos(\theta_i - \phi) \Big ) \\
&= (1 -\cos\beta) \left[ \frac{|\mathcal{B}|}{N} +\frac{1}{N}\sum_{i\in\mathcal{B}} \cos(\theta_i-\phi) \right] = (1 -\cos\beta) \left[ \frac{N-|\mathcal{A}|}{N} +R-\frac{1}{N}\sum_{i\in\mathcal{A}} \cos(\theta_i-\phi) \right]\\
&\geq (1-\cos\beta) \left[ \frac{N-|\mathcal{A}|}{N} + R -\frac{|\mathcal{A}|}{N} \right] =(1-\cos\beta) \left[ 1 + R -\frac{2|\mathcal{A}|}{N} \right].
\end{aligned}
\end{align}
We combine \eqref{D-29} and \eqref{D-31} to obtain
\[
1+R \geq 2\lambda + \frac{ \Delta}{1-\cos\beta} \geq 2\lambda + 1+R -\frac{2|\mathcal{A}|}{N}.
\]
This implies 
\[ |\mathcal{A}|\ge \lambda N. \]
\end{proof}

\subsection{Stage C (persistence and relaxation)}
Finally, we are in a situation in which Theorem \ref{L4.4} is applicable because of \eqref{C-1}-(${\mathcal F}_3$). This allows us to finish the proof of Theorem \ref{T3.1}.

\begin{proof}[Proof of Theorem \ref{T3.1}]
Suppose that the framework \eqref{C-1}-$({\mathcal F}_1), ({\mathcal F}_2)$ and $({\mathcal F}_3)$ hold. Now we use the condition \eqref{C-1}-(${\mathcal F}_1$) and Lemma \ref{L4.2} to see that there exists a time $t_0 \in[\eta m,\infty)$ such that
\[
R(t)\ge \delta R^0>0,~ \forall~ t\in[0,t_0],\quad \mbox{and}\quad \Delta(t_0) \leq \left( \frac{\xi(\eta)}{\delta R^0} \right)^2.
\]
Then, use condition \eqref{C-1}-$({\mathcal F}_2)$ to apply Lemma \ref{L4.3} with $\beta = \ell/2$ to see that there is a subensemble $\Theta_\mathcal{A}$ with $|\mathcal{A}|\ge\lambda N$, confined in an arc of length $\le \ell$ at time $t=t_0$. Now, condition \eqref{C-1}-$({\mathcal F}_3)$ enables us to use statements (1), (3), and (4) of Theorem \ref{L4.4} with ${\mathcal B}=[N]$, to prove statements (1) and (2) of Theorem \ref{T3.1}. If we have condition \eqref{C-1}-$({\mathcal F}_4)$ in addition, we may use statement (2) of Theorem \ref{L4.4} to prove statement (3) of Theorem \ref{T3.1}.
\end{proof}

%
%
%
%
\section{Conclusion and future directions} \label{sec:6}
\setcounter{equation}{0}
In this paper, we have presented two main results on phase-locking under some framework which is formulated in terms of (system and free) parameters and initial data and which does not impose any restriction on the diameter of the initial data. Under the proposed framework, we first show that a majority cluster emerges in finite-time and persists. Here a majority cluster denotes a subensemble of phase oscillators whose cardinality is larger than half of the full cardinality. Second, we show that the oscillators outside the majority cluster also stays in a bounded neighborhood of the majority cluster in a large coupling regime so that the whole ensemble has a diameter bounded in time. Then, we use the inertial gradient flow formulation of the inertial Kuramoto model with an analytic potential and boundedness of phase and frequency diameters to derive asymptotic convergence to the traveling solution moving with a fixed velocity, namely the average of the natural frequencies. In particular, the uniform boundedness of the phase diameter depends on detailed technical estimates on the order parameters.

There are several interesting issues to be explored in future works. Here are a few examples. \newline

\begin{itemize}
\item What is the convergence rate towards a phase-locked state? The inertial gradient flow formulation and the {\L}ojasiewicz gradient theorem only guarantee an algebraic convergence rate toward a phase-locked state, since the {\L}ojasiewicz exponent may be less than $\frac 12$. However, for the unique phase-locked state confined in a quarter-circle, the convergence rate is exponential \cite{C-H-J-K}, and the {\L}ojasiewicz exponent is $\frac 12$  (see \cite{L-X-Y, L-X-Y2}). We pose the following question.
\begin{question}
    Is a phase-locked state of \eqref{A-2} unstable if and only if its {\L}ojasiewicz exponent is less than $\frac 12$? Is the convergence toward an unstable equilibrium always asymptotically algebraic?
\end{question}

\item Note that the mean-field limit $N\to\infty$ for \eqref{A-2} leads to a Vlasov-McKean equation, first derived in \cite{L05} based on the kinetic theory developed in \cite{N78}. This topic has been extensively studied in literature \cite{C-C-H-K-K13,C15,D16,H-K-M-P, M-S07, M-P22}.
The inertial Kuramoto model \eqref{A-1} on graphs has also been explored in the limit $N\to\infty$ \cite{C-M22,C-M-M23}. We also refer to \cite{C23,C20} for the hydrodynamic model derived from \eqref{A-1} to describe synchronization phenomena. Researching the asymptotic dynamics of the Vlasov-McKean equation for \eqref{A-1} obtained in the mean-field limit, such as in the flavor of \cite{M-P22}, is an interesting direction.
\vspace{0.2cm}

\item Note that we did not use the full power of Theorem \ref{L4.4} in the proof of Theorem \ref{T3.1}, i.e., Theorem \ref{L4.4} does not require control on $\mathcal{V}_{[N]\setminus \mathcal{A}}$. What if, in the framework $({\mathcal F})$, instead of an upper bound on $\mathcal{D}(\mathcal{V})/\kappa$, we had an upper bound on $\sqrt{\operatorname{Var}(\mathcal{V})}/\kappa$? Passing via Chebyshev's inequality to a subensemble $\mathcal{B}\subset [N]$ with a good upper bound on $\mathcal{D}(\mathcal{V}_\mathcal{B})/\kappa$, can we still deduce the emergence of a majority cluster $\mathcal{A}\subset \mathcal{B}$, and $\mathcal{B}$-partial phase-locking?
\begin{question}
    For each $\varepsilon\in (0,\frac 12)$, do there exist constants $a(\varepsilon)$, $b(\varepsilon)$, $c(\varepsilon)$ such that if $\sqrt{\operatorname{Var}(\mathcal{V})}/\kappa<a(\varepsilon)|R^0|^2$, $\sqrt{\operatorname{Var}(\Omega^0)}/\kappa<b(\varepsilon)|R^0|^2$, and $m\kappa<c(\varepsilon)|R^0|^2$, then there exists $\mathcal{B}\subset [N]$ with $|\mathcal{B}|\ge (1-\varepsilon)N$ such that the solution $\Theta$ to \eqref{A-1} exhibits $\mathcal{B}$-partial phase-locking? Does there exist a majority cluster $\mathcal{A}\subset \mathcal{B}$ such that $\Theta_\mathcal{A}$ is confined in a quarter-circle for large enough times?
\end{question}
\noindent {\bf Conjecture} \ref{conj:R} is also relevant to this question. \newline

\item Our proposed framework $({\mathcal F})$ is a sufficient one, namely, our lower bound for a minimum coupling strength given $\kappa$ having fixed $m$, $\mathcal{V}$, $\Theta^0$, and $\Omega^0$ is likely not optimal. Thus, it will be interesting to compute the optimal lower bound, namely a ``pathwise-critical coupling strength'', which induces asymptotic phase-locking. This definition is different from the usual critical coupling strength for the existence of a phase-locked state, and the pathwise-critical coupling strength for phase-locking is not known even for the Kuramoto model itself. We pose the following:
\begin{question}\label{ques:crit-coincide}
    Do the pathwise-critical coupling strength and the critical coupling strength coincide for the Kuramoto models \eqref{A-1} and \eqref{A-2}?
\end{question}
An affirmative answer to Question \ref{ques:crit-coincide} would elucidate the role of the inertia $m$: it does not affect whether synchronization occurs, but the timescale at which synchronization occurs is multiplicative in $m$ (see Figure \ref{Fig1-2} and its interpretation). This would resolve the aforementioned tension between \cite{D-B12} and \cite{A-B00,C-C10} regarding the role of $m$ in synchronization.

Constructing a weak Lyapunov functional for \eqref{A-1} or \eqref{A-2} on the torus $\mathbb{T}^N$, as opposed to the gradient potential on $\mathbb{R}^N$, will be useful towards improving asymptotic phase-locking results for generic initial data, and possibly answering Question \ref{ques:crit-coincide} in the affirmative (see Conjecture \ref{conj:lyapunov}). Because we do not know a bona fide Lyapunov functional, we have used the diameter of some subensemble and the order parameter as ``pseudo-Lyapunov functionals'' and have obtained some partial results. To solve the synchronization problem for the Kuramoto models \eqref{A-1} and \eqref{A-2} by affirmatively answering Question \ref{ques:crit-coincide}, we strongly suggest resolving Conjecture \ref{conj:lyapunov}, namely finding this elusive Lyapunov functional.
\end{itemize}

\bigskip

\appendix

\newpage 

\section{Description of previous results}\label{app:suppt}
In this appendix, we present previous results to support Section \ref{sec:2-divided}.
\subsection{Stability of majority clusters}\label{app:suppt-1}
In the case of the first-order Kuramoto model \eqref{A-2}, partial phase-locking was proven by the authors as the following theorem.
\begin{theorem}[{\cite[Theorem 3.1, Corollary 3.1]{H-R}}]\label{thm:1stpartial}
For given sets of indices $\mathcal{A}\subset\mathcal{B}\subset [N]$, suppose that the real parameters $\lambda,\ell$ and $\kappa$ satisfy
\begin{equation*} 
\frac{1}{2} < \lambda \leq 1, \quad \ell \in\left(0,2\cos^{-1} \Big( \frac{1}{\lambda} - 1 \Big) \right),\quad |\mathcal{A}|\ge \lambda N, \quad \kappa > \frac{\mathcal{D}(\Omega_\mathcal{B})}{\lambda \sin \ell -2 (1-\lambda) \sin \frac{\ell}{2}},
\end{equation*}
and let $\Theta$ be a global solution to \eqref{A-2} such that $\Theta_\mathcal{A}$ is a $\lambda$-ensemble of arc length $\le \ell$ at time $0$:
\[
\mathcal{D}(\Theta_\mathcal{A}^0)\leq \ell.
\]
Then, the following assertions hold. 
\begin{enumerate}
\item The ensemble $\Theta_\mathcal{A}$ is stable:
\[
\sup_{0\leq t<\infty}\mathcal{D}(\Theta_\mathcal{A}(t))\leq \ell
\]
and
\[\limsup_{t\rightarrow\infty}\mathcal{D}(\Theta_\mathcal{A}(t))\leq \phi_1(\lambda, \mathcal{D}(\Omega_\mathcal{A})/\kappa)\stackrel{\mathrm{Lemma~}\ref{f-estimates}(i)}{<}\frac{3\pi}{4(2\lambda-1)}\frac{\mathcal{D}(\Omega_\mathcal{A})}{\kappa}.
\]
\item The ensemble $\Theta_\mathcal{B}$ is partially phase-locked:
\[
\sup_{0\leq t<\infty}\mathcal{D}(\Theta_\mathcal{B}(t))<\infty.
\]
In particular, if $\mathcal{B}=[N]$, then asymptotic phase-locking occurs.
\vspace{0.2cm}
\item If we assume in addition that
\[
\frac{\mathcal{D}(\Omega_\mathcal{A})}{\kappa}< \frac{(2\lambda-1)^{3/2}}{\sqrt{2\lambda}}\frac{2-\lambda}{\sqrt{\lambda/2}+(1-\lambda)},
\]
then the oscillators of $\Theta_\mathcal{A}$ become linearly ordered according to their natural frequencies: for $i,j\in \mathcal{A}$, with $\nu_i\ge \nu_j$,
\begin{align*}
\begin{aligned}
& \frac{\nu_i-\nu_j}{\kappa}\le \liminf_{t\rightarrow\infty}[\theta_i(t)-\theta_j(t)]\le \limsup_{t\rightarrow\infty}[\theta_i(t)-\theta_j(t)] \le \frac{\pi}{2\sqrt{2}(\lambda\cos\phi_1-(1-\lambda))}\frac{\nu_i-\nu_j}{\kappa}.
\end{aligned}
\end{align*}
\end{enumerate}
\end{theorem}
In the case of the inertial Kuramoto model \eqref{A-1}, the following result was proven for partial phase-locking. It can be shown that Corollary \ref{cor:partialphaselocking-initial} subsumes this theorem.
\begin{theorem}[{\cite[Theorem 1.2]{H-J-K}} ] \label{T2.4}
The following assertions hold.
\begin{enumerate}
\item
(Complete frequency synchronization):~ Suppose that the parameters and initial data satisfy the following conditions:
 \begin{equation*}
 \begin{cases}
\displaystyle \mathcal{A} \subset [N],\quad |\mathcal{A}| = M,\quad M > \frac{N}{2}, \quad 0<\beta<\alpha, \quad 2\beta + \alpha < \pi, \\
\displaystyle  \frac{M}{N}\sin(\frac{\alpha}{2}-\frac{\beta}{4})\cos(\frac{\alpha}{2}+\frac{5\beta}{8}) - (1-\frac{M}{N})\cos(\frac{\alpha}{2}-\frac{\beta}{8})>0, \\
\displaystyle 0< \mu \leq \frac{M}{N}\sin(\frac{\alpha}{2}-\frac{\beta}{4})\cos(\frac{\alpha}{2} + \frac{5\beta}{8}) - (1-\frac{M}{N})\cos(\frac{\alpha}{2}-\frac{\beta}{8}),~~\lambda>\mu+2, \\
\displaystyle m\kappa \leq \frac{\beta}{4(\lambda + \mu+2) \ln (\frac{\lambda +2\mu+2}{\mu})},\quad {\mathcal D}(\mathcal{V}) < \mu\kappa, \\
\displaystyle {\mathcal D}(\Theta_\mathcal{A}^0) \leq \pi - \alpha - \beta,\quad {\mathcal D}(\Omega_\mathcal{A}^0) < \lambda \kappa,
\end{cases}
\end{equation*}
and let $(\Theta, \Omega)$ be a global solution to \eqref{B-1}.  Then, we have complete synchronization:
\[ \sup_{0 \leq t <\infty} {\mathcal D}(\Theta_\mathcal{A}(t)) \leq \pi -\alpha,\quad \sup_{t_* < t < \infty}
{\mathcal D}(\Theta_\mathcal{A}(t)) \leq \pi -\alpha -\beta, \quad  \lim_{t \to \infty} \max_{i \in [N]}  |{\dot \theta}_i - \nu_c| = 0, \]
where $t_*$ is a positive constant defined by 
\[ t_* \coloneqq  \frac{\beta}{4(\lambda + \mu + 2)\kappa} + \frac{\beta}{\kappa(\mu - (\lambda + 2\mu +2 ) e^{-\frac{\tau}{m}})}. \]

\item
(Complete phase synchronization):~Suppose that the parameters and initial data satisfy the following conditions:
\begin{equation*}  
\begin{cases} 
\displaystyle  0 < \beta < \alpha < \pi, \quad 2\beta + \alpha < \pi, \\
\displaystyle \mu \coloneqq \sin( \frac{\alpha}{2} - \frac{\beta}{4}) \cos( \frac{\alpha}{2} + \frac{5\beta}{8}), \quad \lambda > \mu + 2,     \\
\displaystyle  m \kappa \leq\frac{\beta}{ 4(\lambda + \mu + 2) \ln ( \frac{\lambda + 2\mu + 2}{\mu})}, \quad D({\mathcal V})  = 0,    \\
\displaystyle  {\mathcal D}(\Theta^0) \leq \pi - \alpha -\beta, \quad {\mathcal D}(\Omega^0) < \lambda \kappa,
\end{cases}
\end{equation*}
and let $\Theta(t)$ be a global solution to \eqref{B-1}. Then, we have complete phase synchronization:
\[ \lim_{t \to \infty} \max_{i,j \in [N]}  | \theta_i(t) -\theta_j(t)| = 0. \]
\end{enumerate}
\end{theorem}
In the special case of $N=3$ for the Kuramoto mdoel \eqref{A-2}:
\begin{equation} \label{n=31st}
\begin{cases}
\vspace{.1cm}\displaystyle  {\dot \theta}_1 = \nu_1 + \frac{\kappa}{3} \Big( \sin(\theta_2 - \theta_1) + \sin(\theta_3 - \theta_1)   \Big), \\
\vspace{.1cm}\displaystyle {\dot \theta}_2 = \nu_2 + \frac{\kappa}{3} \Big( \sin(\theta_1- \theta_2) + \sin(\theta_3 - \theta_2)   \Big), \\
\displaystyle {\dot \theta}_3 = \nu_3 + \frac{\kappa}{3} \Big( \sin(\theta_1 - \theta_3) + \sin(\theta_2- \theta_3)   \Big),
\end{cases}
\end{equation}
Theorem \ref{thm:1stpartial} in Section \ref{sec:3}, coupled with a finite-collision argument, can be used to verify that, in a large coupling regime, asymptotic phase-locking happens regardless of initial data.
\begin{theorem}[{\cite[Proposition 4.1]{H-R}}]\label{thm:n=31st}
Suppose system parameters satisfy
\[
N=3,\quad \kappa>\frac{\sqrt{138+22\sqrt{33}}}{4}\mathcal{D}(\mathcal{V}).
\]
Then, the Kuramoto flow of \eqref{n=31st} exhibits asymptotic phase-locking.
\end{theorem}

Likewise, for the inertial Kuramoto model \eqref{A-1} with $N=3$:
\begin{equation} \label{n=32nd}
\begin{cases}
\vspace{.1cm}\displaystyle m {\ddot \theta}_1 + {\dot \theta}_1 = \nu_1 + \frac{\kappa}{3} \Big( \sin(\theta_2 - \theta_1) + \sin(\theta_3 - \theta_1)   \Big), \\
\vspace{.1cm}\displaystyle m {\ddot \theta}_2 + {\dot \theta}_2 = \nu_2 + \frac{\kappa}{3} \Big( \sin(\theta_1- \theta_2) + \sin(\theta_3 - \theta_2)   \Big), \\
\displaystyle m {\ddot \theta}_3 + {\dot \theta}_3 = \nu_3 + \frac{\kappa}{3} \Big( \sin(\theta_1 - \theta_3) + \sin(\theta_2- \theta_3)   \Big),
\end{cases}
\end{equation}
we may prove asymptotic phase-locking for all initial data under large coupling strength and small inertia, via a partial-locking argument. Previously, in \cite{H-J-K}, Theorem \ref{T2.4} was used to this end to obtain the following result:
\begin{theorem}[{\cite[Theorem 1.3]{H-J-K}}] \label{T2.3} 
Suppose that the system parameters satisfy 
\begin{align*}
& m\kappa \leq \eta_\infty \left( 4(\sin \frac{\eta_\infty}{16} \cos\frac{11\eta_\infty}{16} +2 ) \ln ( \frac{4\sin\frac{\eta_\infty}{16} \cos\frac{11\eta_\infty}{16} +6 }{\sin \frac{\eta_\infty}{16} \cos\frac{11\eta_\infty}{16}} ) \right) \approx 0.0319,\\
& D(\mathcal{V}) < \kappa \max_{\eta\in(0,\frac{2\pi}{5}]}\left( \frac{9}{10}\sin\frac{\eta}{16}\cos\frac{11\eta}{16}\right) \approx 0.0458596 \kappa,
\end{align*}
where $\eta_\infty = 2\pi/5$, and let $(\Theta, \Omega)$ be a global solution to \eqref{n=32nd}. Then, we have
\[
\lim_{t\to\infty} |\dot\theta_i(t) - \nu_c| = 0,  \quad \forall~ i = 1,2,3.
\]
\end{theorem}
Note again that this is subsumed by Theorem \ref{thm:n=32nd}. Note that in the zero inertia limit $m\to 0$, i.e., Theorem \ref{thm:n=32nd} fully recovers Theorem \ref{thm:n=31st}.

\subsection{Synchronization in the identical case}\label{app:suppt-2}
In the inertial Kuramoto model \eqref{A-1} with identical natural frequencies $\nu_i=\nu_j$, the energy dissipation formula \eqref{eq:energy-dissipation} can be used to prove the following.

\begin{theorem}[{\cite[Theorem 3.4]{C-H-M}}] \label{T2.2}
Suppose the initial data and system parameters satisfy
\[
\kappa>0,\quad \nu_i =0\quad \forall i\in[N],\quad \text{and}\quad
m\operatorname{Var}(\Omega^0) \leq \kappa |R^0|^2,
\]
and let $(\Theta,\Omega)$ be a global solution to \eqref{B-1}. Then, the following assertions hold.
\begin{enumerate}
\item Complete frequency synchronization emerges asymptotically:
\begin{align*}
\lim_{t\to\infty} |\omega_i (t) - \omega_j(t)| = 0,\quad i,j\in[N].
\end{align*}

\item If $R^0 =0$, then system lies in a phase-locked state and
\begin{align*}
R(\Theta(t)) = 0,\quad t\geq0.
\end{align*}

\item If $R^0>0$, then there exist a positive lower bound for $R(t)$ and a limiting phase $\Theta^\infty$:
\begin{align*}
\inf_{0 \leq t  <\infty} R(\Theta(t)) >0\quad \text{and}
\quad \lim_{t\to\infty} \| \Theta(t) - \Theta^\infty \|_{\infty} = 0.
\end{align*}
\end{enumerate}
\end{theorem}
Furthermore, in the identical case the potential \eqref{B-5-1} can be written as
\[
P(\Theta)=\frac{\kappa N^2}2 (1-R^2).
\]
This can serve as an inertial gradient not only as a dynamical system on $\mathbb{R}^N$ but also on the compact manifold $\mathbb{T}^N$. Thus, the version of the {\L}ojasiewicz gradient theorem (Proposition \ref{P2.1}) on compact manifolds applies to give the following theorem.
\begin{theorem}[\cite{L-X}]
Suppose the initial data and system parameters satisfy
\[
\kappa>0,\quad \nu_i =0\quad \forall i\in[N],
\]
and let $(\Theta,\Omega)$ be a global solution to \eqref{B-1}. Then $(\Theta, \Omega)$ exhibits asymptotic phase-locking and complete frequency synchronization.
\end{theorem}

%
%
%
%
\section{Sturm--Picone Comparison Principle}\label{app:sturm-picone}
In this appendix, we state and prove a stronger version of the Sturm--Picone comparison principle to be used in later sections of the appendix.

\begin{lemma}[Sturm--Picone Comparison Principle]\label{lem:sturm-picone}
    Let $a,b,c>0$ be positive real numbers, and define the extended real-valued number
    \[
    T^*=T^*(a,b,c)\coloneqq
    \begin{cases}
        \infty,&\mathrm{if~}4ac\le b^2,\\
        \frac{\pi a}{\sqrt{4ac-b^2}}+\frac{2 a}{\sqrt{4ac-b^2}}\sin^{-1}\left(\frac{b}{2\sqrt{ac}}\right),&\mathrm{if~}4ac>b^2.
    \end{cases}
    \]
    Let $I\subset \mathbb{R}$ be a connected open interval, and let $y:I\to \mathbb{R}$ be a continuous function such that for any subinterval $J\subset I$ on which $\left.y\right|_J>0$ pointwise, we have that $y$ is $C^2$ on $J$ and
    \[
    a\ddot{y}(t)+b\dot{y}(t)+cy(t)>0,\quad t\in J.
    \]
    Then, the following assertions hold.
    \begin{enumerate}
        \item If there is a time $t_0\in I$ such that $y(t_0)>0$ and $\dot{y}(t_0)\ge 0$, we have
    \[
    y(t)>0\quad\mathrm{for}~t\in I\cap [t_0,t_0+T^*].
    \]
    \item If $4ac\le b^2$, then there cannot exist two times $t_1,t_2\in I$ with $t_1<t_2$ such that $y(t_1)=y(t_2)=0$ yet $y>0$ on $(t_1,t_2)$, and the set $\{t\in I:y(t)\le 0\}$ is a closed connected subinterval of $I$.
    \item If $4ac\le b^2$ and if we assume in addition that the same condition that holds for $y$ also holds for the function $-y$, i.e., that on any subinterval $J\subset I$ on which $\left. y\right|_J<0$ pointwise, we have that $y$ is $C^2$ and
    \[
    a\ddot{y}(t)+b\dot{y}(t)+cy(t)<0,\quad t\in J,
    \]
    then $y$ cannot change sign twice on $I$: either
    \begin{enumerate}
        \item $y>0$ on $I$, or
        \item $y<0$ on $I$, or
        \item $y=0$ on $I$, or
        \item there exist times $t_1,t_2\in I$ with $t_1\le t_2$ such that $y>0$ on $I\cap (-\infty,t_1)$, $y=0$ on $[t_1,t_2]$, and $y>0$ on $I\cap (t_2,\infty)$, or
        \item there exists a time $t_1\in I$ such that $y>0$ on $I\cap (-\infty,t_1)$ and $y=0$ on $I\cap [t_1,\infty)$, or
        \item there exist times $t_1,t_2\in I$ with $t_1\le t_2$ such that $y>0$ on $I\cap (-\infty,t_1)$, $y=0$ on $[t_1,t_2]$, and $y<0$ on $I\cap (t_2,\infty)$, or
        \item there exists a time $t_2\in I$ such that $y=0$ on $I\cap (-\infty,t_2]$ and $y>0$ on $I\cap (t_2,\infty)$, or
        \item there exists a time $t_2\in I$ such that $y=0$ on $I\cap (-\infty,t_2]$ and $y<0$ on $I\cap (t_2,\infty)$, or
        \item there exist times $t_1,t_2\in I$ with $t_1\le t_2$ such that $y<0$ on $I\cap (-\infty,t_1)$, $y=0$ on $[t_1,t_2]$, and $y>0$ on $I\cap (t_2,\infty)$, or
        \item there exists a time $t_1\in I$ such that $y<0$ on $I\cap (-\infty,t_1)$ and $y=0$ on $I\cap [t_1,\infty)$, or
        \item there exist times $t_1,t_2\in I$ with $t_1\le t_2$ such that $y<0$ on $I\cap (-\infty,t_1)$, $y=0$ on $[t_1,t_2]$, and $y<0$ on $I\cap (t_2,\infty)$.
    \end{enumerate}  
    \end{enumerate}
\end{lemma}
\begin{remark}
    If $y$ has the extra property that if $y=0$ on a nonempty open subinterval of $I$ then $y=0$ on $I$, then we can guarantee in statements (d), (f), (i), (k) of statement (3) of Lemma \ref{lem:sturm-picone} that $t_1=t_2$. In this paper, we will take $y=\theta_i-\theta_j$, the relative phase difference between two oscillators, so that this extra property is guaranteed from the time-autonomy or the uniqueness of solutions to \eqref{B-1}, respectively. This extra property also follows from real-analyticity.
\end{remark}
\begin{proof} Below, we present proof for each assertion one by one.  \newline

\noindent (i)~Let $z:\mathbb{R}\to \mathbb{R}$ be the solution to the second-order linear ordinary differential equation
\begin{equation}\label{eq:2ndODE}
a\ddot z+ b\dot z + c z=0,\quad z(0)=1,\quad \dot z(0)=0.
\end{equation}
If $4ac < b^2$, the solution is
\begin{equation*}
z(t)=\frac{e^{-bt/2a}}{2}\left(\left(\frac{b}{\sqrt{b^2-4ac}}+1\right)\exp\left(\frac{\sqrt{b^2-4ac}}{2a}t\right)-\left(\frac{b}{\sqrt{b^2-4ac}}-1\right)\exp\left(-\frac{\sqrt{b^2-4ac}}{2a}t\right)\right),
\end{equation*}
and does not have a zero on $[0,\infty)$. If $4ac= b^2$, the solution is
\begin{equation*}
z(t)=e^{-bt/2a}\left(1+\frac{bt}{2a}\right)
\end{equation*}
and again does not have a zero on $[0,\infty)$.
Finally, if $4ac>b^2$, the solution is
\begin{equation*}
z(t)=e^{-bt/2a}\left(\cos\left(\frac{\sqrt{4ac -b^2}}{2a}t\right)+\frac{b}{\sqrt{4ac -b^2}}\sin\left(\frac{\sqrt{4ac -b^2}}{2a}t\right)\right)
\end{equation*}
and has its first zero on $[0,\infty)$ at $t=T^*$. All in all, in any case, $\theta(t)>0$ for $t\in [0,T^*)$. \newline

Next, we claim that under the constraints $y(t_0)>0$ and $\dot y(t_0)\ge 0$,
\[ y(t)>0 \quad \mbox{for $t\in I\cap [t_0,t_0+T^*]$}. \]
Define
\[
t^*\coloneqq \sup\left\{T>0:y(t)>0~\forall t\in I\cap [t_0,t_0+T]\right\}>0,
\]
the positivity of which follows from the continuity of $y$ and $y(t_0)>0$. To prove statement (1), it is equivalent to show $t^*= \infty$ if $T^*=\infty$ and $t^*> T^*$ if $T^*<\infty$. For the sake of contradiction, assume $0<t^*\le T^*$ and $t^*<\infty$. Then, by maximality of $t^*$, openness of $I$, and the continuity of $y$, we have that $[t_0-\varepsilon,t_0+t^*]\subset I$, $y(t_0+t^*)=0$ and $y>0$ on $(t_0-\varepsilon,t_0+t^*)$ for some small $\varepsilon>0$, which by assumption implies that $y$ is $C^2$ on $(t_0-\varepsilon,t_0+t^*)$ and
\begin{equation}\label{eq:y-2ndODE}
a\ddot y+b\dot y+cy>0\quad \mathrm{on~} (t_0,t_0+t^*).
\end{equation}
Since $y(t_0+t^*)=0$ and $y$ is positive and differentiable on $(t_0,t_0+t^*)$, we have
\begin{equation}\label{eq:y-deriv-nonneg}
\liminf_{t\to (t_0+t^*)-}\dot y(t)\le 0.
\end{equation}

On the other hand, since $t^*\le  T^*$, we have $z(t)> 0$ on $(0,t^*)$ and $z(t^*)\ge 0$. Consider the following Wronskian estimate for $t\in (0,t^*)$:
\begin{align*}
&\frac{d}{dt}\left(ae^{bt/a}\left(z(t) \dot{y}(t_0+t)-y(t_0+t)\dot z(t)\right)\right)\\
& \hspace{0.5cm} =be^{bt/a}\left(z(t) \dot{y}(t_0+t)-y(t_0+t)\dot z(t)\right)+ae^{bt/a}\left(z(t) \ddot{y}(t_0+t)-y(t_0+t)\ddot z(t)\right)\\
& \hspace{0.5cm}  =e^{bt/a}z(t)\left(\underbrace{a\ddot{y}(t_0+t)+b\dot{y}(t_0+t)+c y(t_0+t)}_{>0~\because\eqref{eq:y-2ndODE}}\right)-e^{bt/a}y(t_0+t)\left(\underbrace{a\ddot z(t)+b\dot z(t)+c z(t)}_{=0~\because \eqref{eq:2ndODE}}\right)\\
& \hspace{0.5cm}  >0.
\end{align*}
Therefore $ae^{bt/a}\left(z(t) \dot{y}(t_0+t)-y(t_0+t)\dot z(t)\right)$ is a strictly increasing function of $t$ on $(0,t^*)$, and, as $\dot{y}(t_0)\ge 0$, $z(0)=1$, and $\dot z(0)=0$, it is nonnegative at $t=0$; hence
\[
0\le ae^{0}\left(z(0) \dot{y}(t_0)-y(t_0)\dot z(0)\right)<\liminf_{t\to t^*-}ae^{bt/a}\left(z(t) \dot{y}(t_0+t)-y(t_0+t)\dot z(t)\right).
\]
But, recalling $z(t^*)\ge 0$, \eqref{eq:y-deriv-nonneg}, and $y(t_0+t^*)=0$, we have
\[
0<\liminf_{t\to t^*-}\left(z(t)\dot{y}(t_0+t)-y(t_0+t)\dot{z}(t)\right)=z(t^*)\liminf_{t\to (t_0+t^*)-}\dot{y}(t)-0\stackrel{\eqref{eq:y-deriv-nonneg}}{\le} 0,
\]
which gives a contradiction. This completes the proof of statement (1). \newline

\noindent (ii)~Assume $4ac\le b^2$ and assume for the sake of contradiction the existence of such $t_1$ and $t_2$. If we set $J=(t_1,t_2)$, we then have that $y>0$ on $J$, so by assumption $y$ is $C^2$ on $J$. By Rolle's theorem, there exists a time $t_0\in J$ such that $\dot{y}(t_0)=0$, while we have $y(t_0)>0$ since $t_0\in J$. Thus, by (1), and $T^*=\infty$, we conclude that
\[
y(t)>0 \quad\mathrm{for}~t\in I\cap [t_0,\infty).
\]
But $t_2\in I\cap [t_0,\infty)$, leading to $y(t_2)>0$, a contradiction. Hence, such $t_1$ and $t_2$ cannot exist.

Denote $I_{\le 0}\coloneqq\{t\in I:y(t)\le 0\}$, which is relatively closed in $I$ by continuity of $y$. We are to show that $I_{\le 0}$ is a closed subinterval of $I$. This amounts to proving that whenever $t_1',t_2'\in I_{\le 0}$ with $t_1'<t_2'$, we have $[t_1',t_2']\subset I_{\le 0}$. If this were not the case, there would be some $t_0\in (t_1',t_2')\setminus I_{\le 0}$. Since $I_{\le 0}$ is relatively closed in $I$, the set $(t_1',t_2')\setminus I_{\le 0}$ is an open set containing $t_0$, so if we define
\[
t_1=\inf\{t\in (t_1',t_0):(t,t_0]\subset (t_1',t_2')\setminus I_{\le 0}\},\quad t_2=\sup\{t\in (t_0,t_2'):[t_0,t)\subset (t_1',t_2')\setminus I_{\le 0}\},
\]
we have 
\[ t_1'\le t_1<t_0<t_2\le t_2'. \] Since $I_{\le 0}$ is closed in $I$ and $t_1',t_2'\in I_{\le 0}$, we must have $t_1,t_2\in I_{\le 0}$ as well, yet $(t_1,t_2)\subset (t_1',t_2')\setminus I_{\le 0}$. Equivalently, $y(t_1)\le 0$ and $y(t_2)\le 0$ yet $y(t)> 0$ for all $t\in (t_1,t_2)$, so that $y(t_1)=y(t_2)=0$, so the existence of $t_1$ and $t_2$ with these properties contradicts the earlier part of statement (2). This contradiction proves that $I_{\le 0}$ is a closed subinterval of $I$. \newline

\noindent (iii)~With the additional assumption of (3) (i.e., that if $y<0$ on an interval $J$ then $y$ is $C^2$ on $J$ and $a\ddot y+b\dot y+cy<0$ on $J$), statement (2) for the function $-y$ tells us that the set
    \[
    I_{\ge 0}\coloneqq\{t\in I:y(t)\ge 0\}
    \]
    is a closed connected subinterval of $I$. Therefore
    \[
    I_0\coloneqq\{t\in I:y(t)= 0\}=I_{\le 0}\cap I_{\ge 0},
    \]
    being the intersection of two closed connected subintervals $I_{\le 0}$ and $I_{\ge 0}$ of $I$, is itself a closed connected subinterval of $I$. Therefore, either $I_0=\emptyset$, in which case we have cases (a) or (b), or $I_0=[t_1,t_2]$ for some $t_1,t_2\in I$ with $t_1\le t_2$, in which case we have cases (d), (f), (i), or (k), or $I_0=I\cap [t_1,\infty)$ for some $t_1\in I$, in which case we have cases (e) or (j), or $I_0=I\cap (-\infty,t_2]$ for some $t_2\in I$, in which case we have cases (g) or (h), or $I_0=I$, in which case we have case (iii).
\end{proof}
Complementary to Lemma \ref{lem:sturm-picone}, we have the following Barbalat-type lemma.
\begin{lemma}[Barbalat-type lemma]\label{lem:barbalat}
Suppose that $a,b, c, T\in \mathbb{R}$ are real numbers with $c>0$, and let $y:(T,\infty)\to [0,\infty)$ be a $C^2$-function such that
    \begin{equation}\label{eq:2nd-rev}
        a\ddot y+b\dot y+cy\le 0,\quad t>T,
    \end{equation}
    with the uniform $C^1$ bound
    \begin{equation}\label{eq:C1-bound}
    \sup_{t>T} \Big( |y(t)| + |\dot{y}(t)|  \Big) <\infty.
    \end{equation}
    Then,  $y$ tends to zero asymptotically:
    \[
    \lim_{t\to\infty}y(t)=0.
    \]
  If we have an additional uniform a priori $C^2$-bound:
    \begin{equation}\label{eq:C2-bound}
    \sup_{t>T} |\ddot{y}(t)|<\infty,
    \end{equation}
    then, we have
    \[
    \lim_{t\to\infty} \dot{y}(t)=0.
    \]
\end{lemma}
\begin{proof}
\noindent (i)~We integrate \eqref{eq:2nd-rev} to see that  for $t_1>t_2>T$,
    \[
    c\int_{t_2}^{t_1}y(t)dt\le a(\dot{y}(t_2)-\dot{y}(t_1))+b(y(t_2)-y(t_1)),
    \]
    and so, invoking the uniform bounds of \eqref{eq:C1-bound} and the nonnegativity of $y$, the positivity of $c$, and taking $t_1\to\infty$ and $t_2\to T+$, we have
    \[
    \int_T^\infty y(t)dt<\infty.
    \]
    The uniform upper bound on $|\dot{y}|$ given in the second part of \eqref{eq:C1-bound} tells us easily that
\[ \lim_{t\to\infty}y(t)=0. \]
\noindent (ii)~Barbalat's lemma says that if $y(t)$ has a finite limit as $t\to\infty$ and if $\dot{y}(t)$ is uniformly continuous in $t$, then 
\[ \lim_{t\to\infty}\dot{y}(t)=0. \]
If we assume \eqref{eq:C2-bound}, both of these conditions are satisfied, and so we may conclude 
\[ \lim_{t\to\infty}\dot{y}(t)=0. \]
\end{proof}
\begin{remark}
In our applications, we first find $a,b,c>0$ such that Lemma \ref{lem:sturm-picone} applies to a certain function $y$, to deduce its positivity for sufficiently large time. Then, if possible, we will find another triple $a,b,c>0$ such that Lemma \ref{lem:barbalat} is satisfied, so that we obtain $y(t),\dot{y}(t)\to 0$ as $t\to\infty$. It is also of interest when the same differential inequality holds regardless of the sign of $y$.
\end{remark}
\begin{lemma}\label{lem:cutoff-comparison}
Suppose that real numbers $a, b, c, T$ and extended real number  $M$ satisfy 
\[ a > 0, \quad b > 0, \quad c > 0, \quad 4 ac \leq b^2,~-\infty < T < \infty, \quad  0 < M  \leq \infty, \]  
and let $y:(T,\infty)\to \mathbb{R}$ be a $C^2$ function satisfying 
    \begin{equation} \label{App-B-1}
    \sup_{t>T}|\dot{y}(t)|<\infty \quad \mbox{and} \quad   a\ddot y(t)+b\dot y(t)+c\max \{-M,y(t)  \}\ge 0,\quad t>T.
    \end{equation}
Then, we have
\[ \liminf_{t\to\infty}y(t)\ge 0. \]
\end{lemma}
\begin{proof}
    We need to show that for each $\varepsilon\in (0,M)$, there exists $T_\varepsilon>T$ such that
    \begin{equation}\label{eq:t-epsilon}
    y(t)>-\varepsilon,\quad \forall t>T_\varepsilon.
    \end{equation}
Suppose this were not the case. Note that, if $t>T$ is a time at which $y(t)>-\varepsilon$, then
    \[
    a\frac{d^2}{dt^2}(y(t)+\varepsilon)+b\frac{d}{dt}(y(t)+\varepsilon)+c(y(t)+\varepsilon)=(a\ddot y(t)+b\dot y(t)+cy(t))+c\varepsilon\stackrel{\mathclap{\eqref{App-B-1}}}{\ge} c\varepsilon>0,
    \]
    so by Lemma \ref{lem:sturm-picone} applied to $y+\varepsilon$ instead of $y$, we have that if $t^*>T$ is so that $y(t^*)>-\varepsilon$ and $\dot y(t^*)\ge 0$, then $y(t)>-\varepsilon$ for $t\ge t^*$. Since we are assuming that there is no $T_\varepsilon$ for which \eqref{eq:t-epsilon} holds, it must be that $\dot{y}<0$ whenever $y>-\varepsilon$.

 By the assumption that there is no $T_\varepsilon$ for which \eqref{eq:t-epsilon} holds, there must exist a time $t_0>T$ such that $y(t_0)\le -\varepsilon$. Since $\dot{y}<0$ whenever $y>-\varepsilon$, it must be that $y(t)\le -\varepsilon$ for all $t\ge t_0$. But then
    \[
    0\stackrel{\mathclap{\eqref{App-B-1}}}{\le} a\ddot y(t)+b\dot y(t)+c\max\{-M,y(t)\}\le a\ddot y(t)+b\dot{y}(t)-c\varepsilon,\quad t\ge t_0,
    \]
    so integrating gives
    \[
    c\varepsilon(t-t_0)\le a(\dot{y}(t)-\dot{y}(t_0))+b(y(t)-y(t_0))\le 2a\sup_{\tau >T}|\dot{y}(\tau)|+b(-\varepsilon-y(t_0)),\quad t\ge t_0,
    \]
    which is a contradiction since the right-hand side is uniformly bounded in $t$ yet the left-hand side grows to infinity as $t\to\infty$. This contradiction proves that there exists a $T_\varepsilon>T$ such that \eqref{eq:t-epsilon} holds.
\end{proof}

%
%
%
%
\section{Proof of Theorem \ref{simplemainthm}} \label{app:mainthm}
\setcounter{equation}{0}
Recall the framework \eqref{C-1}:
\[
\begin{cases}
\displaystyle ({\mathcal F}_1):~~R^0>0, \quad \zeta(\eta)
\leq (1-\delta)R^0. \\
\displaystyle  ({\mathcal F}_2):~~
\delta R^0\ge \lambda+(1-\lambda)\cos\frac{\ell}{2} \quad \mathrm{or}\quad 2\lambda + \left( \frac{\xi(\eta)}{\delta R^0} \right)^2\frac{1}{1-\cos(\ell/2)} \leq 1+ \delta R^0. \\
\displaystyle  ({\mathcal F}_3):~~\xi(\eta)<\sin \frac{\ell}{2} \left( \lambda \cos\frac{\ell}{2} - (1-\lambda) \right). \\
\displaystyle  ({\mathcal F}_4):~~\frac{\mathcal{D}(\mathcal{V})}{\kappa} +4m\kappa+2m\mathcal{D}(\mathcal{V})< \frac{(2\lambda-1)^{3/2}}{\sqrt{2\lambda}}\frac{2-\lambda}{\sqrt{\lambda/2}+(1-\lambda)},
\end{cases}
\]
and the definitions of $\zeta(\eta)$ and $\xi(\eta)$ in \eqref{B-14-0-1}:
\[
\begin{cases}
\displaystyle \zeta(\eta)\coloneqq\frac{m(1-e^{-\eta})}{2}\left[\mathcal{D}(\Omega^0)+\mathcal{D}(\mathcal{V})\eta\right]
+m^2\kappa\left(1-e^{-\eta}\right)^3\left[\frac 34  \mathcal{D}(\Omega^0)+  (\mathcal{D}(\mathcal{V})+2\kappa)\eta \right],\\
\displaystyle \xi(\eta)\coloneqq \left(\mathcal{D}(\mathcal{V})+2\kappa\right)m+\mathcal{D}(\Omega^0)m\max\{1,\eta\}e^{-\max\{1,\eta\}}+\frac{\mathcal{D}(\mathcal{V})}{2\kappa}+\frac{\mathcal{D}(\Omega^0)}{2\kappa}\frac{e^{-\eta}}{1-e^{-\eta}}.
\end{cases}
\]
Define
\[
{\tilde \zeta}(\eta)\coloneqq \frac{1-e^{-\eta}}{2}(yz+\eta xy)+(1-e^{-\eta})^3y^2(\frac 34 z+\eta x+2\eta)
\]
and
\[
{\tilde \xi}(\eta)=y(x+2)+\max\{1,\eta\}e^{-\max\{1,\eta\}}yz+\frac x2+\frac{e^{-\eta}}{1-e^{-\eta}}\frac z2.
\]
Then, we have
\[
\zeta(\eta)\le{\tilde \zeta}(\eta)|R^0|^2,\quad \mbox{and} \quad  \xi(\eta)\le \tilde{\xi}(\eta)|R^0|^2.
\]
So the given condition 
\begin{align*}
\inf_{\eta>0}&(1-e^{-\eta})y\left(\frac{1}{2}(z+\eta x)+(1-e^{-\eta})^2y(\frac 34 z+\eta x+2\eta)\right)\\
&+\sqrt{\frac{1}{0.3259}\left(y(x+2)+\max\{1,\eta\}e^{-\max\{1,\eta\}}yz+\frac x2+\frac{e^{-\eta}}{1-e^{-\eta}}\frac z2\right)}< 1
\end{align*}
is equivalent to the existence of an $\eta>0$ such that
\[
\tilde{\zeta}(\eta) +\sqrt{\frac{\tilde{\xi}(\eta)}{0.3259}}\le 1.
\]
Now, we choose some $\delta\in (0,1)$ such that
\[
\tilde{\zeta}(\eta)\le 1-\delta \quad \mbox{and} \quad \tilde{\xi}(\eta)\le 0.3259\delta^2.
\]
Finally, we set
\[
\lambda = \begin{cases}
    0.5+\frac{35}{94}\delta R^0,&0<\delta R^0\le 0.94,\\
    2.5\delta R^0-1.5,& 0.94<\delta R^0\le 1,
\end{cases}
\qquad 
\ell=
\begin{cases}
    2\cos^{-1}(1-\frac{20}{47}\delta R^0),& 0<\delta R^0\le 0.94,\\
    2\cos^{-1}0.6,& 0.94<\delta R^0\le 1.
\end{cases}
\]
By Theorem \ref{T3.1}, it is enough to show that this choice of free parameters $\eta, \delta,\lambda,\ell$ satisfy the framework (${\mathcal F}$) in \eqref{C-1}. Condition \eqref{C-1}-$({\mathcal F}_1)$ is satisfied since
\[
\zeta(\eta)\le \tilde{\zeta}(\eta)R^0\le (1-\delta)R^0.
\]
For the other conditions, we need the following lemma.
\begin{lemma}\label{lem:numeric}
The above choice of $\lambda$ and $\ell$ obey the following estimates. 
\begin{enumerate}
\item If $0.94<\delta R^0\le 1$, then
\[
\delta R^0= \lambda+(1-\lambda)\cos\frac{\ell}{2}.
\]
Otherwise, if $0<\delta R^0\le 0.94$, then
\[
(\delta R^0)\sqrt{1-\cos(\ell/2)}\sqrt{1+\delta R^0-2\lambda}\ge 0.3296(\delta R^0)^2.
\]
\item \[
\sin \frac{\ell}{2} \left( \lambda \cos\frac{\ell}{2} - (1-\lambda) \right)>0.3259(\delta R^0)^2.
\]
\item \[
\frac{(2\lambda-1)^{3/2}}{\sqrt{2\lambda}}\frac{2-\lambda}{\sqrt{\lambda/2}+(1-\lambda)}\ge 0.729 (\delta R^0)^2.
\]
\end{enumerate}
\end{lemma}
\begin{proof} 
The proof will be given at Step B in the sequel.
\end{proof}
\vspace{0.2cm}
\noindent Now, we are ready to provide a proof of Theorem \ref{simplemainthm} in two steps.  \newline

\noindent $\bullet$~Step A (Lemma \ref{lem:numeric} implies the proof of  Theorem \ref{simplemainthm}): ~Suppose Lemma \ref{lem:numeric} holds. Then, we note that 
\[
\xi(\eta)\le \tilde{\xi}(\eta)|R^0|^2\le 0.3259(\delta R^0)^2<0.3296(\delta R^0)^2.
\]
Thus, statement (1) of Lemma \ref{lem:numeric} tells us that condition $({\mathcal F}_2)$ is satisfied, and statement (2) of Lemma \ref{lem:numeric} tells us that condition $({\mathcal F}_3)$ is satisfied. Finally, by statement (3) of Lemma \ref{lem:numeric},
\[
\frac{\mathcal{D}(\mathcal{V})}{\kappa} +4m\kappa+2m\mathcal{D}(\mathcal{V})<2\xi(\eta)<2\cdot 0.3259(\delta R^0)^2<0.729(\delta R^0)^2\le \frac{(2\lambda-1)^{3/2}}{\sqrt{2\lambda}}\frac{2-\lambda}{\sqrt{\lambda/2}+(1-\lambda)}.
\]
To recapitulate, we have shown, assuming Lemma \ref{lem:numeric}, that the conditions of Theorem \ref{simplemainthm} implies the existence of parameters $\eta,\delta,\lambda,\ell$ that satisfy framework \eqref{C-1}, so that the results of Theorem \ref{T3.1} apply. This proves asymptotic phase-locking and well-ordering of a majority cluster, as claimed in Theorem \ref{simplemainthm}.

\vspace{0.5cm}

\noindent $\bullet$~Step B (Verification of Lemma \ref{lem:numeric}): In the sequel, we verify three assertions one by one. \newline

\noindent $\diamond$~Case A (Verification of the first assertion):~ If $0.94<\delta R^0\le 1$, then
        \[
        \lambda+(1-\lambda)\cos\frac \ell 2=0.6+0.4\lambda = \delta R^0.
        \]
        If $0<\delta R^0\le 0.94$, then
        \begin{align*}
            \sqrt{1-\cos(\ell/2)}\sqrt{1+\delta R^0-2\lambda}&=\sqrt{\frac{20}{47}\delta R^0}\sqrt{\frac{24}{94}\delta R^0} =  \frac{4\sqrt{15}}{47}\delta R^0 > 0.32961560 (\delta R^0) .
        \end{align*}
  \vspace{0.2cm}
  
  \noindent $\diamond$~Case B (Verification of the second assertion):~If $0<\delta R^0\le 0.94$, then
        \begin{align*}
        \begin{aligned}
            & \sin \frac{\ell}{2} \left( \lambda \cos\frac{\ell}{2} - (1-\lambda) \right) \\
            & \hspace{1cm} =\sqrt{1-\cos \frac{\ell}{2}}\sqrt{1+\cos\frac\ell 2} \left( \lambda \cos\frac{\ell}{2} - (1-\lambda) \right)  \\
            & \hspace{1cm} =\sqrt{\frac {20}{47}\delta R^0}\sqrt{2-\frac{20}{47}\delta R^0}\left(\frac{25}{47}\delta R^0-\frac{350}{47^2}(\delta R^0)^2\right)\\
            & \hspace{1cm} =\sqrt{\frac {20}{47}}\sqrt{2-\frac{20}{47}\delta R^0}\left(\frac{25}{47\sqrt{\delta R^0}}-\frac{350}{47^2}\sqrt{\delta R^0}\right)\cdot (\delta R^0)^2 \\
            & \hspace{1cm} \ge \sqrt{\frac {20}{47}}\sqrt{2-\frac{20}{47}(0.94)}\left(\frac{25}{47\sqrt{0.94}}-\frac{350}{47^2}\sqrt{0.94}\right)\cdot (\delta R^0)^2\\
            & \hspace{1cm} =\frac{720}{47^2}(\delta R^0)^2 >0.32593933(\delta R^0)^2.
        \end{aligned}     
        \end{align*}
        If $0.94<\delta R^0\le 1$, then
        \begin{align*}
            \sin \frac{\ell}{2} \left( \lambda \cos\frac{\ell}{2} - (1-\lambda) \right)=0.8 (4\delta R^0-3.4)=0.8\left(\frac{4}{\delta R^0}-\frac{3.4}{(\delta R^0)^2}\right)(\delta R^0)^2,
        \end{align*}
        but since the derivative of $x\mapsto \frac 4x-\frac{3.4}{x^2}$ is $x\mapsto \frac{6.8-4x}{x^3}$ and is positive on the interval $(0,1)$, we have that
        \[
        0.8\left(\frac{4}{\delta R^0}-\frac{3.4}{(\delta R^0)^2}\right)\ge 0.8\left(\frac{4}{0.94}-\frac{3.4}{0.94^2}\right)\ge \frac{720}{47^2}.
        \]
        
        \vspace{0.5cm}
  \noindent $\diamond$ Case C~ (Verification of the third assertion)::~If $0<\delta R^0\le 0.94$, we observe that the function defined on $(\frac 12,1]$ by
    \[
    \lambda \quad  \mapsto \quad \frac{1}{\sqrt{2\lambda}\sqrt{2\lambda-1}}\frac{2-\lambda}{\sqrt{\lambda/2}+(1-\lambda)}
    \]
    has negative logarithmic derivative on $(\frac 12,1)$:
    \begin{align*}
        -\frac{1}{2-\lambda}-\frac{1}{2\lambda}-\frac{1}{2\lambda -1}+\frac{-\frac{1}{2\sqrt{2\lambda}}+1}{\sqrt{\lambda/2}+(1-\lambda)}<-\frac{1}{2-0.5}-\frac{1}{2}-1+\frac{-\frac{1}{2\sqrt{2}}+1}{1/2+0}=-\frac 16 -\frac 1{\sqrt{2}}<0.
    \end{align*}
    Thus, the above function is decreasing on $(\frac 12,1]$. Since $0<\delta R^0\le 0.94$, by definition we have 
\begin{equation} \label{App-C-1}
0.5<\lambda\le 0.85.
\end{equation}
Then, we use \eqref{App-C-1} to find 
    \begin{align*}
    \begin{aligned}
\frac{(2\lambda-1)^{3/2}}{\sqrt{2\lambda}}\frac{2-\lambda}{\sqrt{\lambda/2}+(1-\lambda)} &=\frac{1}{\sqrt{2\lambda}\sqrt{2\lambda-1}}\frac{2-\lambda}{\sqrt{\lambda/2}+(1-\lambda)}(2\lambda-1)^2  \\
&\ge \frac{1}{\sqrt{1.7}\sqrt{0.7}}\frac{1.15}{\sqrt{0.425}+0.15}\frac{35^2}{47^2}(\delta R^0)^2 \ge 0.7290 (\delta R^0)^2.
\end{aligned}
\end{align*}
If $0.94<\delta R^0\le 1$, the function defined on $(\frac 12,1]$ by
\[
\lambda \mapsto \frac{(2\lambda-1)^{3/2}}{\sqrt{2\lambda}(1.5+\lambda)^2}\frac{2-\lambda}{\sqrt{\lambda/2}+(1-\lambda)}
\]
has positive logarithmic derivative on $(\frac 12, 1)$:
\[
    \frac{3}{2\lambda-1}-\frac{1}{2-\lambda}-\frac{1}{2\lambda}-\frac{2}{\lambda+1.5}+\frac{-\frac{1}{2\sqrt{2\lambda}}+1}{\sqrt{\lambda/2}+(1-\lambda)}>3-1-1-1+\frac{-0.5+1}{\sqrt{1/2}+0.5}=\frac{1}{1+\sqrt{2}}>0.
\]
So the above function is increasing on $(\frac 12, 1]$. Since $0.94<\delta R^0\le 1$, we have $0.85<\lambda\le 1$ and
\begin{align*}
    \frac{(2\lambda-1)^{3/2}}{\sqrt{2\lambda}}\frac{2-\lambda}{\sqrt{\lambda/2}+(1-\lambda)}&=\frac{(2\lambda-1)^{3/2}}{\sqrt{2\lambda}(1.5+\lambda)^2}\frac{2-\lambda}{\sqrt{\lambda/2}+(1-\lambda)}(1.5+\lambda)^2\\
    &\ge \frac{0.7^{3/2}}{\sqrt{1.7}\cdot 2.35^2}\frac{1.15}{\sqrt{0.425}+0.15}2.5^2(\delta R^0)^2 \ge 0.7290(\delta R^0)^2.
\end{align*}
\begin{remark}
    There is a tradeoff between $\eta$ and the possible ranges of $\frac{\mathcal{D}(\mathcal{V})}{\kappa}$, $m\kappa$, and $\frac{\mathcal{D}(\Omega^0)}{\kappa}$: the larger $\eta$ is, the larger we may take $\frac{\mathcal{D}(\Omega^0)}{\kappa}$ but the smaller we are forced to take $\frac{\mathcal{D}(\mathcal{V})}{\kappa}$ and $m\kappa$, and vice versa.
\end{remark}

%
%
%
%
\section{Asymptotic phase-locking versus finiteness of collisions}\label{app:collision}
\setcounter{equation}{0}
In this appendix, we elaborate on statement (2) of Remark \ref{R2.2}. More precisely, we define collisions and show that their finiteness is equivalent to asymptotic phase-locking in the small inertia regime $m\kappa\le \frac 14$.  Let $\Theta(t)$ be a global solution to \eqref{A-1}. We begin with the simple observation that for all $i,j\in [N]$, since $\theta_i(t)$ and $\theta_j(t)$ are analytic in $t$ with uniformly bounded derivative, the countable union of zero sets given by $\{t\ge 0: \theta_i(t)\equiv \theta_j(t)\mod 2\pi\}$ is either a discrete subset of $[0,\infty)$ or the entire interval $[0,\infty)$. The latter case necessitates that $\nu_i=\nu_j$, $\theta_i^0\equiv \theta_j^0\mod 2\pi$, and $\omega_i^0=\omega_j^0$, so the $i$th and $j$th oscillators were indistinguishable in the first place. Next, we first recall the concept of collisions in the following definition. 
\begin{definition} \label{D-1}
Let $\Theta(t)$ be a solution to \eqref{A-1}.  
\begin{enumerate}
\item
$\theta_i$ and $\theta_j$  \emph{collide} at time $t_0\ge 0$ if and only if $t_0$ is an isolated point of the set $\{t\ge 0:\theta_i(t)\equiv \theta_j(t)\mod 2\pi\}$; by the previous paragraph, this is equivalent to requiring that $(\nu_i,\theta_i^0\mod 2\pi,\omega_i^0)\neq (\nu_j,\theta_j^0\mod 2\pi,\omega_j^0)$ as elements of $\mathbb{R}\times (\mathbb{R}/2\pi \mathbb{Z})\times \mathbb{R}$, while $\theta_i(t_0)=\theta_j(t_0)$. 
\vspace{0.2cm}
\item
$\Theta(t)$ exhibits \emph{finite collision} if for every pair of oscillators $\theta_i(t)$ and $\theta_j(t)$, there are only finitely many times at which they collide.
\end{enumerate}
\end{definition}
It turns out that finiteness of collisions implies asymptotic phase-locking, and the converse holds under some additional conditions; see the following theorem for the precise statement. A weaker version of this theorem was stated and proven in \cite[Theorem 3.1]{C-D-H}.
%
%
\begin{theorem}\label{thm:finite-collision}
    Let $\Theta(t)$ be a solution to \eqref{A-1}.
    \begin{enumerate}
        \item If $\Theta(t)$ exhibits finiteness of collisions, then $\Theta(t)$ exhibits asymptotic phase-locking.
        \vspace{0.2cm}
        \item Suppose that $\Theta(t)$ exhibits asymptotic phase-locking. Then, for $i,j\in [N]$, $\theta_i$ and $\theta_j$ can collide infinitely many times only if
\begin{enumerate}
\item $\nu_i=\nu_j$,
\item $\lim_{t\to\infty}(\theta_i(t)-\theta_j(t))\in 2\pi \mathbb{Z}$,
\item $m\kappa>\frac 14$, and
\item $m\kappa \lim_{t\to\infty}\frac 1N\sum_{l=1}^N \cos(\theta_l(t)-(\theta_i(t)+\theta_j(t))/2)\ge\frac 14$.
\end{enumerate}
 In particular, if either
        \begin{enumerate}
        \item $\nu_i\neq \nu_j$ for all $i,j\in [N]$ with $i\neq j$, or
	\item $\lim_{t\to\infty}(\theta_i(t)-\theta_j(t))\notin 2\pi \mathbb{Z}$ whenever $i,j\in [N]$ with $i\neq j$ and $\nu_i=\nu_j$, or
        \item $m\kappa \le \frac 14$, or
        \item $m\kappa  \lim_{t\to\infty}\frac 1N\sum_{l=1}^N \cos(\theta_l(t)-\theta_i(t))< \frac 14$ whenever $i,j\in [N]$ with $i\neq j$, $\nu_i=\nu_j$, and $\lim_{t\to\infty}(\theta_i(t)-\theta_j(t))\in 2\pi \mathbb{Z}$,
        \end{enumerate}
        then $\Theta(t)$ must exhibit finiteness of collisions.
    \end{enumerate}
\end{theorem}
\begin{proof}
\noindent (i)~For each $i,j\in [N]$, either $(\nu_i,\theta_i^0\mod 2\pi,\omega_i^0)= (\nu_j,\theta_j^0\mod 2\pi,\omega_j^0)$, in which case $\theta_i-\theta_j$ is constant in $t$, or $(\nu_i,\theta_i^0\mod 2\pi,\omega_i^0)\neq (\nu_j,\theta_j^0\mod 2\pi,\omega_j^0)$, in which case the set $\{t\ge 0:\theta_i(t)\equiv \theta_j(t)\mod 2\pi\}$ is finite and hence upper bounded. This implies the existence of an integer $k\in \mathbb{Z}$ such that
\[
2k\pi<\theta_i(t)-\theta_j(t)<2(k+1)\pi \mathrm{~for~sufficiently~large~}t
\]
and a fortiori the boundedness of $\theta_i-\theta_j$ in $t$. Since this is true for all pairs $i,j\in [N]$, it follows that
\[
\sup_{t\ge 0}\mathcal{D}(\Theta(t))< \infty.
\]
Hence  the first statement of Remark \ref{R2.2}, $\Theta(t)$ exhibits asymptotic phase-locking.

\vspace{0.2cm}

\noindent (ii)~Suppose that $\Theta(t)$ exhibits asymptotic phase-locking. Suppose that for some $i,j\in [N]$, the oscillators $\theta_i$ and $\theta_j$ collide infinitely many times. Then the set $\{t\ge 0:\theta_i(t)\equiv \theta_j(t)\mod 2\pi\}$ is infinite discrete and hence unbounded; because $\lim_{t\to\infty} \left(\theta_i(t)-\theta_j(t)\right)$ exists by assumption, it must equal $2k\pi$ for some $k\in \mathbb{Z}$; let $k=0$ by harmlessly replacing $\theta_i$ by $\theta_i-2k\pi$. By Proposition \ref{P2.1}, we also have $\lim_{t\to\infty}\mathcal{D}(\Omega(t))=0$. By Invoking the Duhamel principle \eqref{B-4}, we have
\begin{align}\label{eq:duhamel-difference}
\begin{aligned}
\omega_i(t)-\omega_j(t) &= (\omega^0_i-\omega^0_j) e^{-t/m} + (\nu_i-\nu_j) (1 -e^{-t/m})\\
&+ \frac{\kappa}{Nm}\sum_{l=1}^N \int_0^t e^{-(t-s)/m}\left(\sin(\theta_l(s) -\theta_i(s))-\sin(\theta_l(s) -\theta_j(s))\right) ds.
\end{aligned}
\end{align}
Then we use the mean-value theorem and $|\cos(\cdot)| \leq 1$ to bound the last term of the right-hand side of \eqref{eq:duhamel-difference} by
\begin{align*}
&\left|\frac{\kappa}{Nm}\sum_{l=1}^N \int_0^t e^{-(t-s)/m}\left(\sin(\theta_l(s) -\theta_i(s))-\sin(\theta_l(s) -\theta_j(s))\right) ds\right|\\
& \hspace{0.5cm} \le \frac{\kappa}{Nm}\sum_{l=1}^N \int_0^t e^{-(t-s)/m}\left|\theta_i(s)-\theta_j(s)\right| ds. 
\end{align*}
Since $\theta_i(s)-\theta_j(s)\to 0$ as $s\to\infty$, it follows that the last term of $\eqref{eq:duhamel-difference}$ converges to 0 as $t\to\infty$. Therefore, we take the limit as $t\to\infty$ in \eqref{eq:duhamel-difference} to get 
\[
0=0+(\nu_i-\nu_j)+0.
\]
Hence, we have
\[ \nu_i=\nu_j. \]
We take a large enough time $T>0$ such that 
\[ \theta_i(t)-\theta_j(t)\in \left(-\frac \pi4,\frac \pi 4\right) \quad \mbox{for $t\ge T$.} \]
It follows from \eqref{A-1} that
\begin{equation}\label{eq:2ndODE-diff}
m(\ddot\theta_i-\ddot\theta_j)+(\dot\theta_i-\dot\theta_j)=-2\sin\left(\frac{\theta_j-\theta_i}{2}\right)\cdot \frac \kappa N\sum_{l=1}^N \cos\left(\theta_l-\frac{\theta_i+\theta_j}{2}\right).
\end{equation}
Since 
\[ \frac 1N\sum_{l=1}^N\cos(\theta_l-(\theta_i+\theta_j)/2)\le 1, \]
we have, on $[T,\infty)$,
\[
m(\ddot\theta_i-\ddot\theta_j)+(\dot\theta_i-\dot\theta_j)+\kappa (\theta_i-\theta_j)>0\quad\mathrm{whenever~}\theta_i-\theta_j>0.
\]
Similarly, on $[T,\infty)$, we have
\[
m(\ddot\theta_i-\ddot\theta_j)+(\dot\theta_i-\dot\theta_j)+\kappa (\theta_i-\theta_j)<0\quad\mathrm{whenever~}\theta_i-\theta_j<0.
\]
If it were true that $m\kappa\le \frac 14$, then the hypotheses of third statement of Lemma \ref{lem:sturm-picone} are satisfied with 
\[ a=m, \quad b=1, \quad c=\kappa, \quad I=[T,\infty), \quad y=\theta_i -\theta_j, \]
and the third statement of Lemma \ref{lem:sturm-picone} contradicts the existence of two distinct collision times after time $T$. Therefore we have
\[ m\kappa >\frac 14. \]
If it were true that
\[
m\kappa \lim_{t\to\infty}\frac 1N\sum_{l=1}^N\cos(\theta_l(t)-\theta_i(t))<\frac 14,
\]
then, since $\theta_i(t)-\theta_j(t)\to 0$ as $t\to\infty$, we would have
\[
m\kappa \lim_{t\to\infty}\frac 1N\sum_{l=1}^N\cos\left(\theta_l(t)-\frac{\theta_i(t)+\theta_j(t)}2\right)<\frac 14,
\]
and there would be a time $T'>T$ such that
\[
\frac{\kappa}{N}\sum_{l=1}^N \cos\left(\theta_l(t)-\frac{\theta_i(t)+\theta_j(t)}{2}\right)\le \frac{1}{4m}, \quad \forall~t\ge T'.
\]
Thus, it follows from \eqref{eq:2ndODE-diff} that on the interval $[T',\infty)$, we have
\[
m(\ddot\theta_i-\ddot\theta_j)+(\dot\theta_i-\dot\theta_j)+\frac{1}{4m} (\theta_i-\theta_j)>0,\quad \mathrm{whenever~}\theta_i-\theta_j>0,
\]
and
\[
m(\ddot\theta_i-\ddot\theta_j)+(\dot\theta_i-\dot\theta_j)+\frac{1}{4m} (\theta_i-\theta_j)<0,\quad \mathrm{whenever~}\theta_i-\theta_j<0,
\]
so the hypotheses of Lemma \ref{lem:sturm-picone} (3) are satisfied with $a=m$, $b=1$, $c=\frac 1{4m}$, $I=[T',\infty)$, and $y=\theta_i -\theta_j$, and Lemma \ref{lem:sturm-picone} (3) contradicts the existence of two distinct collision times after time $T'$. Therefore 
\[ m\kappa \lim_{t\to\infty}\frac 1N\sum_{l=1}^N\cos(\theta_l(t)-\theta_i(t))\ge \frac 14. \]
\end{proof}

Theorem \ref{thm:finite-collision}, in conjunction with Theorem \ref{L4.4}, can be used to prove Theorem \ref{thm:n=32nd}.
\begin{proof}[Proof of Theorem \ref{thm:n=32nd}]
Recall that we are given $N=3$ and parameters that satisfy
\[
\xi(m,\kappa, \mathcal{V},\Omega^0,\infty)=m\mathcal{D}(\mathcal{V})+2m\kappa +\frac{\mathcal{D}(\mathcal{V})}{2\kappa}
<\frac{1}{8}\sqrt{\frac{1}{6}(69 - 11\sqrt{33})},
\]
and that $\Theta$ is a solution to \eqref{A-1}. We are to show that $\Theta$ exhibits asymptotic phase-locking. By Theorem \ref{thm:finite-collision}(1), it is enough to show that $\Theta$ exhibits finiteness of collisions. Suppose for the sake of contradiction that $\Theta$ does not exhibit finiteness of collisions, i.e., a pair of oscillators, say $\theta_1$ and $\theta_2$, collide infinitely often. Since the collision times are discrete, the collision times between $\theta_1$ and $\theta_2$ are unbounded.

Now, we invoke Theorem \ref{L4.4} with $\lambda=\frac 23$ and $\ell=2\cos^{-1}\frac{1+\sqrt{33}}{8}$ (this is the argument maximum of $f_\lambda$). We then have that
\[
2\xi(m,\kappa,\mathcal{V},\Omega^0,\infty)<\frac{1}{4}\sqrt{\frac{1}{6}(69 - 11\sqrt{33})}=f_\lambda(\ell).
\]
By continuity, we may choose $\eta>0$ large enough so that
\[
2\xi(m,\kappa,\mathcal{V},\Omega^0,\eta)<f_\lambda(\ell).
\]
Since the collision times of $\theta_1$ and $\theta_2$ are unbounded, there is a collision time $t_1\ge \eta m$. Take $\mathcal{A}=\{1,2\}$. Then, up to modulo $2\pi$ translations, the hypotheses of Theorem \ref{L4.4} are satisfied, since $\mathcal{D}(\{\theta_1(t_1),\theta_2(t_1)\})=0<\ell$. The extra hypothesis \eqref{eq:xi-partial-B} is also satisfied with $\mathcal{B}=\{1,2,3\}$. Therefore by statement (3) of Theorem \ref{L4.4}, $\Theta$ exhibits asymptotic phase-locking. But, since $m\kappa<\frac 14$, it follows from Theorem \ref{thm:finite-collision}(2) that $\Theta$ must exhibit finiteness of collisions, a contradiction. This contradiction completes the proof.
\end{proof}

\vspace{.2cm}
\noindent {\bf Data availability} The data are available from the corresponding author upon reasonable request.

\vspace{.2cm}
\noindent {\bf Conflict of interest} The authors declare that they have no conflict of interest.

%
%
%
%

\end{document}